\documentclass[11pt]{article}
\usepackage[a4paper,margin=26mm]{geometry}
\usepackage{amsmath,amssymb,amsthm,mathtools,mathrsfs}
\usepackage{enumitem}
\usepackage{booktabs}
\usepackage{microtype}
\usepackage{tikz}
\usetikzlibrary{arrows.meta,positioning,calc,decorations.pathreplacing}
\usepackage[hidelinks]{hyperref}
\hypersetup{
  pdftitle={Detecting Einstein metrics on products of spheres},
  pdfauthor={Anna Siffert},
  pdfsubject={Cohomogeneity-one positive Einstein metrics},
  pdfkeywords={Einstein metrics, cohomogeneity one, singular perturbation, matching, response}
}

\newtheorem{theorem}{Theorem}[section]
\newtheorem{lemma}[theorem]{Lemma}
\newtheorem{proposition}[theorem]{Proposition}
\newtheorem{corollary}[theorem]{Corollary}
\newtheorem{remark}[theorem]{Remark}

\newcommand{\Ric}{\operatorname{Ric}}
\newcommand{\Rm}{\operatorname{Rm}}

\newcommand{\eps}{\varepsilon}

\title{Detecting Einstein metrics on products of spheres}

\author{
Anna Siffert
}
\date{}
\begin{document}
\maketitle

\begin{abstract}
On
\[
S^5\times S^2,\qquad S^5\times S^3,\qquad
S^5\times S^4,\qquad S^7\times S^2,
\]
we construct infinitely many pairwise non-isometric real-analytic Einstein
metrics.  After discarding finitely many members, none is isometric or
homothetic to a Riemannian product or to a classical two-summand B\"ohm
metric.  The metrics have the form
\[
dt^2+a(t)^2g_{S^p}+b(t)^2g_{S^p}+c(t)^2g_{S^q},
\qquad p,q\ge2,\qquad 2p+q\le8.
\]

The construction starts with solutions for which the two non-collapsing
radii at one singular orbit are very small.  After magnifying this region,
the Einstein equations approach a Ricci-flat problem whose solutions tend
to a cone by a decaying oscillation.  The main point is to determine
exactly which part of this oscillation can still be seen by the closing
conditions at the other end of the manifold.  Rather than following the
full, increasingly oscillatory shooting map, we extract two response
quantities from the Ricci-flat limit.  One determines the common
oscillatory phase; after that phase has been removed, the second closing
condition is reduced to a scalar sign change.  A uniform inner--outer
matching argument transfers this limiting sign change to the exact
Einstein equations.  As the initial scale tends to zero, the phase makes
arbitrarily many turns, and the closing conditions are therefore met
infinitely many times.
\end{abstract}

\section{Introduction}

\subsection{Geometric setting and main result}
\label{sec:main-results}

Cohomogeneity-one Einstein equations constitute a
singular system of ordinary differential equations; standard references
include \cite{Besse,Bohm1998,DancerWang2000,EschenburgWang2000}.
We study metrics
\begin{equation}\label{eq:intro-ansatz}
g
=
dt^2+a(t)^2g_{S^p}+b(t)^2g_{S^p}+c(t)^2g_{S^q},
\end{equation}
where the sphere metrics have sectional curvature \(1\).
The orbit space is a closed interval \([0,L]\), with \(L>0\) determined
by the solution, and \(t\) is arclength along a geodesic orthogonal to the
orbits.  For \(0<t<L\), the principal orbit is
\(S^p\times S^p\times S^q\).  At the left endpoint \(t=0\) of the
interval \([0,L]\), the first \(S^p\)-factor collapses, while the other
\(S^p\)-factor and the \(S^q\)-factor remain positive.  At the right
endpoint \(t=L\), the second \(S^p\)-factor collapses, while the first
\(S^p\)-factor and the \(S^q\)-factor remain positive.
The precise smoothness conditions and group diagram are given in
Section~\ref{sec:singular-orbit-germs}.  We normalize the scale by
\[
\Ric(g)=(2p+q)g.
\]

Our main result is the following.

\begin{quote}
\textbf{Theorem A} (Existence and distinctness).
Let \(p,q\ge2\) and \(D=2p+q\le 8\).  Then
\(S^{2p+1}\times S^q\) carries infinitely many pairwise non-isometric
real-analytic cohomogeneity-one Einstein metrics of the
form~\eqref{eq:intro-ansatz}.  After discarding finitely many members,
none is isometric or homothetic to a Riemannian product or to a classical
two-summand B\"ohm metric.  Among integer pairs \(p,q\ge2\) with \(2p+q\le8\), these are precisely
\[
S^5\times S^2,\qquad S^5\times S^3,\qquad
S^5\times S^4,\qquad S^7\times S^2.
\]
\end{quote}

B\"ohm's classical work already gives infinitely many cohomogeneity-one
Einstein metrics on these underlying products of spheres
\cite{Bohm1998}.  Recent existence results illustrate several different
approaches to the cohomogeneity-one Einstein boundary-value problem,
including Chi's existence criteria for two-summand double disk bundles,
the phase-space analysis of Nienhaus and Wink on the ten-sphere, and
Wang's computer-assisted construction for doubly warped products
\cite{Chi2024Existence,NienhausWink2025,Wang2026}.  The metrics constructed
here have three genuinely varying radius functions.  The existence statement is proved first in the representative
case \((p,q)=(2,3)\), and Section~\ref{sec:uniform-extension} proves the
remaining dimensions together with the intrinsic distinctness and
comparison statements in Theorem~A.

A second feature of the argument is that the final closing step admits a
finite-dimensional formulation independent of the detailed Einstein ODE.
The Ricci-flat inner problem produces two complex responses
\(Q(\rho)\) and \(X(\rho)\), and the outer linearized Einstein problem
transports them to the matching section by fixed nonzero complex
coefficients \(L_0\) and \(L_1\).  For the abstract closing argument we
write
\[
A(\rho):=L_0Q(\rho),\qquad
B(\rho):=L_1X(\rho).
\]
The small scale \(\eps\) contributes a common logarithmic rotation to both
channels.  Writing its phase modulo \(\pi\) as \(\theta\), the two leading
matching errors have the form
\[
\Re(e^{i\theta}A(\rho)),\qquad
\Re(e^{i\theta}B(\rho)).
\]
Thus \(A\) and \(B\) are the two transported complex responses, while
\(\theta\) is their common oscillatory phase.  The first channel fixes this
phase because \(A\neq0\).  On that phase branch, the second channel is
governed by the phase-invariant scalar
\[
\Phi(\rho)=\Im\!\left(A(\rho)\overline{B(\rho)}\right).
\]
Proposition~\ref{prop:response-crossing} shows abstractly that a strict
endpoint sign change of \(\Phi\) forces a zero of the full two-component
matching map for every sufficiently large phase label.  No simple-zero
condition on \(\Phi\) is needed: the required nondegeneracy occurs in the
phase direction.  Corollary~\ref{cor:response-crossing-stability} records
that this mechanism is stable under small \(C^1\) perturbations.  Thus the
proof separates the geometric task of constructing and transporting the
response from the topological task of forcing an exact crossing.

The proof is written in full detail first for \((p,q)=(2,3)\), where the
principal orbit is \(S^2\times S^2\times S^3\).  The dimension-dependent
estimates are then checked uniformly in
Section~\ref{sec:uniform-extension}.

\subsection{The detection--response mechanism}
\label{sec:nutshell}

The proof is modeled on the oscillatory mechanism behind B\"ohm's
construction, but the three-factor problem has two independent matching
conditions.  The useful object is therefore not a scalar shooting
function but a two-channel response.

Let
\[
\eps=b(0),\qquad
\rho=\frac{c(0)}{b(0)}.
\]
For \(\eps\ll1\), magnifying the left singular orbit produces a complete
Ricci-flat inner problem.  Its two anisotropy variables may be taken as
\[
\zeta=\log\frac ab,
\qquad
\xi=\log\frac{c}{\sqrt2\,(ab)^{1/2}}
\]
in the representative case \((p,q)=(2,3)\).  Near the asymptotic
Ricci-flat cone they have oscillatory expansions
\[
\zeta(\tau)
=
\tau^{-3}\Re\!\left(Q(\rho)\tau^{i\sqrt3}\right)+o(\tau^{-3}),
\]
\[
\xi(\tau)
=
\tau^{-3}\Re\!\left(X(\rho)\tau^{i\sqrt3}\right)+o(\tau^{-3}).
\]
The complex numbers \(Q(\rho)\) and \(X(\rho)\) are the two inner
responses.

The outer positive-Einstein problem transports these amplitudes to the
equatorial region.  If \(L_0,L_1\) denote the two relevant values of the
normalized outer Jacobi solution, the normalized exact matching map has
the limit
\[
G(\theta,\rho)
=
\begin{pmatrix}
\Re(e^{i\theta}L_0Q(\rho))\\
\Re(e^{i\theta}L_1X(\rho))
\end{pmatrix}.
\]
The first equation determines the phase because \(L_0Q(\rho)\ne0\).
After eliminating that phase, the second equation has the sign of
\[
\Phi(\rho)
=
\Im\!\left(
L_0Q(\rho)\,
\overline{L_1X(\rho)}
\right).
\]
Thus the infinite-dimensional shooting problem is reduced, at leading
order, to a scalar sign change.

Two geometric limits determine that sign.  The small-ratio heteroclinic
gives
\[
\frac{X(\rho)}{Q(\rho)}\longrightarrow\frac12
\qquad(\rho\downarrow0),
\]
whereas the balanced invariant family gives
\[
X(\sqrt2)=-\frac12Q(\sqrt2).
\]
The outer Wronskian has fixed orientation
\[
\Im(L_0\overline{L_1})=-\sqrt3.
\]
Consequently
\[
\Phi(\rho)<0
\quad\text{for small }\rho>0,
\qquad
\Phi(\sqrt2)>0.
\]
The abstract response-crossing criterion in
Section~\ref{sec:abstract-response-crossing} converts this sign change,
together with the weighted \(C^1\) inner--outer matching theorem, into an
exact zero of the two matching conditions.

This is the conceptual point of the argument.  We do not attempt to
control the entire shooting map globally.  We detect the two asymptotic
responses, use one to remove the common oscillatory phase, and use the
relative orientation of the second to force a crossing.

\subsection{Why the small scale produces infinitely many metrics}

The cone exponents are
\[
-\alpha_D\pm i\omega_D,
\qquad
\alpha_D=\frac{D-1}{2},
\qquad
\omega_D=\frac12\sqrt{(D-1)(9-D)}.
\]
For \(D<9\), the imaginary part is nonzero.  The logarithmic phase
therefore rotates without bound as \(\eps\downarrow0\):
\[
\theta
\equiv
\omega_D\log(1/\eps)
\pmod{\pi}.
\]
Once the response crossing is known, every sufficiently large phase label
produces an exact reflection point.  A sparse arithmetic subsequence of
phase labels gives strictly decreasing scales.  Curvature at the singular
orbit grows like \(\eps^{-2}\), which later separates infinitely many
isometry classes.

This viewpoint also explains the dimensional restriction in Theorem~A:
the construction uses the oscillatory cone regime \(D<9\).  Under
\(p,q\ge2\) and \(D=2p+q\), the possibilities are exactly
\[
(p,q)=(2,2),(2,3),(2,4),(3,2).
\]

\subsection{Why detection rather than classical shooting?}
\label{sec:why-not-shooting}

There is a natural classical shooting formulation.  Starting with the
parameters \((\eps,\rho)\) at the left singular orbit, solve the
positive-Einstein equations until the first \(H=0\) section and consider
the two closing errors
\[
\mathcal S(\eps,\rho)
=
\left.(a-b,c')\right|_{H=0}.
\]
A zero of \(\mathcal S\) gives the reflection conditions and hence a
compact Einstein metric.  Thus the issue is not that shooting is the
wrong formulation.  The issue is that, in the regime relevant to the
construction, the raw shooting map hides the mechanism producing its
zeros.

Indeed, when \(\eps\) is small, the solution spends a long rescaled time
near a Ricci-flat conical model.  Its two independent anisotropies decay
while oscillating with the same logarithmic frequency.  Decreasing
\(\eps\) therefore makes the vector \(\mathcal S(\eps,\rho)\) wind
repeatedly in its two-dimensional target.  A direct shooting proof would
have to control this increasingly rapid rotation while simultaneously
showing that the winding curve passes through the origin.

The detection and response viewpoint developed in
\cite{DetectionPrinciple,ResponseGeometry} reorganizes exactly the same boundary-value
problem.  The Ricci-flat inner analysis extracts two complex responses
\[
Q(\rho),\qquad X(\rho),
\]
which retain the leading amplitude and phase information that survives
into the closing region.  Changing \(\eps\) rotates both responses by the
same phase.  We use the first matching condition to choose this
\emph{common phase}.  Once it has been removed, the information relevant
to the second condition is the relative orientation of the two transported
responses
\[
A(\rho)=L_0Q(\rho),\qquad B(\rho)=L_1X(\rho),
\]
measured by
\[
\Phi(\rho)
=
\Im\!\left(A(\rho)\overline{B(\rho)}\right).
\]
A closely related realization of the same detection--response architecture
appears in the \(SU(3)\)-invariant \(S^7\) problem
\cite{SiffertS7Quantized}.  In both settings one starts from a
two-parameter singular-orbit family, stops at a canonical \(H=0\) section,
passes to a Ricci-flat small-scale limit, and uses an oscillatory cone mode
to turn the closing problem into a finite-dimensional response crossing.
The analytic reductions are nevertheless different.  In the \(S^7\)
problem the Ricci-flat dynamics select a distinguished threshold trajectory,
and the limiting detector has one oscillatory channel that fixes the phase
and a second reduced response whose sign changes.  In the present
product-of-spheres problem the Ricci-flat limit is instead a family indexed
by \(\rho\), both leading matching channels are complex, and after removal
of their common phase the decisive quantity is their relative orientation,
\[
\Phi(\rho)=\Im\!\left(A(\rho)\overline{B(\rho)}\right).
\]
Thus neither result is a specialization of the other: they are two
different geometric realizations of the same finite-dimensional
detection--response principle.

The original two-dimensional shooting problem is therefore reduced, at
leading order, to
\[
\begin{array}{c}
\eps\\[-1mm]
\downarrow\\[-1mm]
\text{choose the common phase}
\end{array}
\qquad+\qquad
\begin{array}{c}
\rho\\[-1mm]
\downarrow\\[-1mm]
\text{detect a sign change of }\Phi.
\end{array}
\]

This is the gain from detection.  We do not try to follow the entire
shooting map as it winds ever more rapidly.  We first identify the much
smaller amount of asymptotic information that the closing conditions can
see, remove the universal oscillation, and then force a zero through a
one-dimensional sign change.  The weighted \(C^1\) matching theorem is
what makes this reduction rigorous: it proves that the responses extracted
from the Ricci-flat model actually control the exact positive-Einstein
matching errors.

The proof is organized accordingly.  We first construct the singular
Einstein family and its Ricci-flat blow-up, then extract and orient the two
responses, transport them through the outer region, and finally compare
the response model with the exact \(H=0\) matching map.  The case
\((p,q)=(2,3)\) is treated in full detail because it displays the
heteroclinic and matching geometry without dimension indices.  We then
prove the uniform extension to all four pairs in Theorem~A, and close the
paper by isolating the abstract response-crossing principle used by the
construction.

\subsection*{Acknowledgements}

The author used OpenAI's ChatGPT as an interactive assistant during the
preparation of this manuscript for exposition, organization, creating the pictures,
proofreading, and suggestions for presenting and checking arguments.
All these contributions were independently verified by the author,
who takes full responsibility for the mathematical content and
conclusions of the paper.

\section{Singular-orbit initial-value theory}\label{sec:singular-orbit-germs}
This section answers a basic geometric question: what data may be prescribed
at the left endpoint, where one orbit direction collapses?  We first identify
the two singular orbits from the cohomogeneity-one group diagram.  We then
translate smooth extension across the left singular orbit into parity
conditions for \(a,b,c\), and finally apply the singular initial-value theory
of Eschenburg and Wang~\cite{EschenburgWang2000}.  The outcome needed later is
simple: the local Einstein metrics form a two-parameter family, parametrized
by the two positive radii \(b_0,c_0\).  We subsequently write
\[
b_0=\eps,\qquad c_0=\eps\rho,
\]
so that \(\eps\) measures the common small scale and \(\rho\) measures the
shape, namely the ratio of the two non-collapsing radii.

\subsection{The group diagram}
To keep the group notation tied directly to the metric ansatz, we use
lower-case subscripts \(a,b,c\) for the three spherical directions.
Thus a subscript \(a\), \(b\), or \(c\) indicates the factor whose radius
is described by \(a(t)\), \(b(t)\), or \(c(t)\), respectively.

\smallskip

Consider the compact Lie group
\[
G=\mathrm{SO}(3)_a\times\mathrm{SO}(3)_b\times\mathrm{SO}(4)_c
\]
and let
\[
K=\mathrm{SO}(2)_a\times\mathrm{SO}(2)_b\times\mathrm{SO}(3)_c\subset G,
\]
where \(\mathrm{SO}(3)_c\subset\mathrm{SO}(4)_c\) is the standard isotropy subgroup of the
round three-sphere.  Then the principal orbit is 
\[
G/K\cong S^2_a\times S^2_b\times S^3_c.
\]
The two
singular isotropy groups are
\[
H_-=\mathrm{SO}(3)_a\times\mathrm{SO}(2)_b\times\mathrm{SO}(3)_c,
\qquad
H_+=\mathrm{SO}(2)_a\times\mathrm{SO}(3)_b\times\mathrm{SO}(3)_c.
\]
Thus the complete group diagram is
\begin{equation}\label{eq:group-diagram}
H_-\ \supset\ K\ \subset\ H_+\qquad\subset\qquad G=\mathrm{SO}(3)_a\times\mathrm{SO}(3)_b\times\mathrm{SO}(4)_c.
\end{equation}
The quotients
\[
H_-/K\cong S^2_a,
\qquad
H_+/K\cong S^2_b
\]
are the normal spheres to the two singular orbits.  Consequently
\[
G/H_-\cong S^2_b\times S^3_c,
\qquad
G/H_+\cong S^2_a\times S^3_c.
\]

Because the action has cohomogeneity one and the metric is compact,
the orbit space is a closed interval.  We therefore choose a radial
coordinate
\[
        t\in[0,L],
\]
where \(t=0\) and \(t=L\) correspond to the two singular orbits, and
\(0<t<L\) parametrizes the principal orbits.  Thus the metric on the
open principal-orbit region has the form
\[
        g=dt^2+g_t,\qquad 0<t<L,
\]
with \(g_t\) a \(G\)-invariant metric on the principal orbit.

In the present ansatz, the metric \(g_t\) is described by three radius
functions \(a(t),b(t),c(t)\).  The endpoint values at \(t=0\) are
therefore not arbitrary initial values at an ordinary point of an ODE:
they encode how the metric extends smoothly across the left singular
orbit.  In particular, the factor corresponding to \(a(t)\) collapses
there, while the \(b\)- and \(c\)-factors remain nonzero.  This is the
singular-orbit initial-value problem to which we now specialize the
theory of Eschenburg and Wang~\cite{EschenburgWang2000}.

Thus, in the metric notation, the left singular orbit is the endpoint
where \(a=0\), while the right singular orbit is the endpoint where
\(b=0\).  The third factor never collapses in the group diagram.  In
particular,
\[
a(0)=0,\qquad b(0)>0,\qquad c(0)>0.
\]

\smallskip

The corresponding tubular neighborhoods are
\[
G\times_{H_-}D^3\cong D^3\times S^2\times S^3,
\qquad
G\times_{H_+}D^3\cong S^2\times D^3\times S^3.
\]
Gluing them along their common boundary
\(G/K\cong S^2\times S^2\times S^3\) gives
\[
\bigl(D^3\times S^2\bigr)\cup_{S^2\times S^2}
\bigl(S^2\times D^3\bigr)\;\times S^3
\cong
(S^2*S^2)\times S^3
\cong S^5\times S^3.
\]

\subsection{Singular-orbit smoothness and the local Einstein family}
The smoothness condition used here is the standard cohomogeneity-one singular-orbit smoothness condition for the specific setting studied here; compare the general singular initial-value framework of Eschenburg--Wang~\cite{EschenburgWang2000}.  We record the elementary specialization explicitly because it will be used throughout the paper.

\begin{lemma}[Smooth polar-product criterion]
\label{lem:smooth-polar-product}
Let \(b_0,c_0>0\).  The metric
\begin{equation*}
g
=
dt^2+a(t)^2g_{S^2}+b(t)^2g_{S^2}+c(t)^2g_{S^3},
\end{equation*}
extends smoothly across the
left singular orbit (with collapsing first \(S^2\)-factor) if and only if,
after choosing the geodesic distance \(t\ge0\) from the singular orbit,
there exist smooth functions \(A,B,C\) of \(r=t^2\) such that
\[
a(t)=tA(t^2),\qquad b(t)=B(t^2),\qquad c(t)=C(t^2),
\]
with
\[
A(0)=1,\qquad B(0)=b_0,\qquad C(0)=c_0.
\]
Equivalently,
\[
a(-t)=-a(t),\qquad b(-t)=b(t),\qquad c(-t)=c(t),
\]
and
\[
a'(0)=1,\qquad b'(0)=c'(0)=0.
\]
\end{lemma}

\begin{proof}
The collapsing part is the standard polar metric on a normal \(3\)-disk:
\[
dt^2+a(t)^2g_{S^2}.
\]
Smoothness at the center of \(D^3\) is equivalent to
\[
a(t)=t+O(t^3)
\]
with \(a\) odd.  The noncollapsing coefficients multiply fixed metrics on
the singular orbit and therefore must be smooth even functions of the normal
coordinate.  This gives the stated form.  Conversely, writing
\[
a(t)=tA(t^2),\qquad A(0)=1
\]
shows that the collapsing block is a smooth \(O(3)\)-invariant metric on
\(D^3\), while \(B(t^2)\) and \(C(t^2)\) are smooth invariant coefficient
functions on the noncollapsing factors.
\end{proof}

We now specialize the Eschenburg--Wang singular-orbit initial-value theory
to the present group diagram and three-factor ansatz.  For the metric
\[
g=dt^2+a(t)^2g_{S^2}+b(t)^2g_{S^2}+c(t)^2g_{S^3}
\]
with the normalization \(\Ric(g)=7g\), the three tangential Einstein
equations are
\begin{subequations}\label{eq:PE}
\begin{align}
a''
&=
\frac{1-(a')^2}{a}
-2\frac{a'b'}{b}
-3\frac{a'c'}{c}
-7a,
\label{eq:PE-a}
\\
b''
&=
\frac{1-(b')^2}{b}
-2\frac{a'b'}{a}
-3\frac{b'c'}{c}
-7b,
\label{eq:PE-b}
\\
c''
&=
\frac{2(1-(c')^2)}{c}
-2\frac{a'c'}{a}
-2\frac{b'c'}{b}
-7c.
\label{eq:PE-c}
\end{align}
\end{subequations}
\begin{theorem}[Regular-singular Einstein singular-orbit theorem]
\label{thm:singular-orbit-germ}
For every
\[
(b_0,c_0)\in(0,\infty)^2
\]
there is a unique real-analytic Einstein germ within the three-factor ansatz,
\[
(a,b,c),
\]
satisfying \(\Ric=7g\) and the left-singular-orbit smoothness conditions of
Lemma~\ref{lem:smooth-polar-product}.  The germ depends real-analytically on
\((b_0,c_0)\).

More precisely,
\begin{align}
a(t)
&=
t+a_3t^3+O(t^5),
\\
b(t)
&=
b_0+b_2t^2+O(t^4),
\\
c(t)
&=
c_0+c_2t^2+O(t^4),
\end{align}
where
\begin{align}
b_2
&=
\frac{1-7b_0^2}{6b_0},
\label{eq:b2-singular-orbit}
\\
c_2
&=
\frac{2-7c_0^2}{6c_0},
\label{eq:c2-singular-orbit}
\\
a_3
&=
-\frac1{12}
\left(
4\frac{b_2}{b_0}
+
6\frac{c_2}{c_0}
+
7
\right)
\nonumber\\
&=
\frac7{18}
-\frac1{18b_0^2}
-\frac1{6c_0^2}.
\label{eq:a3-singular-orbit}
\end{align}
All higher odd coefficients of \(a\) and even coefficients of \(b,c\) are
determined recursively and analytically by \((b_0,c_0)\).
\end{theorem}

\begin{proof}
Insert the parity ansatz
\[
a(t)=tA(t^2),\qquad b(t)=B(t^2),\qquad c(t)=C(t^2)
\]
into the Einstein equations \eqref{eq:PE-a}--\eqref{eq:PE-c}.  With
\(r=t^2\),
\[
a'=A+2rA_r,\qquad
a''=t(6A_r+4rA_{rr}),
\]
\[
b'=2tB_r,\qquad
b''=2B_r+4rB_{rr},
\]
and similarly for \(c\).
The apparent singularity in \eqref{eq:PE-a} is regular because
\[
A(0)=1
\]
forces
\[
1-(A+2rA_r)^2=O(r).
\]
Thus the transformed system is a real-analytic regular-singular system in
the variable \(r\), with free data
\[
B(0)=b_0,\qquad C(0)=c_0
\]
and compatibility condition \(A(0)=1\).

To see the first compatibility relations directly, write
\[
A(r)=1+a_3r+O(r^2),\quad
B(r)=b_0+b_2r+O(r^2),\quad
C(r)=c_0+c_2r+O(r^2).
\]
The constant terms of \eqref{eq:PE-b} and \eqref{eq:PE-c} give
\[
6b_2=\frac1{b_0}-7b_0,
\qquad
6c_2=\frac2{c_0}-7c_0,
\]
which are \eqref{eq:b2-singular-orbit}--\eqref{eq:c2-singular-orbit}.  The coefficient of \(t\)
in \eqref{eq:PE-a} gives
\[
12a_3+4\frac{b_2}{b_0}
+6\frac{c_2}{c_0}+7=0,
\]
hence \eqref{eq:a3-singular-orbit}.

It remains to check that the tangential equations produce a genuine
Einstein germ rather than merely a solution of the tangential subsystem.
The radial Einstein equation is
\[
-2\frac{a''}{a}
-2\frac{b''}{b}
-3\frac{c''}{c}
=7.
\]
At the singular orbit its finite limiting value is exactly
\[
-12a_3
-4\frac{b_2}{b_0}
-6\frac{c_2}{c_0}
=7,
\]
which is equivalent to the compatibility relation just obtained from
\eqref{eq:PE-a}.  Hence the radial constraint has the required limiting value at the
singular orbit.  On the regular region \(t>0\), the contracted Bianchi identity gives a
homogeneous first-order propagation equation for the radial Einstein
constraint.  The constraint has limiting value zero at \(t=0\), as the
preceding calculation shows; uniqueness for this propagation equation
therefore gives zero constraint throughout the germ.  Thus the resulting solution of the tangential equations
satisfies the full Einstein system.

The higher-order recursion can be checked explicitly.  Write
\[
A(r)=1+\sum_{j\ge1}A_jr^j,\qquad
B(r)=b_0+\sum_{j\ge1}B_jr^j,\qquad
C(r)=c_0+\sum_{j\ge1}C_jr^j.
\]
After the coefficients through order \(r^n\) has been determined, the
coefficient equations at the next order are linear in
\((A_{n+1},B_{n+1},C_{n+1})\).  Their coefficient matrix is
\[
\begin{pmatrix}
2(n+2)(2n+3) &
\dfrac{4(n+1)}{b_0} &
\dfrac{6(n+1)}{c_0}
\\[1.2ex]
0 & 2(n+1)(2n+3) & 0
\\
0 & 0 & 2(n+1)(2n+3)
\end{pmatrix},
\qquad n\ge0.
\]
Its determinant is
\[
8(n+2)(n+1)^2(2n+3)^3>0.
\]
Thus no additional free Taylor coefficient can occur at any order: the
formal series is uniquely determined by \((b_0,c_0)\).

After setting \(A=1+rU\), the parity-reduced equations form an analytic
regular-singular system in \(r\) for \(U,B,C\) and their first derivatives.
The singular-orbit initial-value theory of
Eschenburg--Wang~\cite{EschenburgWang2000}, specialized to the present
three-factor ansatz, promotes the uniquely determined formal series to a
convergent real-analytic Einstein germ and gives real-analytic dependence on
\((b_0,c_0)\) on compact subsets of \((0,\infty)^2\).  The calculation
above identifies explicitly the two free parameters and the first Taylor
coefficients in this specialization.
\end{proof}

\begin{corollary}[Parameter coordinates]
\label{cor:singular-orbit-parameters}
Within the fixed labelled three-factor ansatz, the real-analytic left-singular-orbit
Einstein germs are parametrised by
\[
(b_0,c_0)\in(0,\infty)^2,
\]
or equivalently by
\[
x=\log b_0,\qquad
y=\log\frac{c_0}{\sqrt2\,b_0}.
\]
\end{corollary}

\begin{proof}
Theorem~\ref{thm:singular-orbit-germ} gives existence and uniqueness in the
real-analytic class for each positive pair \((b_0,c_0)\).  The logarithmic
change of coordinates is a smooth diffeomorphism from
\((0,\infty)^2\) to \(\mathbb R^2\).
\end{proof}

\begin{remark}[Small-singular-orbit rescaling and the Ricci-flat inner family]
\label{rem:inner-rescaling}
Write the two non-collapsing radii at the left singular orbit as
\[
b_0=\eps,\qquad c_0=\eps\rho,
\]
where \(\eps>0\) is the radius of the non-collapsing \(S^2\)-factor at the left singular orbit and \(\rho>0\) is the ratio of the two positive radii there.  To examine
the geometry on the scale of this small singular orbit, introduce
\[
t=\eps\tau,\qquad
A_\eps(\tau)=\eps^{-1}a(\eps\tau),\qquad
B_\eps(\tau)=\eps^{-1}b(\eps\tau),\qquad
C_\eps(\tau)=\eps^{-1}c(\eps\tau).
\]
Equivalently, after pulling the metric back by \(t=\eps\tau\) and multiplying
it by \(\eps^{-2}\), we obtain the inner metric
\[
g_{\mathrm{inner},\eps}
=
d\tau^2+A_\eps(\tau)^2g_{S^2}
+B_\eps(\tau)^2g_{S^2}
+C_\eps(\tau)^2g_{S^3}.
\]
Its singular-orbit data are
\[
A_\eps(0)=0,\qquad A_\eps'(0)=1,\qquad
B_\eps(0)=1,\qquad C_\eps(0)=\rho.
\]
Since the original metric satisfies \(\Ric(g)=7g\), the rescaled metric
satisfies
\[
\Ric(g_{\mathrm{inner},\eps})
=
7\eps^2 g_{\mathrm{inner},\eps}.
\]
Thus the Einstein constant disappears in the limit \(\eps\to0\).
The limiting inner metric
\[
g_{\mathrm{inner},\rho}
=
d\tau^2+A_\rho(\tau)^2g_{S^2}
+B_\rho(\tau)^2g_{S^2}
+C_\rho(\tau)^2g_{S^3}
\]
therefore satisfies
\[
\Ric(g_{\mathrm{inner},\rho})=0.
\]
Equivalently, its radius functions satisfy the Ricci-flat system
\begin{align*}
A_\rho''
&=
\frac{1-(A_\rho')^2}{A_\rho}
-2\frac{A_\rho'B_\rho'}{B_\rho}
-3\frac{A_\rho'C_\rho'}{C_\rho},\\
B_\rho''
&=
\frac{1-(B_\rho')^2}{B_\rho}
-2\frac{A_\rho'B_\rho'}{A_\rho}
-3\frac{B_\rho'C_\rho'}{C_\rho},\\
C_\rho''
&=
\frac{2(1-(C_\rho')^2)}{C_\rho}
-2\frac{A_\rho'C_\rho'}{A_\rho}
-2\frac{B_\rho'C_\rho'}{B_\rho}.
\end{align*}
Its singular-orbit data are
\[
A_\rho(0)=0,\qquad A_\rho'(0)=1,\qquad
B_\rho(0)=1,\qquad C_\rho(0)=\rho.
\]

Smoothness at the left singular orbit requires \(A_\rho\) to be odd in
\(\tau\) and \(B_\rho,C_\rho\) to be even.  We therefore write
\[
A_\rho(\tau)=\tau+a_3\tau^3+O(\tau^5),\qquad
B_\rho(\tau)=1+b_2\tau^2+O(\tau^4),\qquad
C_\rho(\tau)=\rho+c_2\tau^2+O(\tau^4).
\]
Substitution into the Ricci-flat equations and comparison of the
lowest-order coefficients gives
\[
b_2=\frac16,\qquad
c_2=\frac1{3\rho},\qquad
a_3=-\frac1{18}-\frac1{6\rho^2}.
\]
Hence the one-parameter Ricci-flat inner family has the expansion
\[
\begin{aligned}
A_\rho(\tau)
&=
\tau-\left(\frac1{18}+\frac1{6\rho^2}\right)\tau^3
+O(\tau^5),\\
B_\rho(\tau)
&=
1+\frac16\tau^2+O(\tau^4),\\
C_\rho(\tau)
&=
\rho+\frac1{3\rho}\tau^2+O(\tau^4).
\end{aligned}
\]
This is the one-parameter Ricci-flat inner family used in the response
analysis.
\end{remark}

\section{Analytic setting}\label{sec:analytic-setting}

This short section prepares two ingredients used throughout the proof.
First, we record the mean-curvature identities that define the transverse
Einstein matching section \(H=0\).  Second, we magnify the small left singular orbit
and show directly that the positive-Einstein equations differ from the
Ricci-flat inner equations only by \(O(\eps^2)\).  This rescaling identifies
the variables and scale used in the inner--outer comparison.

\subsection{Mean curvature and the Einstein matching section}

For general background on Einstein metrics we refer to \cite{Besse}.  For a multiply warped product
\[
g=dt^2+\sum_{j=1}^3 f_j(t)^2g_j,
\]
where \(\Ric(g_j)=(d_j-1)g_j\), the standard warped-product Ricci
formulas give
\[
\Ric(\partial_t,\partial_t)
=-\sum_j d_j\frac{f_j''}{f_j},
\]
and the tangential equation on the \(j\)-th factor is
\[
f_j''
=\frac{d_j-1}{f_j}
-\frac{d_j-1}{f_j}(f_j')^2
-f_j'\sum_{\ell\ne j}d_\ell\frac{f_\ell'}{f_\ell}
-\lambda f_j
\]
when \(\Ric(g)=\lambda g\).  For the present three-factor ansatz, with factor dimensions
\(2,2,3\) and \(\lambda=7\), this gives precisely the three equations
\eqref{eq:PE-a}--\eqref{eq:PE-c} stated in
Section~\ref{sec:singular-orbit-germs}.
Define the positive-Einstein mean curvature and squared shape norm by
\[
H=2\frac{a'}a+2\frac{b'}b+3\frac{c'}c,
\qquad
|L|^2=2\left(\frac{a'}a\right)^2
+2\left(\frac{b'}b\right)^2
+3\left(\frac{c'}c\right)^2.
\]
A basic monotonicity observation going back to
B\"ohm~\cite{Bohm1998} follows from the radial Einstein equation:
\begin{equation}\label{eq:PE-Hprime}
H'=-|L|^2-7<0.
\end{equation}
The Gauss constraint is
\begin{equation}\label{eq:PE-Hamiltonian}
H^2
=|L|^2+\frac2{a^2}+\frac2{b^2}+\frac6{c^2}-42.
\end{equation}
Here \(42=(7-1)\cdot7\) is the Einstein-constant term in the
Hamiltonian constraint for a seven-dimensional principal orbit and the
normalization \(\Ric(g)=7g\).  Geometrically, \(H\) is the mean curvature
of a principal orbit with respect to increasing \(t\).  Since
\[
H'=-|L|^2-7<0,
\]
a trajectory can cross \(H=0\) at most once, and whenever it does the
crossing is transverse.  This is why \(H=0\) is a convenient matching
hypersurface later in the proof.

\subsection{Magnifying the small singular orbit}

Remark~\ref{rem:inner-rescaling} introduced the inner rescaling geometrically.  We now apply it
to the full positive-Einstein equations in order to identify precisely the
error relative to the Ricci-flat inner system.  Recall that the small-singular-orbit
data are
\[
b_0=\eps,\qquad c_0=\eps\rho .
\]
Thus both non-collapsing radii at the left singular orbit are of order
\(\eps\).  To study the geometry on that same
length scale, we measure radial distance in units of \(\eps\), writing
\[
t=\eps\tau.
\]
We rescale the three orbit radii by the same factor:
\[
A_\eps(\tau)=\eps^{-1}a(\eps\tau),\qquad
B_\eps(\tau)=\eps^{-1}b(\eps\tau),\qquad
C_\eps(\tau)=\eps^{-1}c(\eps\tau).
\]
Equivalently,
\[
a(t)=\eps A_\eps(\tau),\qquad
b(t)=\eps B_\eps(\tau),\qquad
c(t)=\eps C_\eps(\tau),\qquad \tau=t/\eps.
\]
This is not an additional ansatz for the solution; it is simply the original
metric written in coordinates and units adapted to the small left singular
orbit.  In particular,
\[
B_\eps(0)=1,\qquad C_\eps(0)=\rho,
\]
so the rescaled singular orbit has order-one size as \(\eps\to0\).

For each fixed \(\varepsilon>0\), the original radial variable
\(t\in[0,L]\) corresponds to
\[
        0\leq \tau\leq \frac{L}{\varepsilon}.
\]
Thus the rescaled interval becomes longer as \(\varepsilon\to0\);
in the limiting Ricci-flat inner problem, the radial variable
\(\tau\) ranges over the half-line \([0,\infty)\).

The chain rule gives
\[
a_t=(A_\eps)_\tau,\qquad a_{tt}=\eps^{-1}(A_\eps)_{\tau\tau},
\]
and the analogous identities for \(b\) and \(c\).  Substituting these
relations into \eqref{eq:PE-a}--\eqref{eq:PE-c} and multiplying each
equation by \(\eps\) gives
\begin{align}
(A_\eps)_{\tau\tau}
&=
\frac{1-(A_\eps)_\tau^2}{A_\eps}
-2\frac{(A_\eps)_\tau (B_\eps)_\tau}{B_\eps}
-3\frac{(A_\eps)_\tau (C_\eps)_\tau}{C_\eps}
-7\eps^2A_\eps,\label{eq:scaled-PE-a}\\
(B_\eps)_{\tau\tau}
&=
\frac{1-(B_\eps)_\tau^2}{B_\eps}
-2\frac{(A_\eps)_\tau (B_\eps)_\tau}{A_\eps}
-3\frac{(B_\eps)_\tau (C_\eps)_\tau}{C_\eps}
-7\eps^2B_\eps,\label{eq:scaled-PE-b}\\
(C_\eps)_{\tau\tau}
&=
\frac{2(1-(C_\eps)_\tau^2)}{C_\eps}
-2\frac{(A_\eps)_\tau (C_\eps)_\tau}{A_\eps}
-2\frac{(B_\eps)_\tau (C_\eps)_\tau}{B_\eps}
-7\eps^2C_\eps.\label{eq:scaled-PE-c}
\end{align}
The first three terms on the right-hand side of each equation are exactly
the corresponding Ricci-flat inner equation displayed in Remark~\ref{rem:inner-rescaling}.
The only terms left from the positive Einstein constant are
\(-7\eps^2A_\eps\), \(-7\eps^2B_\eps\), and \(-7\eps^2C_\eps\), respectively.  Thus the
positive-Einstein inner system is an \(O(\eps^2)\) perturbation of the
Ricci-flat inner system; in particular, no additional power of \(\eps\) is
hidden in this perturbation.  Thus, on the scale of the small singular
orbit, the Ricci-flat family studied next is the leading-order geometry of
the exact Einstein family.

\section{Global Ricci-flat inner family}\label{sec:global-response}

Section~\ref{sec:analytic-setting} showed that, after rescaling around a
small left singular orbit, the positive-Einstein equation approach the
Ricci-flat inner problem.  We now study tthe resulting Ricci-flat solution
globally in the rescaled radial variable.  There are three steps.  We first
prove global existence and monotonicity of the physical radius functions.
We then divide out the overall scale and use a compact normalized phase
space to prove convergence to a single Ricci-flat cone.  Finally, we return
to physical anisotropy variables and extract the two complex response
coefficients \(Q(\rho)\) and \(X(\rho)\).  These coefficients are the only
data from the inner trajectory that will enter the later matching problem.

For each singular-orbit ratio \(\rho>0\), let
\[
g_{\mathrm{inner},\rho}
=
d\tau^2+A_\rho(\tau)^2g_{S^2}
+B_\rho(\tau)^2g_{S^2}
+C_\rho(\tau)^2g_{S^3}
\]
be the local Ricci-flat solution obtained in
Remark~\ref{rem:inner-rescaling}.  Recall that
\[
\tau=\frac{t}{\eps}
\]
is the rescaled radial variable.  Thus \(\tau\to\infty\) means moving
far from the singular orbit on the inner scale; it does not mean that the
original radial variable \(t\) tends to infinity.  In the overlap regime
used later for inner--outer matching,
\[
1\ll\tau\ll\eps^{-1},
\qquad\text{equivalently}\qquad
\eps\ll t=\eps\tau\ll1.
\]
Hence large \(\tau\) describes points that are still close to the singular
orbit in the original \(t\)-variable, but far from it compared with the
small scale \(\eps\).

Our first task is to prove that the Ricci-flat inner solution exists for
every \(\tau\ge0\) and to determine its asymptotic geometry as
\(\tau\to\infty\).  These large-\(\tau\) asymptotics provide the
response data that will later be transferred through the outer problem to
the Einstein matching section.  Throughout this section we keep the
subscript \(\rho\) on the physical radius functions \(A_\rho,B_\rho,C_\rho\)
to make clear which member of the inner family is being considered.

\subsection{The Ricci-flat singular-orbit family}

For \(\rho>0\), consider
\begin{subequations}\label{eq:RF-global}
\begin{align}
A_\rho''
&=\frac{1-(A_\rho')^2}{A_\rho}
 -2\frac{A_\rho'B_\rho'}{B_\rho}-3\frac{A_\rho'C_\rho'}{C_\rho},\\
B_\rho''
&=\frac{1-(B_\rho')^2}{B_\rho}
 -2\frac{A_\rho'B_\rho'}{A_\rho}-3\frac{B_\rho'C_\rho'}{C_\rho},\\
C_\rho''
&=\frac{2(1-(C_\rho')^2)}{C_\rho}
 -2\frac{A_\rho'C_\rho'}{A_\rho}-2\frac{B_\rho'C_\rho'}{B_\rho},
\end{align}
\end{subequations}
with smooth singular-orbit data
\[
A_\rho(0)=0,\quad A_\rho'(0)=1,\qquad
B_\rho(0)=1,\quad B_\rho'(0)=0,\qquad
C_\rho(0)=\rho,\quad C_\rho'(0)=0.
\]
Put
\[
\alpha=\frac{A_\rho'}{A_\rho},\qquad
\beta=\frac{B_\rho'}{B_\rho},\qquad
\gamma=\frac{C_\rho'}{C_\rho},\qquad
\mathcal H=2\alpha+2\beta+3\gamma .
\]
The radial Ricci-flat equation is
\begin{equation}\label{eq:Hprime-global}
\mathcal H'
=
-\bigl(2\alpha^2+2\beta^2+3\gamma^2\bigr).
\end{equation}
The Ricci-flat Hamiltonian constraint is
\begin{equation}\label{eq:Hamiltonian-global}
\mathcal H^2
=
2\alpha^2+2\beta^2+3\gamma^2
+\frac2{A_\rho^2}+\frac2{B_\rho^2}+\frac6{C_\rho^2}.
\end{equation}

\begin{theorem}[Global existence of the Ricci-flat inner family]
\label{thm:RF-global-existence}
For every \(\rho>0\), the Ricci-flat inner solution
\[
(A_\rho,B_\rho,C_\rho)
\]
with the singular-orbit data specified in Remark~\ref{rem:inner-rescaling} extends smoothly to all
\[
\tau\in[0,\infty).
\]
Moreover,
\[
0<A_\rho'(\tau),B_\rho'(\tau),C_\rho'(\tau)<1
\qquad(\tau>0).
\]
Finally, all three orbit radii grow without bound:
\[
A_\rho(\tau)\longrightarrow\infty,\qquad
B_\rho(\tau)\longrightarrow\infty,\qquad
C_\rho(\tau)\longrightarrow\infty
\qquad\text{as }\tau\to\infty.
\]
Thus the solution does not encounter a second singular orbit or a
finite-radius collapse in the inner problem.
\end{theorem}

\begin{proof}
The Taylor expansions in Remark~\ref{rem:inner-rescaling} imply that, for each fixed
\(\rho>0\), there exists \(\delta=\delta(\rho)>0\) such that
\[
0<A_\rho'(\tau),B_\rho'(\tau),C_\rho'(\tau)<1
\qquad\text{for }0<\tau<\delta.
\]
Suppose, for example, that \(A_\rho'\) had a first zero at some
\(\tau_*>0\).  Then at \(\tau_*\)
\[
A_\rho''=\frac1{A_\rho}>0,
\]
contradicting first contact from above.  The same argument gives positivity
of \(B_\rho'\) and \(C_\rho'\).  If instead \(A_\rho'\) first reached \(1\)
at some \(\tau_*>0\), then at \(\tau_*\)
\[
A_\rho''=-2\frac{B_\rho'}{B_\rho}-3\frac{C_\rho'}{C_\rho}<0,
\]
contradicting first contact from below.  The equations for \(B_\rho'\) and \(C_\rho'\)
give the same contradiction at a first value \(1\).

Fix \(\tau_0>0\).  On a finite interval \([\tau_0,T]\), positivity and
\(0<A_\rho',B_\rho',C_\rho'<1\) give
\[
A_\rho(\tau_0)\le A_\rho(\tau)\le A_\rho(\tau_0)+T-\tau_0,
\]
and the analogous bounds for \(B_\rho,C_\rho\).  Their positive lower bounds at
\(\tau_0\), together with the derivative bounds, keep the state in a compact
subset of the regular ODE domain on every finite interval.  The standard
continuation criterion therefore rules out a finite maximal time and gives
global existence.

Weighted Cauchy--Schwarz in \eqref{eq:Hprime-global} gives
\[
\mathcal H'\le-\frac{\mathcal H^2}{7}.
\]
The Ricci-flat Hamiltonian constraint has a strictly positive right-hand
side, so \(\mathcal H\) cannot vanish.  Since the singular-orbit expansion gives
\(\mathcal H>0\) for small positive \(\tau\), it follows that
\(\mathcal H>0\) on the whole interval of existence.  Integrating
\[
\mathcal H'\le-\frac{\mathcal H^2}{7}
\]
from any \(\tau_0>0\) gives
\[
\frac1{\mathcal H(\tau)}
\ge
\frac1{\mathcal H(\tau_0)}+\frac{\tau-\tau_0}{7},
\qquad \tau\ge\tau_0,
\]
and therefore \(\mathcal H(\tau)\to0\) as \(\tau\to\infty\).  If, say, \(B_\rho\to B_\infty<\infty\), then
\[
0<\frac{B_\rho'}{B_\rho}\le\frac{\mathcal H}{2}\to0,
\]
so \(B_\rho'\to0\), while the \(B_\rho\)-equation yields
\[
B_\rho''\to\frac{1}{B_\infty}>0,
\]
a contradiction.  The arguments for \(A_\rho\) and \(C_\rho\) are identical.
\end{proof}

\subsection{Mean-curvature normalized phase space}
\label{subsec:mean-curvature-normalized}

Theorem~\ref{thm:RF-global-existence} tells us that the three radii grow
forever, but by itself it does not say whether their \emph{ratios} settle
down.  That relative shape is what matters for detection.  To see it, we
divide out the overall radial scale and retain only the relative expansion
rates and relative orbit sizes.  This normalization forgets the absolute size of
the metric, but keeps precisely the shape information needed to distinguish
the two anisotropy channels.  It also turns the asymptotic regimes at
infinite physical distance into finite points of a compact phase space.

For the asymptotic geometry we use mean-curvature normalized variables in
the Dancer--Wang/Chi phase-space tradition
\cite{Chi2019,DancerWang2000,DancerWang2009Expanding}.  Related
three-summand Ricci-flat systems, Hamiltonian formulations, and
superpotential structures are developed in
\cite{DancerWang2000,DancerWang2005,DancerWang2011}.  Nienhaus and
Wink~\cite{NienhausWink2025} likewise use normalized variables built from
orbit radii and their logarithmic derivatives in their two-factor
positive-Einstein problem.  We adapt the normalized dynamical framework
here to the present product-sphere family with dimensions \((2,2,3)\).

\begin{remark}[What the normalized variables do]
The variables \(X_i\) record each logarithmic expansion rate relative to
the total mean curvature \(\mathcal H\), while the variables \(Y_i\) record
each orbit radius in the corresponding mean-curvature scale.  Thus they
discard the absolute size of the metric but retain its relative shape.
This has an important practical effect: geometric limits that occur at
\(\tau=0\) or only after very large changes of scale can be represented by
ordinary finite points of the normalized phase space.  We use this
compactified description to follow the small-ratio limit and then return to
the physical anisotropy variables \(\zeta,\xi\) and their response
coefficients \(Q,X\).
\end{remark}

Define
\[
X_A=\frac{\alpha}{\mathcal H},\qquad
X_B=\frac{\beta}{\mathcal H},\qquad
X_C=\frac{\gamma}{\mathcal H},
\]
\[
Y_A=\frac1{A_\rho\mathcal H},\qquad
Y_B=\frac1{B_\rho\mathcal H},\qquad
Y_C=\frac1{C_\rho\mathcal H},
\qquad
d\eta=\mathcal H\,d\tau,
\]
and
\[
G=2X_A^2+2X_B^2+3X_C^2.
\]
A dot denotes \(d/d\eta\).  The variable \(\eta\) is determined only up
to an additive constant; for example, one may fix a reference point
\(\tau_*>0\) and set
\[
\eta(\tau)=\int_{\tau_*}^{\tau}\mathcal H(s)\,ds.
\]
For a fixed \(\rho>0\), we denote by
\[
U_\rho(\eta)
=
\bigl(X_A,X_B,X_C;Y_A,Y_B,Y_C\bigr)(\eta)
\]
the normalized trajectory associated with the Ricci-flat inner solution
\((A_\rho,B_\rho,C_\rho)\).

\begin{lemma}[Exact normalized system]\label{lem:normalized}
The Ricci-flat equations imply
\begin{subequations}\label{eq:normalized}
\begin{align}
\dot X_A&=X_A(G-1)+Y_A^2,\\
\dot X_B&=X_B(G-1)+Y_B^2,\\
\dot X_C&=X_C(G-1)+2Y_C^2,\\
\dot Y_A&=Y_A(G-X_A),\\
\dot Y_B&=Y_B(G-X_B),\\
\dot Y_C&=Y_C(G-X_C).
\end{align}
\end{subequations}
Moreover the flow is constrained by
\begin{equation}\label{eq:linearconstraint}
2X_A+2X_B+3X_C=1
\end{equation}
and
\begin{equation}\label{eq:quadraticconstraint}
G+2Y_A^2+2Y_B^2+6Y_C^2=1.
\end{equation}
Both constraints are invariant under \eqref{eq:normalized}.
\end{lemma}

\begin{proof}
From \eqref{eq:RF-global},
\[
\alpha'=\frac1{A_\rho^2}-\mathcal H\alpha,\qquad
\beta'=\frac1{B_\rho^2}-\mathcal H\beta,\qquad
\gamma'=\frac2{C_\rho^2}-\mathcal H\gamma.
\]
Equation \eqref{eq:Hprime-global} gives \(\mathcal H'/\mathcal H^2=-G\).  Differentiating the
definitions with respect to \(\eta\) gives \eqref{eq:normalized}.
Equation \eqref{eq:linearconstraint} is the definition of \(\mathcal H\) divided by
\(\mathcal H\), and \eqref{eq:quadraticconstraint} is
\eqref{eq:Hamiltonian-global} divided by \(\mathcal H^2\).

A direct differentiation gives
\[
\dot F_1=(G-1)F_1+F_2,\qquad \dot F_2=2GF_2,
\]
for the left-hand sides \(F_1,F_2\) of the two constraints; hence their
common zero set is invariant.
\end{proof}

Define
\[
\mathcal P=
\left\{
(X,Y):X_i,Y_i\ge0,\ 
\eqref{eq:linearconstraint},\ \eqref{eq:quadraticconstraint}
\right\}.
\]

\begin{lemma}[Compact positive phase space]\label{lem:compact}
The set \(\mathcal P\) is compact and forward invariant.  Its relative interior
with \(Y_A Y_B Y_C>0\) is forward invariant.
\end{lemma}

\begin{proof}
The two constraints give uniform bounds for every coordinate, so
\(\mathcal P\) is compact.  If \(X_A=0\), then
\(\dot X_A=Y_A^2\ge0\), and analogously for \(X_B,X_C\).
Each hypersurface \(Y_i=0\) is invariant because the corresponding
equation has a factor \(Y_i\).  Thus the positive region is forward
invariant.
\end{proof}

\subsection{Three geometric equilibria}

The normalized variables introduced in the previous subsection turn
self-similar limiting geometries into equilibrium points of the dynamicalg
system.  Three such points are important here.  They represent,
respectively, the geometry near the left singular orbit, an intermediate
limiting geometry that appears in the small-ratio analysis, and the
Ricci-flat cone approached at large distance.  We denote them by
\(P_0,P_5,P_7\) and identify them below.

First consider the normalized trajectory \(U_\rho\) associated with a fixed
Ricci-flat inner solution.  The expansion at the left singular orbit is
\[
A_\rho(\tau)=\tau+O(\tau^3),\qquad
B_\rho(\tau)=1+O(\tau^2),\qquad
C_\rho(\tau)=\rho+O(\tau^2),
\]
so
\[
\mathcal H(\tau)=\frac2{\tau}+O(\tau).
\]
Since \(d\eta=\mathcal H\,d\tau\), integration gives
\[
\eta=2\log\tau+O(1)
\qquad(\tau\to0^+).
\]
Thus \(\tau\to0^+\), which means approaching the left singular orbit,
corresponds to \(\eta\to-\infty\).  Here \(-\infty\) is not an additional
point of the physical radial domain: it describes the limiting end of the
new time coordinate \(\eta\).  Substituting the above expansions into the
definitions of \(X_i,Y_i\) gives
\[
U_\rho(\eta)\longrightarrow
P_0:=
\left(\frac12,0,0;\frac12,0,0\right)
\qquad(\eta\to-\infty).
\]
Thus \(P_0\) is the normalized representation of the left singular orbit.

The other two limiting geometries correspond to the equilibrium points
\[
P_5:=
\left(\frac15,0,\frac15;\frac25,0,\frac{\sqrt2}{5}\right),
\]
\[
P_7:=
\left(\frac17,\frac17,\frac17;
\frac{\sqrt6}{7},\frac{\sqrt6}{7},\frac{\sqrt3}{7}\right).
\]
Direct substitution shows that all three are equilibria of
\eqref{eq:normalized} on \(\mathcal P\).

\begin{remark}[Geometric meaning of the three equilibria]
The three points describe three different stages of the limiting geometry.
The point \(P_0\) represents the left singular-orbit end of every normalized
inner solution.  The point \(P_5\) appears only in the limiting process
\(\rho\to0\): it is not itself a smooth member of the physical singular-orbit family,
because setting \(\rho=0\) would give \(C_\rho(0)=0\).  Finally, \(P_7\)
represents the Ricci-flat cone that will govern the large-\(\tau\)
asymptotics.  The small-ratio analysis below shows that these stages are
connected by the limiting trajectory
\[
P_0\longrightarrow P_5\longrightarrow P_7.
\]
\end{remark}

\begin{proposition}[Constrained spectra]\label{prop:spectra}
On the tangent space of the normalized phase space cut out by the linear
constraint \eqref{eq:linearconstraint} and the quadratic constraint
\eqref{eq:quadraticconstraint},
\[
\operatorname{spec}DF(P_0)
=
\left\{-\frac12,-\frac12,\frac12,\frac12\right\},
\]
\[
\operatorname{spec}DF(P_5)
=
\left\{-\frac45,-\frac25+\frac{2i}{5},
-\frac25-\frac{2i}{5},\frac15\right\},
\]
and
\[
\operatorname{spec}DF(P_7)
=
\left\{-\frac37+\frac{i\sqrt3}{7},
-\frac37-\frac{i\sqrt3}{7}\right\},
\]
the last pair having complex multiplicity two.
At \(P_5\), the unstable eigenspace is one-dimensional and is spanned by
the coordinate vector that varies \(Y_B\) while leaving the other five
normalized coordinates fixed.
\end{proposition}

\begin{proof}
Let
\[
F_1=2X_A+2X_B+3X_C-1,\qquad
F_2=G+2Y_A^2+2Y_B^2+6Y_C^2-1.
\]
At an equilibrium \(P\in\mathcal P\), the constrained linearization means
the Jacobian of \eqref{eq:normalized} restricted to
\[
T_P\mathcal P=\ker dF_1(P)\cap\ker dF_2(P).
\]
We now compute this restriction at the three equilibria.

At \(P_5\) one has \(Y_B=0\).  Consequently the coordinate vector
\[
e_{Y_B}=(0,0,0;0,1,0)
\]
satisfies both linearized constraints and hence belongs to
\(T_{P_5}\mathcal P\).  Moreover,
\[
\dot Y_B=Y_B(G-X_B)
\]
gives
\[
\delta\dot Y_B=\frac15\,\delta Y_B
\]
at \(P_5\).  Thus \(e_{Y_B}\) is an eigenvector with the positive
eigenvalue \(1/5\); we call this the \(Y_B\)-direction.
Restricting the Jacobian to a complementary three-dimensional subspace of
\(T_{P_5}\mathcal P\) gives characteristic polynomial
\[
\left(\lambda+\frac45\right)
\left[\left(\lambda+\frac25\right)^2+\frac4{25}\right].
\]
All three roots have negative real part, so the \(Y_B\)-direction is the
entire unstable eigenspace.

At \(P_7\), impose the two linearized constraint equations
\[
dF_1(P_7)\,\delta U=0,\qquad dF_2(P_7)\,\delta U=0.
\]
Solving these two linear relations for two of the six coordinate variations
and substituting them into the linearised system produces a
four-dimensional system on \(T_{P_7}\mathcal P\).  Its characteristic
polynomial is
\[
\left(49\lambda^2+42\lambda+12\right)^2,
\]
which gives the stated conjugate pair, each with multiplicity two.

Finally, at \(P_0\) we again restrict the six-dimensional Jacobian by
imposing
\[
dF_1(P_0)\,\delta U=0,\qquad dF_2(P_0)\,\delta U=0.
\]
The resulting four-dimensional linear map on \(T_{P_0}\mathcal P\) has
characteristic polynomial
\[
\left(\lambda^2-\frac14\right)^2,
\]
and hence eigenvalues \(-1/2,-1/2,1/2,1/2\).
\end{proof}

\subsection{Global asymptotic conicality}

We now show that every member of the Ricci-flat inner family with
positive radius functions approaches the same cone at large distance.  This provides the common
asymptotic background from which the response coefficients are extracted.

\begin{lemma}[Positive normalized orbits converge to \(P_7\)]
\label{lem:positive-to-P7}
Let \(U(\eta)\) be a forward orbit of \eqref{eq:normalized} in
\(\mathcal P\).  If
\[
Y_A(\eta_0)Y_B(\eta_0)Y_C(\eta_0)>0
\]
at one finite time, then
\[
U(\eta)\longrightarrow P_7
\qquad(\eta\to\infty).
\]
\end{lemma}

\begin{proof}
The relative interior \(Y_AY_BY_C>0\) is forward invariant by
Lemma~\ref{lem:compact}.  Define
\[
V_7=Y_A^2Y_B^2Y_C^3.
\]
Using \eqref{eq:normalized} and \eqref{eq:linearconstraint},
\[
\frac{\dot V_7}{V_7}
=
2(G-X_A)+2(G-X_B)+3(G-X_C)
=7G-1.
\]
Weighted Cauchy--Schwarz gives
\[
G\ge
\frac{(2X_A+2X_B+3X_C)^2}{7}
=\frac17,
\]
so \(V_7\) is nondecreasing.  Since \(\mathcal P\) is compact,
the forward orbit is precompact and \(V_7\) is bounded.  Moreover
\(V_7(\eta)\ge V_7(\eta_0)>0\).  Hence the forward orbit is contained in
\[
\mathcal K_{\eta_0}
:=
\{U\in\mathcal P:V_7(U)\ge V_7(U(\eta_0))\},
\]
a compact positively invariant subset of the strict interior
\(Y_AY_BY_C>0\).  Thus the quotient \(\dot V_7/V_7\) used above is
legitimate on the entire invariant set to which LaSalle's invariance
principle~\cite{LaSalleLefschetz1961} will be applied, and its \(\omega\)-limit set cannot meet
\(Y_AY_BY_C=0\).

Equality \(\dot V_7=0\) holds precisely when
\[
X_A=X_B=X_C=\frac17.
\]
A complete invariant orbit contained in this equality set must also have
\(\dot X_A=\dot X_B=\dot X_C=0\).  The first three equations of
\eqref{eq:normalized} then force
\[
Y_A=Y_B=\frac{\sqrt6}{7},\qquad
Y_C=\frac{\sqrt3}{7}.
\]
Thus \(P_7\) is the maximal invariant subset of the zero set of
\(\dot V_7\) contained in \(\mathcal K_{\eta_0}\).  LaSalle's invariance
principle on this compact positively invariant set gives
\(U(\eta)\to P_7\).
\end{proof}

\begin{theorem}[Global asymptotic conicality]
\label{thm:global-AC}
For every \(\rho>0\), let \(U_\rho\) be the normalized trajectory defined
in Subsection~4.2 for the Ricci-flat inner solution of
Theorem~\ref{thm:RF-global-existence}.  Then
\[
U_\rho(\eta)\longrightarrow P_7
\qquad(\eta\to\infty).
\]
Consequently
\[
\frac{A_\rho(\tau)}{\tau}\to\frac1{\sqrt6},\qquad
\frac{B_\rho(\tau)}{\tau}\to\frac1{\sqrt6},\qquad
\frac{C_\rho(\tau)}{\tau}\to\frac1{\sqrt3}.
\]
\end{theorem}

\begin{proof}
For every \(\rho>0\), the normalized trajectory \(U_\rho\) has
\(Y_A,Y_B,Y_C>0\) at every finite normalized time.  Hence
Lemma~\ref{lem:positive-to-P7} gives \(U_\rho(\eta)\to P_7\).

Reconstruction gives the physical cone ratios directly.
Indeed
\[
\frac{d\log \mathcal H}{d\eta}=-G\to-\frac17,\qquad
\frac{d\tau}{d\eta}=\frac1{\mathcal H},
\]
and the retained physical global continuation implies \(\tau\to\infty\).
To make the reconstruction step explicit, write
\[
\tau\mathcal H=\frac{\tau}{1/\mathcal H}.
\]
Both numerator and denominator tend to \(+\infty\), and
\[
\frac{d\tau}{d\eta}=\frac1{\mathcal H},
\qquad
\frac{d}{d\eta}\left(\frac1{\mathcal H}\right)
=\frac{G}{\mathcal H}.
\]
Therefore l'Hospital's rule gives
\[
\lim_{\eta\to\infty}\tau \mathcal H
=\lim_{\eta\to\infty}\frac1G
=7.
\]
Since \(A_\rho\mathcal H=Y_A^{-1}\), \(B_\rho\mathcal H=Y_B^{-1}\), and \(C_\rho\mathcal H=Y_C^{-1}\),
the asserted ratios follow.
\end{proof}

\begin{corollary}[Uniform cone entry on compact ratio intervals]
\label{cor:chi-uniform-entry}
Let \(J\Subset(0,\infty)\).  There are a neighborhood
\(\mathcal U_7\) of \(P_7\) in \(\mathcal P\), a normalized time
\(\eta_J\), and constants \(C_J,\gamma_J>0\) such that every
\(U_\rho\), \(\rho\in J\), lies in \(\mathcal U_7\) for
\(\eta\ge\eta_J\).

To state the corresponding linear estimate, let \(F\) denote the vector
field in \eqref{eq:normalized}.  Along \(U_\rho\), a tangent perturbation
\(v(\eta)\in T_{U_\rho(\eta)}\mathcal P\) satisfies the constrained
variational equation
\begin{equation}\label{eq:constrained-variational}
\dot v
=
DF(U_\rho(\eta))\,v,
\qquad
v(\eta)\in T_{U_\rho(\eta)}\mathcal P.
\end{equation}
We denote by
\[
\mathcal U_\rho(\eta,\sigma):
T_{U_\rho(\sigma)}\mathcal P
\longrightarrow
T_{U_\rho(\eta)}\mathcal P
\]
its evolution operator, so that
\(v(\eta)=\mathcal U_\rho(\eta,\sigma)v(\sigma)\).
Then
\[
\|\mathcal U_\rho(\eta,\sigma)\|
\le C_J e^{-\gamma_J(\eta-\sigma)}
\qquad(\eta\ge\sigma\ge\eta_J,\ \rho\in J).
\]
\end{corollary}

\begin{proof}
The singular-orbit family depends continuously on \(\rho\) on compact finite-time
intervals.  Theorem~\ref{thm:global-AC} gives convergence to the same
hyperbolic sink \(P_7\) for each \(\rho\in J\).  Choose a sufficiently small
forward-invariant neighborhood of \(P_7\).  For each \(\rho_0\in J\),
convergence supplies a finite entry time.  Continuous parameter dependence
on the corresponding finite interval gives a neighborhood of \(\rho_0\)
whose trajectories enter the same sink neighborhood by that time.
Compactness of \(J\) gives a finite subcover, and the maximum of these
finitely many entry times is a common entry time \(\eta_J\).

It remains to justify the uniform estimate for
\eqref{eq:constrained-variational}.  Proposition~\ref{prop:spectra} shows
that the linearization of \(F\) on \(T_{P_7}\mathcal P\) has spectrum
strictly in the left half-plane.  Hence the constant-coefficient
variational equation at \(P_7\) is exponentially stable.  By shrinking
\(\mathcal U_7\), the restricted matrices
\[
DF(U)\big|_{T_U\mathcal P},\qquad U\in\mathcal U_7,
\]
are uniformly close, after a smooth local identification of the tangent
spaces with \(T_{P_7}\mathcal P\), to the stable matrix
\(DF(P_7)|_{T_{P_7}\mathcal P}\).  Choose the local identification smoothly and equal to the identity at
\(P_7\).  There is no circularity here: after the common entry time,
ordinary nonlinear stability of the hyperbolic sink \(P_7\) first gives,
after shrinking \(\mathcal U_7\) if necessary, uniform constants
\(C,\gamma>0\) such that
\[
|U_\rho(\eta)-P_7|
\le
C e^{-\gamma(\eta-\eta_J)}
|U_\rho(\eta_J)-P_7|,
\qquad \rho\in J,\quad \eta\ge\eta_J.
\]
This follows directly from the stable linearization at \(P_7\) and the
quadratic remainder of the smooth vector field, with constants uniform
on the fixed neighborhood.  In the resulting fixed vector space
\(T_{P_7}\mathcal P\), the time dependence of the tangent-space
identification contributes an additional coefficient term.  The preceding
uniform exponential convergence shows that both this term and the
difference from \(DF(P_7)|_{T_{P_7}\mathcal P}\) tend to zero uniformly
for \(\rho\in J\).  The roughness theorem for exponential
dichotomies---that an exponential dichotomy persists, with slightly
weakened constants, under sufficiently small bounded perturbations of the
coefficient matrix---then gives uniform constants \(C_J,\gamma_J>0\);
see, for example, Coppel~\cite[Chapter~4]{Coppel1978}.  This yields the
stated estimate for all \(\rho\in J\) after the common entry time.
\end{proof}

\subsection{Ricci-flat response coefficients}

We now study how the two anisotropies of a Ricci-flat inner solution
behave as \(\tau\to\infty\).  We begin with the limiting linear equation, where the meaning of the
complex response coefficient can be seen explicitly.

\subsubsection{A model calculation for the complex response}
\label{subsec:model-response}

Near the limiting Ricci-flat cone, the leading equation for the first
anisotropy approaches
\begin{equation}\label{eq:model-q}
\zeta_{ss}+6\zeta_s+12\zeta=0,
\end{equation}
where \(s=\log\tau\).  Its characteristic roots are
\[
-3\pm i\sqrt3.
\]
Hence every real solution of \eqref{eq:model-q} has the form
\begin{equation}\label{eq:model-q-solution}
\zeta(s)
=
e^{-3s}
\bigl(
A\cos(\sqrt3\,s)+B\sin(\sqrt3\,s)
\bigr).
\end{equation}
Equivalently, after combining the two real constants into one complex
number \(Q\), we may write
\begin{equation}\label{eq:model-q-complex}
\zeta(s)
=
e^{-3s}\Re\!\left(Qe^{i\sqrt3 s}\right).
\end{equation}
Thus \(Q\) records two pieces of information at once: its modulus gives
the size of the leading anisotropy, while its argument gives its phase.

Since \(\tau=e^s\), formula~\eqref{eq:model-q-complex} becomes
\begin{equation}\label{eq:model-q-tau}
\zeta(\tau)
=
\tau^{-3}
\Re\!\left(
Qe^{i\sqrt3\log\tau}
\right).
\end{equation}
This makes the behavior near the cone transparent.  The factor
\(\tau^{-3}\) tends to zero, so the anisotropy decays, while the factor
\(e^{i\sqrt3\log\tau}\) continues to rotate.  Thus the solution approaches
the cone while its leading anisotropy repeatedly changes sign.

The same calculation explains why the small parameter \(\eps\) produces
a logarithmic phase in the compact Einstein problem.  Recall that the
inner variable is
\[
\tau=\frac{t}{\eps}.
\]
Substituting this into \eqref{eq:model-q-tau} gives
\begin{align}
\zeta(t/\eps)
&=
\left(\frac{\eps}{t}\right)^3
\Re\!\left(
Qe^{i\sqrt3\log(t/\eps)}
\right) \notag\\
&=
\left(\frac{\eps}{t}\right)^3
\Re\!\left(
Qe^{i\sqrt3\log t}
e^{-i\sqrt3\log\eps}
\right).
\label{eq:model-eps-phase}
\end{align}
For fixed \(t>0\), decreasing \(\eps\) therefore does more than decrease
the size of this leading term: it also rotates its phase by
\[
-\sqrt3\log\eps.
\]
For example, replacing \(\eps\) by
\[
e^{-\pi/\sqrt3}\eps
\]
changes this phase by \(\pi\), and hence reverses the sign of the leading
real oscillation.

This model calculation contains the basic mechanism used below.  For the
actual Ricci-flat trajectories \(U_\rho\), the coefficients of the
anisotropy equation are not constant, but they converge to those of
\eqref{eq:model-q}.  The leading large-\(\tau\) behavior is therefore
again encoded by a complex coefficient, now denoted by \(Q(\rho)\).  The
second anisotropy channel similarly gives a complex coefficient
\(X(\rho)\).  The analysis below constructs these coefficients
rigorously and determines how they depend on the Ricci-flat trajectory.

\begin{remark}
The calculation above describes the limiting linearized behavior near the
Ricci-flat cone; equation~\eqref{eq:model-q} is not the full nonlinear
Einstein equation.  It makes the response coefficient and the logarithmic
phase explicit; the corresponding statements for the nonlinear family are
proved below.
\end{remark}

We now return to the nonlinear Ricci-flat family.  Recall that
\[
\zeta_\rho(\tau)=\log\frac{A_\rho(\tau)}{B_\rho(\tau)}
\]
measures the imbalance between the two \(S^2\)-radii, while
\[
\xi_\rho(\tau)
=
\log\frac{C_\rho(\tau)}
{\sqrt2\,\sqrt{A_\rho(\tau)B_\rho(\tau)}}
\]
measures the deviation of the \(S^3\)-radius from the limiting cone
relation
\[
C_\rho=\sqrt2\,\sqrt{A_\rho B_\rho}.
\]
Theorem~\ref{thm:global-AC} implies that both \(\zeta_\rho\) and \(\xi_\rho\)
tend to zero as \(\tau\to\infty\), but their leading decay is oscillatory.

We will show that this leading oscillatory behavior is described by two
complex numbers: \(Q(\rho)\) for \(\zeta_\rho\) and \(X(\rho)\) for
\(\xi_\rho\).  More precisely, they are characterized by
\[
\zeta_\rho(\tau)
=
\tau^{-3}\Re\!\left(Q(\rho)\tau^{i\sqrt3}\right)
+o(\tau^{-3}),
\qquad
\xi_\rho(\tau)
=
\tau^{-3}\Re\!\left(X(\rho)\tau^{i\sqrt3}\right)
+o(\tau^{-3}).
\]
We call \(Q(\rho)\) and \(X(\rho)\) the \emph{response coefficients} of
the Ricci-flat inner solution with singular-orbit ratio \(\rho\), and we collect them
as its \emph{response data}
\[
\boxed{\mathcal R(\rho):=(Q(\rho),X(\rho))\in\mathbb C^2.}
\]

For later use we also set
\[
m_\rho(\tau)=\frac12\log\!\bigl(A_\rho(\tau)B_\rho(\tau)\bigr).
\]

These coefficients retain precisely the finite-dimensional information
from the large-\(\tau\) inner asymptotics that will be transferred to the
outer matching problem; the full Ricci-flat trajectory will no longer be
needed.

We next derive the equations that determine these asymptotic modes.  For
the remainder of this subsection, with \(\rho\) fixed, we suppress the
subscript \(\rho\) only on the derived variables \(\zeta_\rho,m_\rho,\xi_\rho\)
and write
\[
\zeta=\log\frac{A_\rho}{B_\rho},\qquad
m=\frac12\log(A_\rho B_\rho),\qquad
\xi=\log\frac{C_\rho}{\sqrt2 e^m}.
\]
A straightforward computation from the Ricci-flat equations gives the
exact identities
\begin{equation}\label{eq:q-exact-global}
\zeta''+\mathcal H \zeta'+2e^{-2m}\sinh \zeta=0
\end{equation}
and
\begin{equation}\label{eq:xi-exact-global}
\xi''+\mathcal H\xi'
+e^{-2m}\bigl(\cosh \zeta-e^{-2\xi}\bigr)=0.
\end{equation}
We refer to the pair \eqref{eq:q-exact-global}--\eqref{eq:xi-exact-global}
as the \emph{anisotropy system}: it is simply the system satisfied by the
two variables \(\zeta\) and \(\xi\) that measure departure from the limiting
cone shape.

To read off the oscillatory decay, set
\[
s=\log\tau
\]
and let a dot denote \(d/ds\) in this subsection.  Then
\(\tau^2f''=\ddot f-\dot f\), and the \(\zeta\)-equation becomes
\begin{equation}\label{eq:q-linear-global}
\ddot \zeta+(\tau\mathcal H-1)\dot \zeta
+2\tau^2e^{-2m}\frac{\sinh q}{q}\,\zeta=0,
\end{equation}
with the smooth convention \(\sinh \zeta/\zeta=1\) at \(\zeta=0\).  Since
Theorem~\ref{thm:global-AC} gives
\[
\tau\mathcal H\to7,\qquad \tau^2e^{-2m}\to6,
\]
the limiting linear equation is
\[
u_{ss}+6u_s+12u=0,
\]
whose characteristic exponents are
\[
-3\pm i\sqrt3.
\]
This explains both the factor \(\tau^{-3}\) and the oscillation
\(\tau^{i\sqrt3}=e^{i\sqrt3\log\tau}\) in the definition of the response
coefficients.

For \(\xi\), rewrite \eqref{eq:xi-exact-global} as
\begin{align}
\ddot\xi+(\tau\mathcal H-1)\dot\xi
+2\tau^2e^{-2m}\xi
&=f_\rho(s),\label{eq:xi-log-global}\\
f_\rho(s)
&=\tau^2e^{-2m}
\left[e^{-2\xi}-1+2\xi-(\cosh \zeta-1)\right].
\end{align}
We now relate the decay obtained from the normalized phase space to the
logarithmic physical variable \(s=\log\tau\).  Proposition~\ref{prop:spectra}
shows that the constrained eigenvalues at \(P_7\) have real part
\(-3/7\).  Hence hyperbolic-sink estimates give, for every sufficiently
small \(\delta>0\),
\[
U_\rho(\eta)-P_7
=
O\!\left(e^{-(3/7-\delta)\eta}\right).
\]
The reconstruction in Theorem~\ref{thm:global-AC} gives
\(\tau\mathcal H\to7\), and therefore
\[
\frac{d\eta}{ds}=\tau\mathcal H\to7.
\]
After slightly enlarging \(\delta\), these two facts imply
\[
\zeta(s),\xi(s)=O(e^{-(3-\delta)s}).
\]
Since the bracket in the definition of \(f_\rho\) is quadratic in
\(\zeta,\xi\) at the origin and \(\tau^2e^{-2m}\) stays bounded, it follows that
\[
f_\rho(s)=O(e^{-(6-2\delta)s}).
\]
Choose \(\delta<3/2\).  Then
\[
e^{3s}|f_\rho(s)|
=
O(e^{-(3-2\delta)s}),
\]
and hence
\[
\int_{s_0}^{\infty}e^{3s}|f_\rho(s)|\,ds<\infty.
\]
This is the precise meaning of saying that the forcing is integrable after
the response weight \(e^{3s}\) is applied.  The weighted integrability is
what allows the renormalized complex amplitude to converge.

\begin{remark}[Complex-amplitude convention]
For a real solution of the limiting equation
\[
u_{ss}+6u_s+12u=0
\]
we use the convention
\[
u(s)=\Re\!\left(Ze^{(-3+i\sqrt3)s}\right).
\]
The complex number \(Z\) packages the two real constants of this
two-dimensional oscillatory solution: \(|Z|\) is its leading amplitude and
\(\arg Z\) its phase.  It can be recovered directly from \(u,u_s\) by
\[
Z
=
\lim_{s\to\infty}
e^{(3-i\sqrt3)s}
\left[
u(s)-\frac{i}{\sqrt3}\bigl(3u(s)+u_s(s)\bigr)
\right].
\]
We use this same convention for \(Q(\rho)\) and \(X(\rho)\), so their
relative phase is unambiguous.
\end{remark}

\begin{theorem}[Global Ricci-flat response theorem]
\label{thm:global-response}
For every \(\rho>0\) there are unique complex numbers
\(Q(\rho),X(\rho)\) such that
\begin{align}
\zeta_\rho(\tau)
&=
\tau^{-3}\Re\!\left(Q(\rho)\tau^{i\sqrt3}\right)
+o(\tau^{-3}),\\
\xi_\rho(\tau)
&=
\tau^{-3}\Re\!\left(X(\rho)\tau^{i\sqrt3}\right)
+o(\tau^{-3}),
\end{align}
and the corresponding differentiated expansions hold.  The map
\[
\rho\longmapsto(Q(\rho),X(\rho))
\]
is \(C^1\), indeed smooth, on \((0,\infty)\).
\end{theorem}

\begin{proof}
For either anisotropy channel, write the logarithmic-time first-order
system as
\[
U_s=(A_\infty+E_\rho(s))U+N_\rho(s),
\]
where \(A_\infty\) has eigenvalues \(-3\pm i\sqrt3\).
Theorem~\ref{thm:global-AC} and hyperbolicity of \(P_7\) give
\[
\|E_\rho(s)\|=O(e^{-\gamma s})
\]
for some \(\gamma>0\), uniformly on compact \(\rho\)-intervals.  For the
\(\zeta\)-channel one has \(N_\rho=0\).  For the \(\xi\)-channel,
\eqref{eq:xi-log-global} and the preliminary hyperbolic decay imply
\[
N_\rho(s)=O(e^{-(6-2\delta)s})
\]
for every sufficiently small \(\delta>0\).

Set
\[
W(s)=e^{-A_\infty(s-s_0)}U(s).
\]
Variation of constants gives the exact identity
\[
W(s)
=
U(s_0)
+
\int_{s_0}^{s}
e^{-A_\infty(\sigma-s_0)}
\bigl(E_\rho(\sigma)U(\sigma)+N_\rho(\sigma)\bigr)\,d\sigma.
\]
Since \(e^{-A_\infty\sigma}=O(e^{3\sigma})\), while
\(U(\sigma)=O(e^{-(3-\delta)\sigma})\), the exponential convergence of
\(E_\rho\) makes the first integrand absolutely integrable after choosing
\(\delta>0\) sufficiently small.  The second is integrable by the weighted
estimate for \(N_\rho\) established above.  Hence the integral converges as
\(s\to\infty\), so \(W(s)\) has a finite limit.  Resolving that limiting
real two-vector into the conjugate eigendirections of \(A_\infty\) gives
the unique complex amplitude.  With the convention of the preceding
remark these amplitudes are exactly \(Q(\rho)\) and \(X(\rho)\).
Substitution back into the first-order system gives the differentiated
asymptotic expansions.

It remains to justify the parameter dependence of the response
coefficients.  Fix a compact interval \(J\Subset(0,\infty)\), and use the
common cone-entry time \(\eta_J\) from
Corollary~\ref{cor:chi-uniform-entry}.  Smooth dependence of the finite-time
flow on the singular-orbit parameter implies that
\[
\sup_{\rho\in J}
\left|\partial_\rho U_\rho(\eta_J)\right|<\infty .
\]
Because the normalized vector field itself is independent of \(\rho\), the
parameter derivative
\[
v_\rho(\eta):=\partial_\rho U_\rho(\eta)
\]
satisfies precisely the constrained variational equation
\eqref{eq:constrained-variational}.  The uniform estimate in
Corollary~\ref{cor:chi-uniform-entry} therefore gives
\[
|v_\rho(\eta)|
\le C_J e^{-\gamma_J(\eta-\eta_J)}
|v_\rho(\eta_J)|
\qquad
(\eta\ge\eta_J,\ \rho\in J).
\]
Thus the first \(\rho\)-derivatives of the normalized state converge
exponentially to zero, uniformly for \(\rho\in J\).

The coefficient perturbation \(E_\rho\) and, in the \(\xi\)-channel, the
forcing \(N_\rho\), are smooth functions of the normalized state and of the
anisotropy variables appearing above.  Combining the preceding estimate
with the already established decay of those variables gives, after possibly
decreasing the decay exponent,
\[
\|\partial_\rho E_\rho(s)\|
\le C_J e^{-\gamma'_J s},
\qquad
e^{3s}|\partial_\rho N_\rho(s)|
\le C_J e^{-\gamma'_J s}
\]
for \(s\) sufficiently large and some \(\gamma'_J>0\), uniformly on \(J\).
The same argument applied to the differentiated variational equations gives
the corresponding bounds for every higher \(\rho\)-derivative on compact
parameter intervals.  Consequently the variation-of-constants integrands,
together with all their \(\rho\)-derivatives, admit integrable majorants
independent of \(\rho\in J\).  Differentiation under the integral sign is
therefore justified.  Hence \(Q\) and \(X\) are \(C^1\), and in fact smooth,
on \((0,\infty)\).
\end{proof}

\begin{lemma}[Asymptotic coordinates for the homogeneous channel]
\label{lem:asymptotic-coordinates}
Let
\[
Y_s=(A_\infty+E(s))Y,\qquad s\ge s_0,
\]
where \(E\in L^1([s_0,\infty))\) and \(A_\infty\) is a real
\(2\times2\) matrix with eigenvalues \(-3\pm i\sqrt3\).  Then every real
solution has a unique complex amplitude \(Z\) such that
\[
Y(s)=e^{A_\infty(s-s_0)}\bigl(v_Z+o(1)\bigr),
\]
where \(v_Z\in\mathbb R^2\) is the real vector corresponding to \(Z\) in
the complex eigencoordinate of \(A_\infty\).  Moreover, the map
\[
Y(s_0)\longmapsto (\Re Z,\Im Z)
\]
is a real-linear isomorphism.  In particular, \(Z=0\) if and only if
\(Y\equiv0\).
\end{lemma}

\begin{proof}
Let \(\Phi(s)\) be the fundamental matrix normalized by
\(\Phi(s_0)=I\), and set
\[
R(s)=e^{-A_\infty(s-s_0)}\Phi(s).
\]
Then
\[
R_s=e^{-A_\infty(s-s_0)}E(s)e^{A_\infty(s-s_0)}R(s).
\]
Because the two eigenvalues of \(A_\infty\) have the same real part and
are distinct, \(A_\infty\) is similar over \(\mathbb R\) to
\(-3I+\sqrt3 J\), where \(J\) generates rotations.  Hence
\[
\bigl\|e^{-A_\infty u}E(s)e^{A_\infty u}\bigr\|
\le C\|E(s)\|
\]
for \(u=s-s_0\ge0\).  The coefficient in the equation for \(R\) is
therefore integrable.  Gronwall's inequality gives a uniform bound for
\(R\), and the integral equation then shows that \(R(s)\) converges to a
finite matrix \(R_\infty\).  To see directly that this limit is invertible,
write
\[
B(s)=e^{-A_\infty(s-s_0)}E(s)e^{A_\infty(s-s_0)}.
\]
Then \(R_s=B(s)R\), \(R(s_0)=I\), and Liouville's formula gives
\[
\det R(s)
=
\exp\!\left(\int_{s_0}^{s}\operatorname{tr}B(\sigma)\,d\sigma\right).
\]
Since \(B\in L^1\), the integral converges and therefore
\[
\det R_\infty
=
\exp\!\left(\int_{s_0}^{\infty}\operatorname{tr}B(\sigma)\,d\sigma\right)
\neq0.
\]
Hence \(R_\infty\) is invertible.  Thus
\[
Y(s)=e^{A_\infty(s-s_0)}
\bigl(R_\infty Y(s_0)+o(1)\bigr).
\]
Resolving the real vector \(R_\infty Y(s_0)\) in the complex
eigencoordinate gives the unique amplitude \(Z\).  Since both
\(R_\infty\) and the real identification with that eigencoordinate are
invertible, the Cauchy-data-to-amplitude map is an isomorphism.
\end{proof}

\begin{theorem}[Nonvanishing of the \(\zeta\)-response amplitude]
\label{thm:Q-nonzero}
For every \(\rho>0\),
\[
Q(\rho)\neq0.
\]
Consequently, for every compact \(J\Subset(0,\infty)\),
\[
\inf_{\rho\in J}|Q(\rho)|>0.
\]
The same conclusion applies to the limiting heteroclinic
\(\Gamma_{57}\): its \(\zeta\)-response coefficient \(Q_0\) is nonzero.
\end{theorem}

\begin{proof}
For every singular-orbit trajectory, \(\zeta\) is not identically zero: near the singular orbit,
\(A_\rho(\tau)\sim\tau\) while \(B_\rho(\tau)\to1\), so
\(\zeta(\tau)\to-\infty\).  Fix any positive logarithmic time \(s_0\).  Along
the given solution, \eqref{eq:q-linear-global} is a homogeneous
first-order \(2\times2\) system for \(Y=(\zeta,\zeta_s)\) whose coefficient matrix
converges to the matrix with characteristic roots \(-3\pm i\sqrt3\).
The exponential convergence to \(P_7\) implies that the coefficient
perturbation belongs to \(L^1([s_0,\infty))\).  Lemma~\ref{lem:asymptotic-coordinates}
therefore applies.  In the complex eigencoordinate normalized by
\[
\zeta(s)=\Re\!\left(Ze^{(-3+i\sqrt3)s}\right)+o(e^{-3s}),
\]
the amplitude \(Z\) furnished by that lemma is exactly the response
coefficient \(Q(\rho)\) defined above.  Indeed, both are uniquely
characterized by the same asymptotic convention.  Hence, if
\(Q(\rho)=0\), the asymptotic-coordinate lemma forces
\[
\zeta(s_0)=\zeta_s(s_0)=0.
\]
Along the fixed geometric trajectory, \eqref{eq:q-linear-global} is a
homogeneous linear equation for \(\zeta\), with the coefficient
\[
2\tau^2e^{-2m}\frac{\sinh q}{q}
\]
regarded as the already determined continuous coefficient function of
\(s\).  Uniqueness for this linear Cauchy problem therefore gives
\(\zeta\equiv0\), contradicting the singular-orbit behavior.  Hence
\(Q(\rho)\neq0\).

Continuity of \(Q\) and compactness give the uniform positive lower bound.
The limiting heteroclinic \(\Gamma_{57}\) has a nontrivial \(\zeta\)-component:
if \(\zeta\equiv0\), then \(A\equiv B\), which is incompatible with the fact
that the orbit leaves \(P_5\) in the \(Y_B\)-direction while \(Y_B(P_5)=0\).
Its \(\zeta\)-equation has the same \(L^1\) convergence at \(P_7\).  The same
asymptotic-coordinate lemma, with the same amplitude convention, therefore
gives \(Q_0\neq0\).
\end{proof}

\section{Exact outer transfer theorem}\label{sec:outer-transfer}

The inner problem has produced the two response coefficients
\(Q(\rho)\) and \(X(\rho)\), which describe the leading oscillatory
anisotropies in the large-\(\tau\) inner asymptotics.  We now determine how one
such leading oscillatory mode propagates through the outer region to the
Einstein matching section.

At leading order, the propagation of both anisotropies is obtained by
linearizing their equations about the positive-Einstein sine-cone
reference solution
\[
a_0(t)=b_0(t)=\frac{1}{\sqrt6}\sin t,
\qquad
c_0(t)=\frac{1}{\sqrt3}\sin t.
\]
Indeed, the link metric
\[
\frac16 g_{S^2}+\frac16 g_{S^2}+\frac13 g_{S^3}
\]
is Einstein with Einstein constant \(6\), so this warped sine metric
satisfies \(\Ric=7g\).  It is not a round metric; ``sine-cone'' refers to
the fact that the Ricci-flat cone proportions are multiplied by
\(\sin t\).
Both linearizations give the same scalar Jacobi equation.  We select its
complex solution \(\Psi\) by the normalization
\[
\Psi(t)=t^{-3+i\sqrt3}(1+O(t^2))
\qquad(t\downarrow0).
\]
This normalization fixes the multiplicative ambiguity in the complex
solution and is the one compatible with the normalization used to define
the inner response coefficients.
For the sine-cone reference solution, the mean curvature vanishes at the
midpoint \(t=\pi/2\).  The exact solutions considered later have their
\(H=0\) event at a nearby, parameter-dependent point; Section~\ref{sec:weighted-matching}
controls this displacement.  The present linear transfer calculation is
therefore made at the sine-cone reference point \(t=\pi/2\).  The two numbers
\[
L_0=\Psi(\pi/2),\qquad L_1=\Psi'(\pi/2)
\]
measure, respectively, how a unit inner amplitude contributes to the value
and derivative of the leading linearized mode at this reference point.
We call \(L_0,L_1\) the \emph{outer transfer coefficients}.  Once they are
known, the inner amplitudes \(Q\) and \(X\) can be converted into the two
leading matching responses without solving the full nonlinear outer
problem.

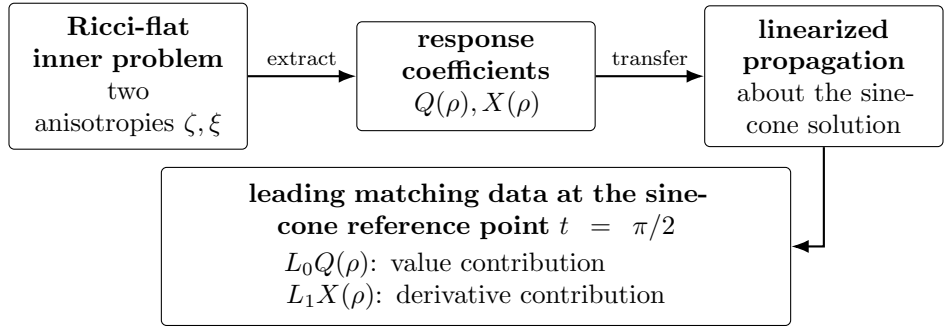
\begin{figure}[t]
\centering
\begin{tikzpicture}[>=Latex,
  mainbox/.style={draw,rounded corners=2pt,align=center,text width=2.8cm,
                 minimum height=1.35cm,inner sep=5pt,font=\small},
  resultbox/.style={draw,rounded corners=2pt,align=center,text width=7.9cm,
                   minimum height=1.25cm,inner sep=6pt,font=\small},
  arr/.style={->,thick}]
\node[mainbox] (inner) at (0,0)
 {\textbf{Ricci-flat inner problem}\\
  two anisotropies \(\zeta,\xi\)};
\node[mainbox] (data) at (4.6,0)
 {\textbf{response coefficients}\\
  \(Q(\rho),X(\rho)\)};
\node[mainbox] (sinecone) at (9.2,0)
 {\textbf{linearized propagation}\\
  about the sine-cone solution};

\draw[arr] (inner) -- node[above,font=\scriptsize] {extract} (data);
\draw[arr] (data) -- node[above,font=\scriptsize] {transfer} (sinecone);

\node[resultbox] (matching) at (4.6,-2.25)
 {\textbf{leading matching data at the sine-cone reference point \(t=\pi/2\)}\\[2pt]
  \(L_0Q(\rho)\): value contribution \(\qquad\)
  \(L_1X(\rho)\): derivative contribution};

\draw[arr] (sinecone.south) |- (matching.east);
\end{tikzpicture}
\caption{The transfer step.  The Ricci-flat inner analysis replaces the
two full anisotropy functions by the two complex coefficients \(Q\) and
\(X\).  The linearized Einstein equation about the sine-cone solution then
transfers these coefficients to the reference point \(t=\pi/2\):
\(L_0Q\) gives the leading value contribution and \(L_1X\) the leading
derivative contribution.  The common logarithmic phase and the displacement
from \(t=\pi/2\) to the exact \(H=0\) event are handled later.}
\label{fig:response-transfer}
\end{figure}

Thus the Jacobi equation governing the leading linearized propagation of
either anisotropy about the sine-cone reference solution is
\begin{equation}\label{eq:outer-Jacobi}
Y''+7\cot t\,Y'+12\csc^2t\,Y=0,
\qquad 0<t<\pi.
\end{equation}
Put
\[
\omega=\sqrt3.
\]

\begin{lemma}[Compatibility of the inner and outer amplitude conventions]
\label{lem:inner-outer-amplitude-compatibility}
Let
\[
\vartheta(\eps)=\sqrt3\log(1/\eps).
\]
Suppose an inner anisotropy has response coefficient \(Z\) in the convention
\[
Y_{\rm in}(\tau)
=
\tau^{-3}\Re\!\left(Z\tau^{i\sqrt3}\right)
+o(\tau^{-3}).
\]
With \(t=\eps\tau\), its leading expression on an overlap region is
\[
Y_{\rm in}(t/\eps)
=
\eps^3
\Re\!\left(
e^{i\vartheta(\eps)}Z\,t^{-3+i\sqrt3}
\right)
+o(\eps^3t^{-3}).
\]
Consequently, if
\[
\Psi(t)=t^{-3+i\sqrt3}(1+O(t^2))
\]
is the normalized outer Frobenius solution, then the coefficient of
\(\Psi\) is exactly
\[
e^{i\vartheta(\eps)}Z.
\]
There is no additional conjugation, sign, phase shift, or real scale
factor.  Thus the two inner responses are transferred as
\[
Q\longmapsto e^{i\vartheta(\eps)}L_0Q,
\qquad
X\longmapsto e^{i\vartheta(\eps)}L_1X.
\]
\end{lemma}

\begin{proof}
Using \(\tau=t/\eps\),
\begin{align*}
\tau^{-3}Z\tau^{i\sqrt3}
&=
\left(\frac{\eps}{t}\right)^3
Z\exp\!\left(i\sqrt3\log\frac{t}{\eps}\right)\\
&=
\eps^3
e^{i\sqrt3\log(1/\eps)}
Z\,t^{-3+i\sqrt3}.
\end{align*}
Taking real parts gives the first assertion.  The outer solution
\(\Psi\) has leading coefficient \(1\) in precisely the same complex
Frobenius convention, so matching the leading singular term identifies
its complex coefficient with \(e^{i\vartheta(\eps)}Z\).  Since this
calculation is an identity in the chosen complex convention, neither
complex conjugation nor any additional multiplicative constant is
available.  Evaluation at \(t=\pi/2\) gives the factors \(L_0\) and
\(L_1\).
\end{proof}

\begin{theorem}[Outer Wronskian transfer]
\label{thm:outer-transfer}
Let \(\Psi\) be the complex solution of \eqref{eq:outer-Jacobi} normalized by
\[
\Psi(t)=t^{-3+i\sqrt3}(1+O(t^2))
\qquad(t\downarrow0),
\]
and set
\[
L_0=\Psi(\pi/2),\qquad L_1=\Psi'(\pi/2).
\]
Then
\begin{equation}\label{eq:outer-Wronskian}
W[\Psi,\overline\Psi]
=
\Psi\overline{\Psi'}-\Psi'\overline\Psi
=-2i\sqrt3\,\sin^{-7}t.
\end{equation}
In particular,
\begin{equation}\label{eq:outer-transfer-sign}
\Im(L_0\overline{L_1})=-\sqrt3<0.
\end{equation}
\end{theorem}

\begin{proof}
For \eqref{eq:outer-Jacobi}, Abel's identity gives
\[
W'=-7\cot t\,W,
\]
so
\[
W=C\sin^{-7}t
\]
for some constant \(C\).  The Frobenius normalization at the singular end
fixes this constant.  Indeed,
\[
\Psi(t)=t^r(1+O(t^2)),
\qquad
r=-3+i\sqrt3,
\]
and hence, for real \(t\),
\begin{align*}
W
&=\Psi\overline{\Psi'}-\Psi'\overline\Psi\\
&=(\overline r-r)t^{r+\overline r-1}(1+O(t^2))\\
&=-2i\sqrt3\,t^{-7}(1+O(t^2))
\qquad(t\downarrow0).
\end{align*}
Since \(\sin^{-7}t\sim t^{-7}\), this gives
\eqref{eq:outer-Wronskian}.  Evaluating at \(t=\pi/2\) yields
\[
L_0\overline{L_1}-L_1\overline{L_0}=-2i\sqrt3,
\]
which is exactly \eqref{eq:outer-transfer-sign}.
\end{proof}

The relative response used in the matching argument is
\begin{equation}\label{eq:Phi-definition}
\Phi(\rho)
=
\Im\!\left(L_0Q(\rho)\overline{L_1X(\rho)}\right).
\end{equation}
For the symmetric member \(\rho=\sqrt2\), the invariant relation
\(C=\sqrt2 B\) gives
\[
X(\sqrt2)=-\frac12Q(\sqrt2).
\]
Therefore \eqref{eq:outer-transfer-sign} gives
\begin{equation}\label{eq:Phi-anchor}
\Phi(\sqrt2)=\frac{\sqrt3}{2}|Q(\sqrt2)|^2>0.
\end{equation}

\section{The small-ratio heteroclinic passage}\label{sec:small-ratio}
\label{sec:small-ratio-heteroclinic}

We now study what happens to the normalized Ricci-flat trajectories
\(U_\rho\) when the singular-orbit ratio \(\rho\) tends to zero.  For every fixed
\(\rho>0\), the trajectory starts near \(P_0\) and eventually converges to
the final cone point \(P_7\).  When \(\rho\) is very small, however, it
first passes close to the intermediate equilibrium \(P_5\) and remains
near it for a long normalized time.

The limiting motion therefore occurs in two stages.  First, there is a
trajectory
\[
\Gamma_{05}:P_0\longrightarrow P_5.
\]
It lies in the invariant set \(Y_B=0\).  Second, a trajectory leaves
\(P_5\) in the \(Y_B\)-direction introduced in
Proposition~\ref{prop:spectra} and converges to \(P_7\):
\[
\Gamma_{57}:P_5\longrightarrow P_7.
\]
In dynamical-systems terminology, two connecting trajectories of this form
are called a \emph{heteroclinic chain}.  Thus the chain relevant here is
simply
\[
P_0\xrightarrow{\ \Gamma_{05}\ }P_5
   \xrightarrow{\ \Gamma_{57}\ }P_7.
\]
For small positive \(\rho\), the actual trajectory \(U_\rho\) follows the
first connection, passes close to \(P_5\), and then follows the second.

\begin{figure}[ht]
\centering
\begin{tikzpicture}[>=Latex,
  every node/.style={font=\small},
  traj/.style={->,thick},
  nearby/.style={->,thick,dashed}]
  \fill (0,0) circle (2pt);
  \fill (3.7,0) circle (2pt);
  \fill (8.0,2.2) circle (2pt);

  \node[below=5pt] at (0,0) {\(P_0\)};
  \node[below=5pt] at (3.7,0) {\(P_5\)};
  \node[right=6pt] at (8.0,2.2) {\(P_7\)};

  \draw[traj] (0.20,0) -- (3.48,0);
  \draw[traj] (3.90,0.10)
      .. controls (5.1,0.35) and (6.6,1.45) .. (7.80,2.10);

  \node at (1.85,-0.55) {\(\Gamma_{05}\)};
  \node at (6.10,0.75) {\(\Gamma_{57}\)};

  \draw[nearby] (0.18,1.15)
      .. controls (1.7,1.18) and (2.9,1.18) .. (3.95,1.20)
      .. controls (5.1,1.28) and (6.55,1.78) .. (7.68,2.12);

  \node at (2.00,1.70) {small-\(\rho\) trajectory \(U_\rho\)};

  \draw[->] (3.7,0.18) -- (3.7,0.78);
  \node[left=4pt] at (3.7,0.62) {\(+Y_B\)};
\end{tikzpicture}
\caption{Schematic small-ratio passage in normalized phase space.
The limiting trajectory \(\Gamma_{05}\) runs from \(P_0\) to \(P_5\)
inside the invariant set \(Y_B=0\).  The second limiting trajectory
\(\Gamma_{57}\) leaves \(P_5\) in the positive \(Y_B\)-direction and
converges to \(P_7\).  For small \(\rho>0\), the trajectory \(U_\rho\)
first stays close to \(\Gamma_{05}\) and then to \(\Gamma_{57}\).
The invariant set \(Y_B=0\) itself is not drawn, in order to keep the
schematic uncluttered.}
\label{fig:heteroclinic-chain}
\end{figure}
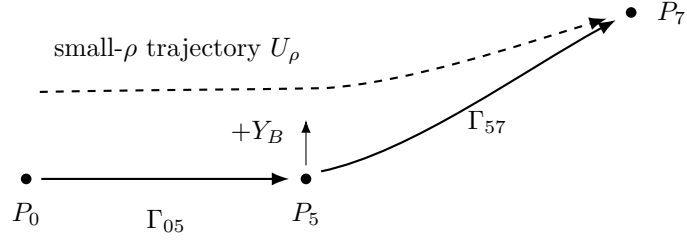

The purpose of this section is to make this limiting picture precise and
then determine its consequence for the relative response
\[
\Phi(\rho)
=
\Im\!\left(L_0Q(\rho)\overline{L_1X(\rho)}\right)
\]
introduced in Section~\ref{sec:outer-transfer}.  In particular, we will
prove the sign of \(\Phi(\rho)\) for all sufficiently small positive
\(\rho\).  This is the sign information needed later in the matching
argument.

\subsection{The heteroclinic chain}

\begin{lemma}[The singular-orbit ratio as a coordinate near \(P_0\)]
\label{lem:rho-projective}
For every normalized inner trajectory,
\[
\frac{Y_B}{Y_C}=\frac{C}{B},
\]
and therefore
\[
\rho=\lim_{\eta\to-\infty}\frac{Y_B}{Y_C}.
\]
The trajectories that converge to \(P_0\) as \(\eta\to-\infty\) form a
two-dimensional local family, the unstable manifold \(W^u(P_0)\).
Changing the origin of normalized time multiplies the two leading
coefficients by the same positive factor.  Their ratio is therefore
unchanged, and for the physical inner family this ratio is precisely
\(\rho\).  Thus \(\rho\) distinguishes the trajectories after the
irrelevant choice of time origin has been removed.
\end{lemma}

\begin{proof}
The first identity follows directly from the definitions.
The singular-orbit data give \(C/B\to\rho\).  Proposition~\ref{prop:spectra}
gives \(\dim W^u(P_0)=2\).  The analytic unstable-manifold theorem gives
local unstable coordinates \((u_B,u_C)\) in which the two leading modes have
the common exponent \(1/2\).  Translation
\(\eta\mapsto\eta+\eta_*\) multiplies both leading coefficients by
\(e^{\eta_*/2}\), so their ratio is a projective coordinate.  The singular-orbit
expansion and \(Y_B/Y_C=C/B\) identify this ratio with \(\rho\).
Consequently the physical singular-orbit family is a smooth projective curve in
\(W^u(P_0)\), and \(\rho=0\) is its limiting boundary direction rather than
an additional physical singular orbit.
\end{proof}

\begin{lemma}[Projective transversality at the boundary direction]
\label{lem:projective-transversality}
Choose a smooth local section \(\Sigma_0\subset W^u(P_0)\) transverse to
the normalized flow and meeting the positive branch corresponding to
\(\rho=0\).  Then the physical orbit family extends to a \(C^1\) curve
\[
z:[0,\rho_0)\longrightarrow \Sigma_0,
\qquad
z(0)=\Gamma_{05}\cap\Sigma_0,
\]
and this curve is transverse at \(\rho=0\) to
\(\Sigma_0\cap\{Y_B=0\}\).  Equivalently, \(Y_B(z(\rho))\) has a
nonzero first-order term:
\[
Y_B(z(\rho))=c_0\rho+o(\rho),
\qquad c_0>0.
\]
No positive-Einstein or singular-orbit germ is differentiated at
\(\rho=0\).
\end{lemma}

\begin{proof}
Near \(P_0\), the unstable manifold is two-dimensional and both unstable
eigenvalues are \(1/2\).  The unstable-manifold parametrization may
therefore be chosen in asymptotic coordinates \((a_B,a_C)\), where
\[
Y_B(\eta)=a_Be^{\eta/2}+o(e^{\eta/2}),
\qquad
Y_C(\eta)=a_Ce^{\eta/2}+o(e^{\eta/2})
\qquad(\eta\to-\infty).
\]
Translation of \(\eta\) multiplies \((a_B,a_C)\) by the same positive
factor.  Hence the quotient \(a_B/a_C\) is the projective coordinate on
the orbit space of \(W^u(P_0)\setminus\{P_0\}\).

For the physical singular-orbit family,
Lemma~\ref{lem:rho-projective} identifies this quotient exactly with the
singular-orbit ratio:
\[
\frac{a_B}{a_C}=\rho.
\]
Choose \(\Sigma_0\) by fixing any smooth positive homogeneous normalization
of \((a_B,a_C)\), for example \(a_C=a_*>0\) with \(a_*\) fixed sufficiently
small in the unstable chart.  In this section the orbit family is represented by
\[
(a_B,a_C)=(\rho a_*,a_*),
\]
followed by the smooth unstable-manifold parametrization and the smooth
finite flow to \(\Sigma_0\).  It therefore extends \(C^1\) to
\(\rho=0\), and its derivative there has a nonzero \(a_B\)-component.

Since \(Y_B=0\) is the invariant face containing the limiting branch
\(\Gamma_{05}\), its intersection with \(\Sigma_0\) has tangent space
given by vanishing \(a_B\)-component.  Thus the extended curve is
transverse to that face.  Positivity of \(Y_B\) for \(\rho>0\) fixes the
sign of the transverse coefficient and gives \(c_0>0\).
\end{proof}

\begin{lemma}[Trajectories approaching \(P_5\) inside \(Y_B=0\)]
\label{lem:P5-stable-face}
Let \(\mathcal C\) denote the smooth four-dimensional constraint manifold
defined near \(P_5\) by \eqref{eq:linearconstraint} and
\eqref{eq:quadraticconstraint}, without imposing the sign inequalities that
define \(\mathcal P\).  Near \(P_5\), the points whose forward trajectories
converge to \(P_5\) are exactly the points of
\[
\mathcal C\cap\{Y_B=0\}.
\]
In standard dynamical-systems terminology, this set is the local stable
manifold \(W^s_{\rm loc}(P_5)\).  Its positive part is
\(W^s_{\rm loc}(P_5)\cap\mathcal P\).
\end{lemma}

\begin{proof}
The two constraint gradients are independent at \(P_5\), so \(\mathcal C\)
is a smooth four-manifold there.  The hypersurface
\(\mathcal C\cap\{Y_B=0\}\) is invariant because
\(\dot Y_B=Y_B(G-X_B)\).  It is three-dimensional near \(P_5\), and
Proposition~\ref{prop:spectra} shows that its tangent space at \(P_5\) is
exactly the three-dimensional stable eigenspace: the unique complementary
direction is the \(Y_B\)-eigendirection with eigenvalue \(1/5\).  The
standard stable-manifold theorem and uniqueness therefore identify this
invariant hypersurface locally with \(W^s_{\rm loc}(P_5)\).
\end{proof}

\begin{lemma}[The limiting connection \(P_0\to P_5\)]
\label{lem:05}
Among the trajectories that converge to \(P_0\) as
\(\eta\to-\infty\), the limiting direction corresponding to
\(\rho=0\) has a positive trajectory that converges to \(P_5\) as
\(\eta\to+\infty\).  We denote this connecting trajectory by
\[
\Gamma_{05}:P_0\longrightarrow P_5.
\]
Equivalently, \(\Gamma_{05}\) is the relevant positive branch of
\(W^u(P_0)\).
\end{lemma}

\begin{proof}
The face
\[
\mathcal F_5=\mathcal P\cap\{X_B=Y_B=0\}
\]
is invariant.  At \(P_0\), its tangent space contains exactly one of the
two unstable eigendirections from Proposition~\ref{prop:spectra}.
The unstable-manifold theorem therefore gives a one-dimensional local
unstable branch of \(P_0\) contained in \(\mathcal F_5\); we choose the
branch entering the positive part of \(\mathcal F_5\).  This is the
trajectory corresponding to the limiting projective direction
\(\rho=0\).

On \(\mathcal F_5\),
\[
2X_A+3X_C=1.
\]
For
\[
V_5=Y_A^2Y_C^3
\]
one obtains
\[
\frac{\dot V_5}{V_5}=5G-1.
\]
Weighted Cauchy--Schwarz gives
\[
G=2X_A^2+3X_C^2
\ge\frac{(2X_A+3X_C)^2}{5}
=\frac15.
\]
Thus \(V_5\) is nondecreasing.  Choose a finite point
\(\eta_0\) on the positive unstable branch.  Since
\(V_5(\eta_0)>0\), the compact superlevel set
\[
\mathcal K^5_{\eta_0}
=
\{U\in\mathcal F_5:V_5(U)\ge V_5(U(\eta_0))\}
\]
is positively invariant and lies entirely in \(Y_A Y_C>0\).  On this set
the identity for \(\dot V_5/V_5\) is legitimate.  Equality occurs only at
\(X_A=X_C=1/5\); invariance then forces
\(Y_A=2/5\), \(Y_C=\sqrt2/5\).  Hence \(P_5\) is the maximal invariant
subset of the zero set of \(\dot V_5\) contained in
\(\mathcal K^5_{\eta_0}\).  LaSalle's invariance principle on this compact
positively invariant set gives the claim.
\end{proof}

\begin{lemma}[The heteroclinic \(P_5\to P_7\)]
\label{lem:57}
There is a positive heteroclinic
\[
\Gamma_{57}:P_5\longrightarrow P_7.
\]
It lies in the invariant submanifold
\[
\mathcal M=
\left\{X_C=X_A,\quad Y_C=\frac{Y_A}{\sqrt2}\right\},
\]
which in physical variables is exactly \(C=\sqrt2 A\).
\end{lemma}

\begin{proof}
A direct substitution into \eqref{eq:normalized} shows that
\(\mathcal M\) is invariant.  The unstable eigenspace of \(P_5\) is
the \(Y_B\)-direction and is tangent to \(\mathcal M\).
By the standard unstable-manifold theorem and uniqueness, the
one-dimensional local unstable manifold satisfies
\(W^u_{\rm loc}(P_5)\subset\mathcal M\).
Its positive branch enters \(Y_A Y_B Y_C>0\).  Since \(\mathcal P\) is
compact and forward invariant, this branch is forward complete, and
Lemma~\ref{lem:positive-to-P7} forces its forward limit to be \(P_7\).
Finally,
\[
\frac{Y_C}{Y_A}=\frac{A}{C},
\]
so \(Y_C=Y_A/\sqrt2\) is equivalent to \(C=\sqrt2 A\).
\end{proof}

\subsection{Saddle passage at \texorpdfstring{\(P_5\)}{P5}}

We now make precise the following simple geometric picture.  When
\(\rho>0\) is small, the trajectory \(U_\rho\) arrives very close to
\(P_5\) and very close to the incoming limiting trajectory
\(\Gamma_{05}\).  It does not converge to \(P_5\), however: its small
positive \(Y_B\)-component eventually grows and carries the trajectory away
from \(P_5\) along the outgoing connection \(\Gamma_{57}\).  We need to
show that, as \(\rho\downarrow0\), the point at which \(U_\rho\) leaves a
fixed small neighborhood of \(P_5\) converges to the corresponding point
of \(\Gamma_{57}\).

Two standard pieces of dynamical-systems terminology occur in this
argument.  First, Lemma~\ref{lem:rho-projective} says that near \(P_0\) the
two leading unstable coefficients are changed by a common factor when the
origin of normalized time is shifted.  Their ratio is therefore the
quantity that distinguishes the trajectories, and for our family this
ratio is \(\rho\).  This is all that is meant below by saying that
\(\rho\) is a \emph{projective unstable coordinate}.

Second, we use the classical \emph{inclination lemma}, also called the
\emph{\(\lambda\)-lemma}.  In the form needed here, it says the following.
Suppose a family of trajectories approaches the stable manifold of a
hyperbolic saddle transversely.  The trajectories then spend longer and
longer near the saddle, and when they leave a fixed small neighborhood,
their outgoing direction approaches the unstable manifold of the saddle.
Thus, in our situation, once we verify that the family \(U_\rho\) approaches
the stable manifold \(Y_B=0\) with a nonzero \(Y_B\)-component, the lemma
implies that its outgoing intersections approach the positive unstable
trajectory \(\Gamma_{57}\).  This standard result is often called the
inclination or \(\lambda\)-lemma; see, for example,
Palis and de Melo~\cite[Chapter~2]{PalisDeMelo1982}.

The next proposition is precisely this statement for the present family.

\begin{proposition}[Small-ratio saddle passage]\label{prop:saddle}
Let \(\Sigma_{\rm in}\) be a smooth section transverse to \(\Gamma_{05}\)
sufficiently close to \(P_5\), and let \(\Sigma_{\rm out}\) be a smooth section transverse to the
outgoing trajectory \(\Gamma_{57}\), also sufficiently close to \(P_5\).
For all sufficiently small \(\rho>0\), the normalized trajectory
\(U_\rho\) meets both sections and
\[
U_\rho\cap\Sigma_{\rm out}
\longrightarrow
\Gamma_{57}\cap\Sigma_{\rm out}
\qquad(\rho\downarrow0).
\]
\end{proposition}

\begin{proof}
By Lemma~\ref{lem:P5-stable-face}, the limiting orbit \(\Gamma_{05}\) lies
in the local stable manifold \(Y_B=0\).  It remains to verify that the singular orbit
family approaches this stable manifold in the transverse unstable direction.
Lemma~\ref{lem:projective-transversality} gives a \(C^1\) curve on a
fixed local section near \(P_0\), extending to the limiting boundary
direction \(\rho=0\), whose transverse \(Y_B\)-component is
\[
Y_B=c_0\rho+o(\rho),
\qquad c_0>0.
\]
Flow this section forward to \(\Sigma_{\rm in}\).  Along the limiting
orbit \(\Gamma_{05}\subset\{Y_B=0\}\), the derivative of the scalar
\(Y_B\)-equation in a transverse direction satisfies
\[
\dot w=(G_0-X_{B,0})w.
\]
Its fundamental solution is
\[
w(\eta)=w(\eta_0)
\exp\!\left(\int_{\eta_0}^{\eta}
(G_0-X_{B,0})\,d\zeta\right),
\]
and is therefore never zero.  The finite-flow map carries the nonzero
transverse direction at \(\Sigma_0\) to a nonzero transverse direction
at \(\Sigma_{\rm in}\).  Hence
\[
Y_{B,\rm in}(\rho)=c_*\rho+o(\rho),
\qquad c_*>0.
\]
Lemma~\ref{lem:P5-stable-face} identifies the local stable manifold
exactly with \(Y_B=0\), so \(Y_B\) is a defining coordinate for the
stable manifold on \(\Sigma_{\rm in}\).
Consequently the family \(U_\rho\) crosses the local stable manifold
\(Y_B=0\) transversely, with \(Y_B>0\) for \(\rho>0\).  We may therefore
apply the inclination lemma in the form described above: the intersections
of these trajectories with the fixed outgoing section converge to the
intersection of that section with the positive local unstable trajectory of
\(P_5\).  Lemma~\ref{lem:57} identifies this outgoing trajectory with
\(\Gamma_{57}\), which proves the claim.
\end{proof}

\begin{remark}[Logarithmic residence time near $P_5$]\label{rem:residence}
The normalized coordinates also give a simple interpretation of the logarithmic residence time near $P_5$.  Near \(P_5\), a smooth unstable coordinate
satisfies
\[
\dot u=\left(\frac15+o(1)\right)u,
\qquad
u_{\rm in}(\rho)=c_*\rho+o(\rho).
\]
To grow from size \(O(\rho)\) to a fixed outgoing size therefore requires
\[
T_\rho=5\log\frac1\rho+O(1).
\]
Thus the logarithmically long interval in the physical proof is precisely
the residence time near the saddle \(P_5\).  No differentiability of the
passage map at \(\rho=0\) is required for the argument below.
\end{remark}

\subsection{Passage to response coordinates}

Let \(\Sigma\) be any fixed transverse section in the basin of \(P_7\)
met by \(\Gamma_{57}\).  The finite flow from \(\Sigma_{\rm out}\) to
\(\Sigma\) is smooth.  Proposition~\ref{prop:saddle} therefore gives
\[
U_\rho\cap\Sigma\longrightarrow\Gamma_{57}\cap\Sigma
\qquad(\rho\downarrow0).
\]
Only this \(C^0\) convergence is used below; no differentiability at the
projective boundary value \(\rho=0\) is required.

We now return to the physical anisotropy variables introduced above,
\[
\zeta=\log\frac AB,\qquad
\xi=\log\frac{C}{\sqrt2(AB)^{1/2}},
\]
and to their stable response coefficients \(Q(\rho),X(\rho)\).
The global response construction itself is not replaced.

\begin{lemma}[Response coordinates near the limiting orbit]
\label{lem:local-response-map}
Let \(\Sigma\) be as above and let
\(z_0=\Gamma_{57}\cap\Sigma\).  After shrinking to a neighborhood
\(\mathcal N\subset\Sigma\) of \(z_0\), every forward orbit from
\(\mathcal N\) remains in the basin of the hyperbolic sink \(P_7\), and the
stable response coefficients define a smooth map
\[
\mathcal S:\mathcal N\longrightarrow\mathbb C^2,
\qquad
z\longmapsto(Q(z),X(z)).
\]
For the physical singular-orbit family this map agrees with
\((Q(\rho),X(\rho))\), and at \(z_0\) it gives the coefficients
\((Q_0,X_0)\) of \(\Gamma_{57}\).  No extension of the parameter map
\(\rho\mapsto(Q(\rho),X(\rho))\) to \(\rho=0\) is assumed: the limiting
coefficients are defined directly by this phase-space response map.
\end{lemma}

\begin{proof}
Choose a later section \(\Sigma_1\) contained in a fixed hyperbolic sink
neighborhood of \(P_7\).  The finite hitting map
\[
\mathcal T:\mathcal N\subset\Sigma\longrightarrow\Sigma_1
\]
is smooth after shrinking \(\mathcal N\), because the reference orbit
meets both sections transversely and the flow depends smoothly on its
initial point.

In a fixed trivialization over the sink neighborhood, each anisotropy
channel has the form
\[
Y_s=(A_\infty+E(s,z))Y,
\]
where \(z\in\Sigma_1\), the matrix \(A_\infty\) has eigenvalues
\(-3\pm i\sqrt3\), and
\[
\|E(\,\cdot\,,z)\|_{L^1([s_1,\infty))}
\]
is uniformly bounded for \(z\) near \(\mathcal T(z_0)\).  Moreover
\(E(s,z)\) and its first \(z\)-derivatives decay exponentially and hence
belong uniformly to \(L^1\).  Applying the variation-of-constants
construction of Lemma~\ref{lem:asymptotic-coordinates} together with one
\(z\)-derivative shows that the Cauchy-data-to-amplitude map
\[
(z,Y(s_1))\longmapsto Z
\]
is smooth.  Applying this to the two channels and composing with the
smooth hitting map \(\mathcal T\) gives the asserted smooth map
\(\mathcal S=(Q,X)\) on \(\mathcal N\).
\end{proof}

\begin{theorem}[Small-ratio response theorem]
\label{thm:smallrho}
For the Ricci-flat singular-orbit family,
\[
\frac{X(\rho)}{Q(\rho)}\longrightarrow\frac12
\qquad(\rho\downarrow0).
\]
Consequently
\[
\Phi(\rho)<0
\]
for all sufficiently small positive \(\rho\).
\end{theorem}

\begin{proof}
By Lemma~\ref{lem:57}, the limiting trajectory satisfies
\(C=\sqrt2A\).  Hence
\[
\xi
=
\log\frac{C}{\sqrt2(AB)^{1/2}}
=
\frac12\log\frac AB
=
\frac12\zeta
\]
identically along \(\Gamma_{57}\).  Both channels use the same logarithmic
time and the same complex-amplitude convention, so linearity of the
amplitude functional gives
\[
X_0=\frac12Q_0.
\]
The limiting trajectory is nontrivial, so the limiting-state part of
Theorem~\ref{thm:Q-nonzero} gives \(Q_0\ne0\).  For all sufficiently small \(\rho\), the section point
\(z_\rho=U_\rho\cap\Sigma\) lies in the fixed neighborhood
\(\mathcal N\) of Lemma~\ref{lem:local-response-map}, and
\[
z_\rho\longrightarrow z_0.
\]
Since the phase-space response map
\(\mathcal S:\mathcal N\to\mathbb C^2\) is continuous,
\[
(Q(\rho),X(\rho))
=\mathcal S(z_\rho)
\longrightarrow
\mathcal S(z_0)
=(Q_0,X_0).
\]
Consequently
\[
\frac{X(\rho)}{Q(\rho)}\to\frac12.
\]
Since \(Q_0\neq0\), the phase-space response convergence also gives
\(|Q(\rho)|\) bounded away from zero for sufficiently small \(\rho\).
By \eqref{eq:Phi-definition},
\[
\Phi(\rho)
=
|Q(\rho)|^2
\Im\!\left(
L_0\overline{L_1}\,
\overline{\frac{X(\rho)}{Q(\rho)}}
\right).
\]
The ratio on the right tends to \(1/2\), while
\(\Im(L_0\overline{L_1})=-\sqrt3\) by
\eqref{eq:outer-transfer-sign}.  Hence
\(\Phi(\rho)<0\) for all sufficiently small positive \(\rho\).
\end{proof}

\begin{remark}[Logical content of the small-ratio theorem]
The argument uses only: transverse approach to \(W^s(P_5)\), the standard
inclination lemma, continuity of the finite flow to a fixed section in the
basin of \(P_7\), and continuity of the phase-space response map there.
In particular, it does not require a differentiable extension of
\(\rho\mapsto(Q(\rho),X(\rho))\) to \(\rho=0\), nor a quantitative rate for
\(X/Q\to1/2\).
\end{remark}

\section{Weighted \texorpdfstring{\(C^1\)}{C1} inner--outer matching}\label{sec:weighted-matching}

So far the inner and outer problems have been studied separately.  We now
have to prove that they describe the same exact positive-Einstein solution
on an intermediate range of scales.  This is the bridge that turns the
limiting response calculation into an existence proof.

There are two problems to keep distinct.  The \emph{limiting response
problem} is obtained by taking the small-singular-orbit limit: the inner geometry is
then exactly Ricci-flat and its amplitudes \(Q,X\) are transported by the
linear outer equation.  The \emph{exact matching problem} concerns the
actual metrics with \(\eps>0\), which remain positive Einstein throughout.
The purpose of this section is to prove uniform \(C^1\) closeness between
these two problems on a common overlap region.  The estimates use standard
variation-of-constants and weighted perturbation arguments; what is
model-specific is the choice of weights and their coupling to the two
response channels and the moving \(H=0\) event.

A small but essential point is that the Ricci-flat inner family approaches
the common cone with a \(\rho\)-dependent radial translation.  Mean-curvature
normalization removes this translation, so it must be reconstructed before
the inner and outer problems are compared.  We do this explicitly below.

Fix a compact ratio interval
\[
I=[\rho_-,\rho_+]\Subset(0,\infty)
\]
on which
\[
\Phi(\rho_-)<0,\qquad \Phi(\rho_+)>0.
\]
For a compact interval \(I\Subset(0,\infty)\), we write
\(O_I(f(\tau))\) for a quantity bounded by \(C_I f(\tau)\) uniformly in
\(\rho\in I\); the same convention is used for the stated
\(\rho\)-derivative estimates.
The physical logarithmic phase associated with a scale \(\eps\) is the
real number
\[
\vartheta(\eps)=\sqrt3\log(1/\eps).
\]
Only its value modulo \(\pi\) is relevant before the parity factor is
separated, but keeping the real-valued lift avoids any ambiguity in the
complex response formulas.  It is not an independent parameter when
\(\eps\) is fixed.  Independent phase variables are introduced only after
restricting to the exact strips
\(\eps_k(\theta)=\exp(-(k\pi+\theta)/\sqrt3)\), for which
\[
e^{i\vartheta(\eps_k(\theta))}
=
(-1)^k e^{i\theta}.
\]

Choose an overlap scale
\[
\delta_\eps=\eps^\beta,\qquad 0<\beta<1,
\]
and write
\[
s_\eps=\log(\delta_\eps/\eps).
\]

\subsection{The reconstructed radial translation}

For a Ricci-flat inner solution put
\[
\mathcal H_\rho
=
2\frac{A_\rho'}{A_\rho}
+2\frac{B_\rho'}{B_\rho}
+3\frac{C_\rho'}{C_\rho}.
\]
Recall that
\[
\mathcal H_\rho'
=
-\mathcal H_\rho^2G_\rho,
\qquad
G_\rho
=
2X_A^2+2X_B^2+3X_C^2,
\]
and the constraint \(2X_A+2X_B+3X_C=1\) implies \(G_\rho\ge1/7\).

\begin{lemma}[Radial translation of the Ricci-flat cone]
\label{lem:radial-shift}
For every \(\rho>0\), the limit
\begin{equation}\label{eq:sigma-def}
\sigma(\rho)
:=
\lim_{\tau\to\infty}
\left(\frac7{\mathcal H_\rho(\tau)}-\tau\right)
\end{equation}
exists and is positive.  On compact intervals \(I\Subset(0,\infty)\),
\(\sigma\) is \(C^1\), and the convergence in \eqref{eq:sigma-def} is
uniform together with one \(\rho\)-derivative.  Moreover
\begin{equation}\label{eq:H-shifted-cone}
\mathcal H_\rho(\tau)
=
\frac7{\tau+\sigma(\rho)}
+O_I(\tau^{-7}),
\end{equation}
and, after fixing the cone normalization used in
Section~\ref{sec:global-response},
\begin{align}
A_\rho(\tau)
&=
\frac{\tau+\sigma(\rho)}{\sqrt6}
+O_I(\tau^{-2}),\label{eq:A-shifted-cone}\\
B_\rho(\tau)
&=
\frac{\tau+\sigma(\rho)}{\sqrt6}
+O_I(\tau^{-2}),\label{eq:B-shifted-cone}\\
C_\rho(\tau)
&=
\frac{\tau+\sigma(\rho)}{\sqrt3}
+O_I(\tau^{-2}),\label{eq:C-shifted-cone}
\end{align}
where the \(O_I(\tau^{-2})\) terms contain the oscillatory anisotropy and
the same bounds hold after one \(\rho\)-derivative.
\end{lemma}

\begin{proof}
Set
\[
S_\rho(\tau)=\frac7{\mathcal H_\rho(\tau)}-\tau.
\]
Then
\[
S_\rho'
=
-\frac{7\mathcal H_\rho'}{\mathcal H_\rho^2}-1
=
7G_\rho-1
\ge0.
\]
Near the singular orbit,
\(\mathcal H_\rho=2/\tau+O(\tau)\), so
\[
S_\rho(\tau)=\frac52\tau+O(\tau^3)>0
\]
for small positive \(\tau\).

By Theorem~\ref{thm:global-AC} and
Corollary~\ref{cor:chi-uniform-entry}, the normalized orbit converges
uniformly exponentially to \(P_7\) on compact \(\rho\)-intervals.  The
normalized time \(\eta\) and the physical logarithmic time
\(s=\log\tau\) satisfy
\[
\frac{ds}{d\eta}=\frac1{\tau\mathcal H}\longrightarrow\frac17.
\]
The constrained eigenvalues at \(P_7\) are
\[
-\frac37\pm\frac{i\sqrt3}{7},
\]
so, after increasing the common entry time if necessary, the two
anisotropy blocks and their first \(\rho\)-derivatives are
\(O_I(e^{-3s})=O_I(\tau^{-3})\).  Write
\[
X_j=\frac17+x_j,\qquad 2x_A+2x_B+3x_C=0.
\]
Then the linear term cancels exactly:
\[
7G-1
=
7\bigl(2x_A^2+2x_B^2+3x_C^2\bigr).
\]
Consequently
\[
7G_\rho-1=O_I(\tau^{-6}),
\qquad
\partial_\rho(7G_\rho-1)=O_I(\tau^{-6}).
\]
Both estimates are integrable in \(d\tau\).  Hence \(S_\rho\) has a finite
limit \(\sigma(\rho)\), positivity follows from monotonicity and the
strictly positive singular-orbit expansion, and
\[
\sigma(\rho)
=
S_\rho(\tau_0)
+\int_{\tau_0}^{\infty}(7G_\rho-1)\,d\tau
\]
may be differentiated under the integral sign.  This proves the asserted
\(C^1\)-dependence and uniform convergence on \(I\).

Since
\[
S_\rho(\tau)-\sigma(\rho)=O_I(\tau^{-5}),
\]
we obtain
\[
\frac7{\mathcal H_\rho}
=
\tau+\sigma(\rho)+O_I(\tau^{-5}),
\]
which gives \eqref{eq:H-shifted-cone}.  Finally,
\[
\frac{A_\rho'}{A_\rho}
=
X_A\mathcal H_\rho,\qquad
\frac{B_\rho'}{B_\rho}
=
X_B\mathcal H_\rho,\qquad
\frac{C_\rho'}{C_\rho}
=
X_C\mathcal H_\rho.
\]
Using
\[
X_A=\frac17+O_I(\tau^{-3}),\quad
X_B=\frac17+O_I(\tau^{-3}),\quad
X_C=\frac17+O_I(\tau^{-3}),
\]
and \eqref{eq:H-shifted-cone}, we obtain, for example,
\[
\frac{d}{d\tau}
\log\frac{A_\rho(\tau)}{\tau+\sigma(\rho)}
=
O_I(\tau^{-4}),
\]
and analogously for \(B_\rho\) and \(C_\rho\).  Integration from infinity,
using the cone ratios of Theorem~\ref{thm:global-AC}, gives
\[
\frac{A_\rho(\tau)}{\tau+\sigma(\rho)}
=
\frac1{\sqrt6}\bigl(1+O_I(\tau^{-3})\bigr),
\]
with the analogous formulas for \(B_\rho,C_\rho\).  This is precisely
\eqref{eq:A-shifted-cone}--\eqref{eq:C-shifted-cone}.  The differentiated
estimates follow from the differentiated \(O_I(\tau^{-3})\) normalized
bounds and the already proved \(C^1\)-control of \(\sigma\).
\end{proof}

For the positive-Einstein problem define the physical shift
\[
\alpha_{\eps,\rho}:=\eps\sigma(\rho)
\]
and the shifted outer coordinate
\begin{equation}\label{eq:shifted-u}
u=t+\alpha_{\eps,\rho}.
\end{equation}
The corresponding translated sine-cone background is
\begin{equation}\label{eq:shifted-sine-background}
r_{\eps,\rho}(t)=\sin u,
\qquad
H_{\sin,\alpha}=7\cot u,
\qquad
K_{\sin,\alpha}=6\csc^2u.
\end{equation}
Notice that
\[
D_\eps\alpha_{\eps,\rho}=\alpha_{\eps,\rho},
\qquad
\partial_\rho\alpha_{\eps,\rho}
=\eps\sigma'(\rho),
\]
so the translation and its first parameter derivatives are uniformly
\(O_I(\eps)\).

\subsection{Inner weighted norm}

Use a smooth gauge-fixed logarithmic-time chart near the Ricci-flat cone,
obtained from the scale-free variables of
Section~\ref{subsec:mean-curvature-normalized} together with the isotropic
reconstruction variable.  After eliminating the Hamiltonian constraint, let
\(U(s)\) denote the resulting first-order state.  In these coordinates the
inner positive-Einstein equations have the form
\begin{equation}\label{eq:innerpert}
\dot U=F_0(U)+\eps^2e^{2s}R(U),
\end{equation}
where \(F_0\) is the Ricci-flat autonomous vector field and \(R\) is
smooth on a common cone neighborhood.  The factor
\(\eps^2e^{2s}\) is the logarithmic-time form of the explicit
\(-7\eps^2(A,B,C)\) perturbation obtained from the inner scaling in
Section~\ref{sec:analytic-setting}.

Let \(P_a\) denote the spectral projection onto the two real anisotropy
planes corresponding to the repeated pair
\[
-3\pm i\sqrt3.
\]
The normalized constrained system has no additional radial eigenmode: radial
translation is lost under mean-curvature normalization and is reconstructed
separately by Lemma~\ref{lem:radial-shift}.  We therefore perform the
inner perturbation estimate on the four-dimensional constrained normalized
state, and only afterwards reconstruct the isotropic radial variable.  For
a perturbation \(V\), define
\[
\|V\|_{a,[s_0,S]}
=
\sup_{s_0\le s\le S} e^{3(s-s_0)}|P_aV(s)|
+
\sup_{s_0\le s\le S}|(I-P_a)V(s)|.
\]

Uniform asymptotic-conical convergence on the compact ratio interval gives a
common cone-entry time \(s_0\) and a common hyperbolic neighborhood.  We fix
the hyperbolic coordinates so that the spectral projections and their first
\(\rho\)-derivatives are uniformly bounded on \(I\).  The nonlinear
remainder is \(C^2\) there, uniformly on thee compact family.

\begin{lemma}[Uniform inner response estimate]\label{lem:innerweighted}
There is \(C_I>0\) such that, for all sufficiently small \(\eps\),
the difference between the positive-Einstein and Ricci-flat anisotropy
response coordinates at \(s=s_\eps\) satisfies
\[
|\Delta Q_\eps(\rho)|+|\Delta X_\eps(\rho)|
\le C_I\delta_\eps^2,
\]
uniformly for \(\rho\in I\).  The same estimate holds after one
\(\rho\)-derivative and after one logarithmic scale derivative
\(D_\eps:=\eps\partial_\eps\), with the latter taken at fixed normalized
singular-orbit ratio \(\rho\).
\end{lemma}

\begin{proof}
Let \(U_0(s;\rho)\) be the Ricci-flat orbit and
\(U_\eps(s;\rho)\) the positive-Einstein orbit with the same normalized singular-orbit
data, and set \(V=U_\eps-U_0\).  On the common cone neighborhood,
\[
 \dot V=A_\rho(s)V+\mathcal N_\rho(V,s)
       +\eps^2e^{2s}R(U_\eps),
\]
where
\[
 |\mathcal N_\rho(V,s)|\le C|V|^2.
\]
In the logarithmic physical time \(s\), both constrained complex blocks at
the cone have exponents \(-3\pm i\sqrt3\).  Thus the entire constrained
normalized state is uniformly exponentially stable once the common cone
neighborhood has been entered.  Corollary~\ref{cor:chi-uniform-entry} and
roughness of exponential stability give constants \(C,\gamma>0\),
independent of \(\rho\in I\), for the corresponding evolution operators.
No estimate here is being assigned to the radial translation; that scalar
is treated separately by Lemma~\ref{lem:radial-shift}.

The isotropic Einstein forcing does not directly create an anisotropy mode
on the cone.  Its anisotropy projection is consequently Lipschitz in the
anisotropy:
\[
 |P_aR(U_\eps)|
 \le C|P_a(U_\eps-U_c)|.
\]
The Ricci-flat asymptotic-conical estimate gives
\[
 |P_a(U_0-U_c)|\le Ce^{-3(s-s_0)}.
\]
Write \(M(S)=\sup_{s_0\le s\le S}|V(s)|\).  Roughness of the
exponential dichotomy and Duhamel's formula give, uniformly for
\(S\le s_\eps\),
\[
M(S)
\le
C\eps^2
+
C\int_{s_0}^{S}e^{-\gamma(S-\sigma)}
\left(M(\sigma)^2+\eps^2e^{2\sigma}\right)d\sigma .
\]
Since
\[
\sup_{S\le s_\eps}
\int_{s_0}^{S}e^{-\gamma(S-\sigma)}\eps^2e^{2\sigma}\,d\sigma
\le C\eps^2e^{2s_\eps}
=C\delta_\eps^2,
\]
a continuity bootstrap closes to
\[
M(s_\eps)\le C\delta_\eps^2.
\]
For the renormalized anisotropy amplitudes, the Einstein forcing carries
the additional anisotropy factor displayed above.  Thus, after multiplication
by the inverse decay \(e^{3s}\), its contribution is bounded by
\[
C\int_{s_0}^{s_\eps}\eps^2e^{2s}\,ds
\le C\delta_\eps^2.
\]
Terms caused by replacing \(U_0\) by \(U_\eps\) are controlled by the
preceding \(M=O(\delta_\eps^2)\) bootstrap and the uniform \(C^2\) bounds in
the cone chart, and give the same or smaller order.  Hence the renormalized
variation-of-constants formula gives
\[
 |\Delta Q_\eps|+|\Delta X_\eps|
 \le C\delta_\eps^2.
\]
The finite pre-entry interval contributes \(O(\eps^2)\), which is absorbed
by \(\delta_\eps^2\).

For the differentiated estimate, let \(p=\rho\) and set
\[
W_p=\partial_pV.
\]
Differentiating the difference equation gives
\[
\dot W_p
=
A_\rho(s)W_p
+
(\partial_\rho A_\rho(s))V
+
D_V\mathcal N_\rho(V,s)W_p
+
\partial_\rho\mathcal N_\rho(V,s)
+
\eps^2e^{2s}D R(U_\eps)\,\partial_\rho U_\eps .
\]
On the common cone neighborhood,
\[
|\partial_\rho A_\rho(s)|+
|\partial_\rho U_0(s;\rho)|+
|\partial_\rho U_\eps(s;\rho)|
\le C_I,
\]
while
\[
|D_V\mathcal N_\rho(V,s)|\le C|V|,
\qquad
|\partial_\rho\mathcal N_\rho(V,s)|\le C|V|^2.
\]
The finite pre-entry difference satisfies
\(W_p(s_0)=O_I(\eps^2)\).  Duhamel's formula and
\(V=O_I(\delta_\eps^2)\) therefore give
\[
\sup_{s_0\le s\le s_\eps}|W_p(s)|
\le
C_I\left(
\eps^2+\delta_\eps^2
+\int_{s_0}^{s_\eps}
e^{-\gamma(s_\eps-\sigma)}
\eps^2e^{2\sigma}\,d\sigma
\right)
\le C_I\delta_\eps^2.
\]
Projecting to the anisotropy block and applying the same renormalized
variation-of-constants formula gives
\[
|\partial_\rho\Delta Q_\eps|
+
|\partial_\rho\Delta X_\eps|
\le C_I\delta_\eps^2.
\]

For the logarithmic scale derivative put
\[
W_\eps=D_\eps V,\qquad D_\eps=\eps\partial_\eps .
\]
The Ricci-flat reference orbit is independent of \(\eps\), and
differentiating the explicit forcing yields
\[
D_\eps\!\left(\eps^2e^{2s}R(U_\eps)\right)
=
2\eps^2e^{2s}R(U_\eps)
+
\eps^2e^{2s}DR(U_\eps)\,D_\eps U_\eps .
\]
Thus \(W_\eps\) satisfies the same stable inequality and
\[
\sup_{s_0\le s\le s_\eps}|W_\eps(s)|
\le C_I\delta_\eps^2.
\]
Finally,
\[
D_\eps s_\eps=\beta-1.
\]
If \(\mathcal A_\eps(s)\) denotes either renormalized response-coordinate
difference, then the already proved equations give
\[
\mathcal A_\eps(s_\eps)=O_I(\delta_\eps^2),
\qquad
\partial_s\mathcal A_\eps(s_\eps)=O_I(\delta_\eps^2).
\]
Hence differentiation of the moving endpoint contributes
\[
(D_\eps s_\eps)\,\partial_s\mathcal A_\eps(s_\eps)
=
O_I(\delta_\eps^2).
\]
Therefore
\[
|D_\eps\Delta Q_\eps|
+
|D_\eps\Delta X_\eps|
\le C_I\delta_\eps^2.
\]
This proves all three stated bounds.
\end{proof}

\subsection{Outer weighted norm}

The outer argument also needs control of the isotropic coefficients.  The
correct background is the translated sine solution
\eqref{eq:shifted-sine-background}.

\begin{lemma}[Uniform translated isotropic background control]
\label{lem:isotropic-background}
Take \(\delta_\eps=\eps^\beta\) with \(0<\beta\le3/5\), and put
\[
m=\frac12\log(ab),\qquad K=e^{-2m}.
\]
On every compact ratio interval \(I\Subset(0,\infty)\), the
positive-Einstein solution issued from the corresponding small singular orbit exists
on \([\delta_\eps,t_0]\) for all sufficiently small \(\eps\), and
\begin{equation}\label{eq:isotropic-coef-norm}
\sup_{\delta_\eps\le t\le t_0}
\left(
 u|H-H_{\sin,\alpha}|+u^2|K-K_{\sin,\alpha}|
\right)
\le C_I\delta_\eps^2,
\qquad
u=t+\alpha_{\eps,\rho}.
\end{equation}
The same estimate holds after one \(\rho\)-derivative or one logarithmic
scale derivative \(D_\eps=\eps\partial_\eps\), where the derivative of the
translated background is included.  These are the weighted coefficient
bounds used in the \(\zeta,\xi\) equations below.
\end{lemma}

\begin{proof}
We work throughout in the shifted coordinate
\[
u=t+\alpha_{\eps,\rho},
\]
so that \(d/du=d/dt\) and the translated sine-ccone background is
\[
H_0(u)=7\cot u,\qquad K_0(u)=6\csc^2u.
\]
Write
\[
p=m_u=\frac12\left(\frac{a_u}{a}+\frac{b_u}{b}\right).
\]
From
\[
\zeta_u=\frac{a_u}{a}-\frac{b_u}{b},
\qquad
\xi_u=\frac{c_u}{c}-p,
\]
we have
\[
\frac{a_u}{a}=p+\frac12\zeta_u,\qquad
\frac{b_u}{b}=p-\frac12\zeta_u,\qquad
\frac{c_u}{c}=p+\xi_u,
\]
and therefore
\begin{equation}\label{eq:H-p-xi}
H=7p+3\xi_u,
\qquad
p=\frac{H-3\xi_u}{7}.
\end{equation}
A direct substitution gives the exact decomposition
\begin{equation}\label{eq:L2-isotropic-split}
|L|^2
=
\frac{H^2}{7}
+\zeta_u^2
+\frac{12}{7}\xi_u^2.
\end{equation}
Using \(H'=-|L|^2-7\) and the sine-cone identity
\[
H_0'=-\frac{H_0^2}{7}-7,
\]
the mean-curvature eror
\[
h:=H-H_0
\]
satisfies the exact constrained scalar equation
\begin{equation}\label{eq:h-exact-transfer}
h_u+\frac{H+H_0}{7}\,h
=
-\zeta_u^2-\frac{12}{7}\xi_u^2.
\end{equation}
This constrained isotropic transfer equation shows, in particular, that there
is no singular independent isotropic mode once the Hamiltonian constraint and
the radial translation have been fixed.

At the overlap \(u_\delta=\delta_\eps+\alpha_{\eps,\rho}\asymp\delta_\eps\),
Lemma~\ref{lem:innerweighted} and Lemma~\ref{lem:radial-shift} give
\begin{equation}\label{eq:h-overlap-data}
u_\delta |h(u_\delta)|
\le C_I\delta_\eps^2,
\qquad
|\zeta|+|\xi|+u(|\zeta_u|+|\xi_u|)
\le C_I\eps^3u^{-3},
\end{equation}
together with the same bounds after one \(\rho\)- or
\(D_\eps\)-derivative.  Shrinking the common outer interval if necessary,
the bootstrap assumption
\[
|h|\le \frac{1}{2}H_0
\]
holds initially and implies
\[
\frac{H+H_0}{7}\ge c\,u^{-1}
\]
there.  The integrating factor in \eqref{eq:h-exact-transfer} is therefore
forward damping.  Variation of constants gives
\[
|h(u)|
\le
C\left(\frac{u_\delta}{u}\right)^c |h(u_\delta)|
+
C\int_{u_\delta}^{u}
\left(\frac{v}{u}\right)^c
\left(\zeta_v^2+\xi_v^2\right)\,dv.
\]
Using \eqref{eq:h-overlap-data},
\[
\zeta_v^2+\xi_v^2
\le C_I\eps^6v^{-8}.
\]
Consequently
\[
\sup_{u_\delta\le u\le u_0}u|h(u)|
\le
C_I\left(
\delta_\eps^2
+
\eps^6\delta_\eps^{-6}
\right).
\]
For \(\beta\le3/4\),
\[
\eps^6\delta_\eps^{-6}\le\delta_\eps^2,
\]
so
\begin{equation}\label{eq:h-uniform-transfer}
\sup_{u_\delta\le u\le u_0}u|H-H_0|
\le C_I\delta_\eps^2.
\end{equation}
This improves the bootstrap and hence closes it.

It remains to control \(K=e^{-2m}\).  From
\eqref{eq:H-p-xi},
\[
\frac{K_u}{K}
=
-2p
=
-\frac{2}{7}H+\frac{6}{7}\xi_u,
\]
whereas
\[
\frac{(K_0)_u}{K_0}=-\frac{2}{7}H_0.
\]
Thus, with
\[
z:=\log\frac{K}{K_0},
\]
we have the second exact constrained identity
\begin{equation}\label{eq:z-exact-transfer}
z_u
=
-\frac{2}{7}h+\frac{6}{7}\xi_u.
\end{equation}
The translated-cone reconstruction at the overlap gives
\[
|z(u_\delta)|\le C_I\delta_\eps^2.
\]
Integrating \eqref{eq:z-exact-transfer}, using
\eqref{eq:h-uniform-transfer} and
\[
|\xi(u)-\xi(u_\delta)|
\le C_I\eps^3\delta_\eps^{-3},
\]
gives
\[
\sup_{u_\delta\le u\le u_0}|z(u)|
\le
C_I\left(
\delta_\eps^2(1+|\log\delta_\eps|)
+\eps^3\delta_\eps^{-3}
\right).
\]
For the value used below,
\(\delta_\eps=\eps^{1/3}\), the second term is \(O(\eps^2)\) and is
absorbed by \(\delta_\eps^2\).  More generally it is harmless for every
fixed \(\beta\le3/5\).  Since the logarithm in the \(h\)-integral is an
artifact of this crude absolute integration, we now use
\eqref{eq:h-exact-transfer} once more: its damping coefficient is
\(2u^{-1}+O(u)\) at the sine-cone end, so
\[
\int_{u_\delta}^{u_0}|h(u)|\,du\le C_I\delta_\eps^2.
\]
Hence
\begin{equation}\label{eq:z-uniform-transfer}
\sup_{u_\delta\le u\le u_0}|z(u)|
\le C_I\delta_\eps^2
\end{equation}
for \(\beta=1/3\), and therefore, because \(u^2K_0\) is uniformly bounded
and \(z=o(1)\),
\[
\sup_{u_\delta\le u\le u_0}
u^2|K-K_0|
\le C_I\delta_\eps^2.
\]
Together with \eqref{eq:h-uniform-transfer} this proves
\eqref{eq:isotropic-coef-norm} on the interval from the overlap to the fixed
equatorial neighborhood.

The corresponding parameter equations are as follows.  Let \(D\) denote
either \(\partial_\rho\) or \(D_\eps\), always taken at fixed shifted
coordinate \(u\).  Differentiating \eqref{eq:h-exact-transfer} gives
\begin{align}
(Dh)_u+\frac{H+H_0}{7}Dh
&=
-\frac{h}{7}DH
-2\zeta_uDq_u
-\frac{24}{7}\xi_uD\xi_u. \label{eq:Dh-transfer}
\end{align}
Since \(DH=Dh\) at fixed \(u\), the first term on the right is
\(O(|h||Dh|)\) and is absorbed into the left-hand coefficient for
sufficiently small \(\eps\).  The differentiated inner and anisotropy
bounds give
\[
|Dq|+|D\xi|+u(|Dq_u|+|D\xi_u|)
\le C_I\eps^3u^{-3},
\]
while the overlap data satisfy
\[
u_\delta|Dh(u_\delta)|\le C_I\delta_\eps^2.
\]
The same integrating-factor estimate therefore yields
\[
\sup_{u_\delta\le u\le u_0}u|Dh(u)|
\le C_I\delta_\eps^2.
\]
Differentiating \eqref{eq:z-exact-transfer} gives
\[
(Dz)_u
=
-\frac27Dh+\frac67D\xi_u,
\]
with \(Dz(u_\delta)=O_I(\delta_\eps^2)\), and hence
\[
\sup_{u_\delta\le u\le u_0}|Dz(u)|
\le C_I\delta_\eps^2.
\]
Thus \eqref{eq:isotropic-coef-norm} holds with one
\(\rho\)- or \(D_\eps\)-derivative as claimed.

Finally, the estimates imply that \(a,b,c\) remain uniformly close, in
relative scale, to the positive translated sine background on every compact
subinterval ending at the fixed equatorial neighborhood.  Hence no warping
factor vanishes there and the positive-Einstein solution continues on
\([\delta_\eps,t_0]\).
\end{proof}

In the shifted coordinate \(u\), the outer anisotropy equation has the same
sine-cone linear part as before:
\[
Y_{uu}+7\cot u\,Y_u+12\csc^2u\,Y=0.
\]
Thus the normalized Frobenius pair
\(\Psi,\overline\Psi\) and the exact transfer constants \(L_0,L_1\) from
Section~\ref{sec:outer-transfer} are unchanged.  Its Wronskian is
\[
W=-2i\sqrt3\,\sin^{-7}u.
\]
For a forcing \(f\), variation of constants from the singular end to fixed
equatorial Cauchy data is bounded by
\[
C\int_{\delta_\eps}^{u_0} u^4|f(u)|\,du,
\]
where replacing the lower limit by
\(\delta_\eps+\alpha_{\eps,\rho}\) changes only the uniform constant.

At the overlap,
\[
|\zeta|+|\xi|+u(|\zeta_u|+|\xi_u|)
\le C\eps^3u^{-3}.
\]
There is also a finite-overlap error coming from expressing the response
expansion, originally written in \(\tau\), in the translated variable
\[
u=\eps(\tau+\sigma(\rho)).
\]
Indeed, with \(\lambda=-3+i\sqrt3\),
\[
\tau^\lambda
=
(\tau+\sigma)^\lambda
\left(1-\frac{\sigma}{\tau+\sigma}\right)^\lambda
=
(\tau+\sigma)^\lambda
\left(1+O_I(\tau^{-1})\right).
\]
Since
\[
\tau+\sigma=\frac{u}{\eps},
\]
the relative error at \(u\asymp\delta_\eps\) is
\begin{equation}\label{eq:shift-conversion-error}
O_I\!\left(\frac{\eps}{\delta_\eps}\right).
\end{equation}
A \(\rho\)-derivative differentiates only the bounded functions
\(Q,X,\sigma\), while \(D_\eps\) differentiates
\(\eps/u\) at fixed shifted outer coordinate; both preserve this order.
Thus the same estimate holds for the differentiated incoming data.
The nonlinear anisotropy forcing is bounded by
\(C\eps^6u^{-8}\), while the translated isotropic background error
contributes
\(C\delta_\eps^2\eps^3u^{-5}\).

\begin{lemma}[Weighted forcing integrability]\label{lem:forcing-integrability}
Assume on the shifted outer interval that
\[
 |\zeta|+|\xi|+u(|\zeta_u|+|\xi_u|)\le C\eps^3u^{-3}
\]
and
\[
u|H-H_{\sin,\alpha}|+u^2|K-K_{\sin,\alpha}|
=O(\delta_\eps^2).
\]
Then the normalized equatorial Cauchy-data error generated by the finite
translation of the response coordinate, the quadratic anisotropy terms,
and the translated-background error is bounded by
\[
 C\left(
\frac{\eps}{\delta_\eps}
+\eps^3\delta_\eps^{-3}
+\delta_\eps^2\bigl(1+|\log\delta_\eps|\bigr)
\right).
\]
The same estimate holds for one \(\rho\)-derivative or one logarithmic
scale derivative \(D_\eps=\eps\partial_\eps\), provided the differentiated
fields satisfy the corresponding weighted bounds.
\end{lemma}

\begin{proof}
The change from the \(\tau\)-based response expansion to the translated
coordinate contributes the term
\(C\eps/\delta_\eps\) by
\eqref{eq:shift-conversion-error}.  The Green kernel is bounded by \(Cu^4\).
The quadratic anisotropy terms therefore contribute, after division by the
natural amplitude \(\eps^3\),
\[
C\eps^{-3}\int_{\delta_\eps}^{u_0}
u^4\eps^6u^{-8}\,du
\le C\eps^3\delta_\eps^{-3}.
\]
The translated-background coefficient error times a linear anisotropy field
has size
\[
O(\delta_\eps^2\eps^3u^{-5}),
\]
and hence contributes
\[
C\delta_\eps^2
\int_{\delta_\eps}^{u_0}u^{-1}\,du
\le
C\delta_\eps^2(1+|\log\delta_\eps|)
\]
after normalization.  Differentiation produces terms of the same weights.  The derivatives of
the shift are \(O_I(\eps)\); when they hit a leading singular-end mode they
produce precisely the already displayed relative error
\(O_I(\eps/\delta_\eps)\), and no larger term.
\end{proof}

\begin{lemma}[Uniform outer transfer estimate]\label{lem:outerweighted}
At a fixed shifted equatorial neighborhood \(u\approx\pi/2\),
\[
\eps^{-3}
\left[
\begin{pmatrix}\zeta\\ \xi_u\end{pmatrix}
-
\begin{pmatrix}
\Re(e^{i\vartheta(\eps)}L_0Q(\rho))\\
\Re(e^{i\vartheta(\eps)}L_1X(\rho))
\end{pmatrix}\eps^3
\right]
=
O_I\!\left(
\frac{\eps}{\delta_\eps}
+\delta_\eps^2\bigl(1+|\log\delta_\eps|\bigr)
+\eps^3\delta_\eps^{-3}
\right),
\]
uniformly in \(\rho\in I\).  The same estimate holds after one
\(\rho\)-derivative or one logarithmic scale derivative
\(D_\eps=\eps\partial_\eps\).
\end{lemma}

\begin{proof}
Lemma~\ref{lem:inner-outer-amplitude-compatibility} identifies the
unshifted leading coefficient with
\(e^{i\vartheta(\eps)}Q\) or \(e^{i\vartheta(\eps)}X\), respectively.
Lemma~\ref{lem:forcing-integrability} controls the outer forcing and the
finite-overlap conversion from \(\tau\) to \(u\).
Lemma~\ref{lem:innerweighted} supplies an additional
\(O(\delta_\eps^2)\) error in the incoming response coefficients, which
is absorbed by the logarithmic term.  The shifted-coordinate derivatives
introduce factors \(\partial_\rho\alpha_{\eps,\rho}\) or
\(D_\eps\alpha_{\eps,\rho}\), both \(O_I(\eps)\); when these differentiate
the leading singular-end profile they give the explicit
\(O_I(\eps/\delta_\eps)\) term and no worse loss.  This proves the stated
\(C^1\) bound.
\end{proof}

Since \(0<\beta<1\), all three error terms
\[
\frac{\eps}{\delta_\eps},
\qquad
\delta_\eps^2\bigl(1+|\log\delta_\eps|\bigr),
\qquad
\eps^3\delta_\eps^{-3}
\]
tend to zero.  The third is the cube of the first and is therefore smaller
once \(\eps/\delta_\eps<1\).  Balancing the first two algebraic powers gives
the convenient choice
\[
 \beta=\frac13,
 \qquad
 \delta_\eps=\eps^{1/3},
\]
for which the total \(C^1\) error is
\[
O\!\left(\eps^{2/3}(1+|\log\eps|)\right).
\]
Only convergence to zero, not this particular rate, is used below.

\subsection{Quantitative moving-\texorpdfstring{\(H=0\)}{H=0} event}

For every positive-Einstein solution,
\begin{equation}\label{eq:Hprime-positive}
H'=-|L|^2-7\le-7.
\end{equation}
Thus \(H\) is strictly decreasing and any zero is automatically unique and
transverse.

The translated sine-cone background has its zero at
\[
t_{\rm eq}(\eps,\rho)
=
\frac\pi2-\alpha_{\eps,\rho}.
\]
Since \(\alpha_{\eps,\rho}=O_I(\eps)\), this point still converges uniformly
to \(\pi/2\), and its first \(\rho\)- and logarithmic-scale derivatives are
\(O_I(\eps)\).

\begin{lemma}[Quantitative translated event lemma]
\label{lem:quant-event}
Let
\[
H_{\alpha}(t,p)=7\cot(t+\alpha_\eps(p))
\]
with
\[
\alpha_\eps=o(1),\qquad
\partial_p\alpha_\eps=o(1)
\]
uniformly.  Fix \(d>0\) sufficiently small and suppose
\[
\|H_\eps-H_\alpha\|_{C^1(
[t_{\rm eq}-d,t_{\rm eq}+d]\times P)}
\le\eta_\eps,
\qquad
\eta_\eps\to0,
\]
and \(H_\eps'\le-7\).  Then for all sufficientlyi small \(\eps\) there is a
unique zero \(t_*(\eps,p)\) in this neighborhood, it is the first zero of
\(H_\eps\), and
\[
|t_*(\eps,p)-t_{\rm eq}(\eps,p)|
\le \frac{\eta_\eps}{7}.
\]
Moreover
\[
\partial_p t_*
=
-\frac{\partial_pH_\eps(t_*,p)}{H_\eps'(t_*,p)},
\]
and
\[
\partial_p\bigl(t_*-t_{\rm eq}\bigr)\to0
\]
uniformly whenever
\[
\partial_p(H_\eps-H_\alpha)\to0
\]
uniformly.
\end{lemma}

\begin{proof}
At \(t=t_{\rm eq}\), \(H_\alpha=0\), so
\[
|H_\eps(t_{\rm eq},p)|\le\eta_\eps.
\]
The inequality \(H_\eps'\le-7\) gives existence and uniqueness of a zero
within distance \(\eta_\eps/7\), exactly as in the unshifted argument.
It is the first zero because \(H_\eps\) starts at \(+\infty\) and is
strictly decreasing.

The implicit-function theorem gives the displayed derivative formula.  The
translated background itself satisfies
\[
\partial_pH_\alpha(t_{\rm eq},p)
=
7\,\partial_p\alpha_\eps(p),
\qquad
H_\alpha'(t_{\rm eq},p)=-7,
\]
and hence
\[
\partial_p t_{\rm eq}=-\partial_p\alpha_\eps.
\]
Because \(H_\eps'(t_*,p)\le-7\), the denominator in the exact implicit
derivative is uniformly separated from zero.  The \(C^1\)-closeness,
together with \(t_*-t_{\rm eq}\to0\), gives
\[
\partial_pH_\eps(t_*,p)
=
7\,\partial_p\alpha_\eps(p)+o(1),
\qquad
H_\eps'(t_*,p)=-7+o(1).
\]
Therefore
\[
\partial_pt_*=-\partial_p\alpha_\eps+o(1)
=\partial_pt_{\rm eq}+o(1),
\]
which proves
\(\partial_p(t_*-t_{\rm eq})\to0\).
\end{proof}

\begin{corollary}[Event replacement in the matching map]
\label{cor:event-replacement}
Suppose, on the same shifted equatorial neighborhood,
\[
(\zeta,\xi_u,\zeta_u,\xi_{uu})=O(\eps^3)
\]
together with the corresponding first parameter derivatives.  Then
replacing \(u=\pi/2\) by the exact event
\[
u_*(\eps,p)
=
t_*(\eps,p)+\alpha_\eps(p)
\]
changes the \(\eps^{-3}\)-normalized matching vector by \(o(1)\) in
\(C^1\).
\end{corollary}

\begin{proof}
Lemma~\ref{lem:quant-event} gives
\[
u_*-\frac\pi2
=
t_*-t_{\rm eq}
=o(1)
\]
in \(C^1\).  Taylor's theorem therefore yields
\[
\zeta(u_*)-\zeta(\pi/2)=o(\eps^3),
\qquad
\xi_u(u_*)-\xi_u(\pi/2)=o(\eps^3).
\]
For completeness, let \(p\) denote either \(\rho\) or the logarithmic
scale parameter.  Differentiating the first difference gives
\[
\partial_p\!\left(\zeta(u_*)-\zeta(\pi/2)\right)
=
\partial_pq(u_*)-\partial_pq(\pi/2)
+\zeta_u(u_*)\,\partial_pu_* .
\]
In the translated coordinate the reference point \(\pi/2\) is fixed.
Since \(u_*-\pi/2=o(1)\) in \(C^1\), while
\(\zeta_u,\partial_pq_u=O(\eps^3)\) on the fixed equatorial neighborhood,
the right-hand side is \(o(\eps^3)\).  Similarly,
\[
\partial_p\!\left(\xi_u(u_*)-\xi_u(\pi/2)\right)
=
\partial_p\xi_u(u_*)-\partial_p\xi_u(\pi/2)
+\xi_{uu}(u_*)\,\partial_pu_* ,
\]
and the assumed bounds for \(\xi_{uu}\) and its first parameter
derivatives again give \(o(\eps^3)\).  Division by \(\eps^3\) proves
the asserted \(C^1\) event replacement.
\end{proof}

Lemma~\ref{lem:isotropic-background}, together with
Lemma~\ref{lem:outerweighted}, gives \(C^1\) convergence to the translated
sine-cone background on the fixed shifted equatorial neighborhood.  Hence the
hypotheses of Lemma~\ref{lem:quant-event} and
Corollary~\ref{cor:event-replacement} hold uniformly on compact
ratio--phase strips.

\begin{lemma}[Normalized differentiated matching estimate]
\label{lem:normalized-C1-matching}
Fix \(I\Subset(0,\infty)\) and put \(\delta_\eps=\eps^{1/3}\).
At the translated reference section \(u=\pi/2\), define
\[
\mathcal R_\eps(\rho)
=
\eps^{-3}
\begin{pmatrix}\zeta\\ \xi_u\end{pmatrix}_{u=\pi/2}
-
\begin{pmatrix}
\Re(e^{i\vartheta(\eps)}L_0Q(\rho))\\
\Re(e^{i\vartheta(\eps)}L_1X(\rho))
\end{pmatrix},
\qquad
\vartheta(\eps)=\sqrt3\log(1/\eps).
\]
Then
\[
\|\mathcal R_\eps\|_{C^1_{\rho,\log\eps}}
\le
C_I\left[
\frac{\eps}{\delta_\eps}
+\delta_\eps^2(1+|\log\delta_\eps|)
+\eps^3\delta_\eps^{-3}
\right]
=
O_I\!\left(\eps^{2/3}(1+|\log\eps|)\right),
\]
where \(C^1_{\rho,\log\eps}\) means that the derivatives
\(\partial_\rho\) and \(D_\eps=\eps\partial_\eps\) are applied to the
\emph{entire normalized expression}.  In particular,
\[
D_\eps(\eps^{-3}\zeta)
=
\eps^{-3}(D_\eps \zeta-3\zeta),
\qquad
D_\eps(\eps^{-3}\xi_u)
=
\eps^{-3}(D_\eps\xi_u-3\xi_u).
\]
The same estimate, with an additional \(o(1)\) term, holds after
replacing \(u=\pi/2\) by the exact \(H=0\) event.
\end{lemma}

\begin{proof}
The assertion at \(u=\pi/2\) is precisely the differentiated form of
Lemma~\ref{lem:outerweighted}, with the normalization derivative written
out explicitly.  The inner response estimate controls both parameter
derivatives of the incoming amplitudes; the translated-background
estimate controls the derivatives of the coefficients; and
\[
\partial_\rho(\eps\sigma)=O_I(\eps),
\qquad
D_\eps(\eps\sigma)=O_I(\eps)
\]
produces only the already recorded relative loss
\(O_I(\eps/\delta_\eps)\).  Corollary~\ref{cor:event-replacement},
with its differentiated Taylor estimate, transfers the statement to
the exact \(H=0\) event.
\end{proof}

\begin{theorem}[\(C^1\) matching on logarithmic phase strips]\label{thm:C1}
Fix a bounded interval \(\Theta\subset\mathbb R\) and define
\[
 \eps_k(\theta)=\exp\!\left(-\frac{k\pi+\theta}{\sqrt3}\right).
\]
Let
\[
\mathcal F_k(\theta,\rho)
=
(-1)^k\eps_k(\theta)^{-3}
\begin{pmatrix}
\zeta(t_*(\eps_k(\theta),\rho))\\
\xi_t(t_*(\eps_k(\theta),\rho))
\end{pmatrix}.
\]
Then
\[
\mathcal F_k\longrightarrow
G(\theta,\rho)=
\begin{pmatrix}
\Re(e^{i\theta}L_0Q(\rho))\\
\Re(e^{i\theta}L_1X(\rho))
\end{pmatrix}
\]
in \(C^1(\Theta\times I)\) as \(k\to\infty\).
\end{theorem}

\begin{remark}[Uniformity actually used]
The interval \(I\Subset(0,\infty)\) is fixed before \(\eps\downarrow0\).
Thus this theorem requires \(C^1\) control only for positive singular-orbit ratios in
a compact interval.  The small-ratio theorem uses merely \(C^0\) convergence
as \(\rho\downarrow0\).  No derivative of the saddle-passage map at the
projective boundary \(\rho=0\) enters the weighted matching theorem.
\end{remark}

\begin{proof}
Work in the shifted coordinate
\[
u=t+\eps_k(\theta)\sigma(\rho).
\]
Since \(du/dt=1\), one has \(\xi_t=\xi_u\).  The inner weighted estimate,
the translated outer transfer estimate, and the translated event
replacement give the \(C^0\) convergence.

For the \(\rho\)-derivative, the dependence of the shift is
\[
\partial_\rho u
=
\eps_k(\theta)\sigma'(\rho)=o(1),
\]
and this contribution has already been included in the differentiated
translated-background and event estimates.  Hence the same lemmas give the
required uniform \(\rho\)-derivative convergence.

For the phase derivative,
\[
 \partial_\theta\eps_k=-\frac1{\sqrt3}\eps_k,\qquad
 \partial_\theta=-\frac1{\sqrt3}D_\eps
\]
along the exact phase strip.  Here \(D_\eps\) acts on the entire
\(\eps^{-3}\)-normalized matching expression; in particular the terms
\(-3\zeta\) and \(-3\xi_u\) arising from differentiating the normalization
are included by Lemma~\ref{lem:normalized-C1-matching}.  The shifted
coordinate contributes
\[
D_\eps(\eps\sigma)=\eps\sigma=o(1),
\]
again covered by the differentiated estimates.  Therefore the logarithmic
scale derivative gives the required \(\theta\)-derivative control.

The normalized error is
\[
O\!\left(
\eps_k^{1-\beta}
+\eps_k^{2\beta}(1+|\log\eps_k|)
+\eps_k^{3(1-\beta)}
\right)
\]
in \(C^1\).  With \(\beta=1/3\) this is
\[
O\!\left(\eps_k^{2/3}(1+|\log\eps_k|)\right)=o(1).
\]
\end{proof}

\begin{remark}[What the corrected estimate actually uses]
The proof does not require the radial translation to vanish, nor does it
require an error smaller than \(o(1)\) after normalization.  The three
independent losses are
\[
\eps/\delta_\eps,\qquad
\delta_\eps^2(1+|\log\delta_\eps|),\qquad
(\eps/\delta_\eps)^3.
\]
For \(\delta_\eps=\eps^{1/3}\) they all tend to zero, with the first two of
order at most
\(O(\eps^{2/3}(1+|\log\eps|))\).  This is the only quantitative input from
the overlap analysis used in the subsequent implicit-function argument.
\end{remark}

\begin{remark}[Why the radial translation does not change the response]
The translation \(\sigma(\rho)\) changes the origin of the outer radial
coordinate but not the leading response amplitudes.  Indeed,
\[
(\tau+\sigma)^{-3+i\sqrt3}
=
\tau^{-3+i\sqrt3}\bigl(1+O(\tau^{-1})\bigr),
\]
so the coefficient of the leading mode is unchanged.  In the shifted
coordinate \(u=t+\eps\sigma(\rho)\), the sine-cone outer equation is exactly the
same equation used in Section~\ref{sec:outer-transfer}; hence the transfer
constants \(L_0,L_1\) and the limiting map \(G\) are unchanged.
\end{remark}

\begin{remark}[No hidden loss at the moving matching event]
The \(C^1\) statement above includes both the reconstructed radial
translation and the dependence of the exact event time.  The matching map
is compared with the translated sine-cone section \(u=\pi/2\), and
Corollary~\ref{cor:event-replacement} transfers this comparison to the exact
\(H=0\) event with \(o(1)\) error after the \(\eps^{-3}\) normalization,
together with its first \(\rho\)- and logarithmic-scale derivatives.
\end{remark}

\section{Reflection, right-singular-orbit smoothness, and topology}\label{sec:reflection}
\label{sec:reflection-topology}

The analytic matching problem ends when
\[
a=b,\qquad c'=0
\]
holds on the Einstein matching section.  These are exactly the conditions that allow
the two \(S^2\)-directions to be exchanged across that section.  We now
verify that this reflection is smooth and that the resulting compact
manifold is \(S^5\times S^3\).

\begin{lemma}[Uniqueness and transversality of the Einstein matching section]
\label{lem:H-section-unique}
Along every positive-Einstein solution in the principal region,
\[
H'=-|L|^2-7<0.
\]
Hence an \(H=0\) point, if it exists, is unique and transverse.  For the
solutions produced by the matching theorem it lies near \(t=\pi/2\), where
all three warping functions remain positive, and is therefore a principal
point.
\end{lemma}

\begin{proof}
The Einstein equations give
\(H'\le-7\), hence strict monotonicity and transversality.  The weighted
matching theorem places the event at \(t=\pi/2+o(1)\), while the limiting
sine-cone profiles are strictly positive there; uniform convergence gives
positivity of \(a,b,c\) at the event for all sufficiently small
\(\varepsilon\).
\end{proof}

\begin{theorem}[Reflection compactification and topology]
\label{thm:reflection-compactification}
Let
\[
g=dt^2+a(t)^2g_{S^2}+b(t)^2g_{S^2}+c(t)^2g_{S^3}
\]
be a smooth left-singular-orbit Einstein solution on \(0\le t\le t_*\), with
\[
a(t)=tA(t^2),\qquad
b(t)=B(t^2),\qquad
c(t)=C(t^2),
\]
where
\[
A(0)=1,\qquad B(0)=b_0>0,\qquad C(0)=c_0>0.
\]
Assume \(t_*>0\) is the \(H=0\) principal point and satisfies
\[
a(t_*)=b(t_*),\qquad
c'(t_*)=0.
\]
Then:

\begin{enumerate}[label=(\roman*)]
\item
\[
a'(t_*)+b'(t_*)=0;
\]

\item the reflection
\begin{equation}\label{eq:reflection-profiles}
a(t_*+s)=b(t_*-s),\qquad
b(t_*+s)=a(t_*-s),\qquad
c(t_*+s)=c(t_*-s)
\end{equation}
defines the unique Einstein continuation through the regular principal
point \(t=t_*\), and this continuation is smooth;

\item with the right-hand geodesic coordinate
\[
u=2t_*-t,
\]
the profiles satisfy
\[
b(2t_*-u)=u\,A(u^2),
\qquad
a(2t_*-u)=B(u^2),
\qquad
c(2t_*-u)=C(u^2);
\]
hence the second \(S^2\)-factor collapses smoothly at \(u=0\);

\item with the chosen labelled principal-orbit trivialization, the resulting
compact cohomogeneity-one manifold is naturally diffeomorphic to
\[
S^5\times S^3.
\]
\end{enumerate}
\end{theorem}

\begin{proof}
At \(t=t_*\),
\[
0=H
=
2\frac{a'}a+2\frac{b'}b+3\frac{c'}c.
\]
Since \(a=b>0\) and \(c'=0\),
\[
a'+b'=0.
\]
This proves (i).

The Einstein equations are invariant under the involution
\[
(t,a,b,c)\longmapsto
(2t_*-t,b,a,c).
\]
The right-hand profiles in \eqref{eq:reflection-profiles} therefore solve the
same regular Einstein ODE for \(s>0\).  At \(s=0\), the values agree because
\(a=b\), and
\[
\partial_s a(t_*+s)|_{s=0}
=
-b'(t_*)=a'(t_*),
\]
\[
\partial_s b(t_*+s)|_{s=0}
=
-a'(t_*)=b'(t_*),
\]
\[
\partial_s c(t_*+s)|_{s=0}
=
-c'(t_*)=0.
\]
Thus the reflected solution has the same Cauchy data at the regular
principal point \(t_*\).  Ordinary uniqueness for the nonsingular Einstein
ODE proves (ii).  Since the Einstein system is real analytic at a principal
point, the common Cauchy-data solution is real analytic across \(t=t_*\).
At the two singular endpoints the parity forms displayed in the statement
are convergent analytic germs, so the compactified Einstein metric is
real analytic in the standard singular-orbit charts.

For (iii), set
\[
u=2t_*-t.
\]
Using the left-singular-orbit parity representation,
\[
a(u)=uA(u^2),\qquad
b(u)=B(u^2),\qquad
c(u)=C(u^2),
\]
the reflection formula gives
\[
b(2t_*-u)=a(u)=uA(u^2),
\]
\[
a(2t_*-u)=b(u)=B(u^2),
\]
\[
c(2t_*-u)=c(u)=C(u^2).
\]
Therefore the right endpoint satisfies the exact smooth polar-product
criterion of Lemma~\ref{lem:smooth-polar-product}, with the second
\(S^2\)-factor now collapsing.  In particular,
\[
b(2t_*)=0,\qquad
\frac{db}{du}(2t_*-u)\bigg|_{u=0}=1,
\]
or equivalently \(b'(2t_*)=-1\) in the original \(t\)-coordinate, while
\[
a'(2t_*)=c'(2t_*)=0.
\]
No higher-order smoothness condition is missing: all parity jets are inherited
exactly from the left-singular-orbit analytic germ.

It remains to identify the topology.  A tubular neighborhood of the left
singular orbit is
\[
M_-
=
D^3\times S^2\times S^3,
\]
because the first \(S^2\) is the unit sphere in the normal \(3\)-disk and the
second \(S^2\) and \(S^3\) survive on the singular orbit.  A tubular
neighborhood of the right singular orbit is
\[
M_+
=
S^2\times D^3\times S^3.
\]
Their common boundary is
\[
S^2\times S^2\times S^3.
\]
Although the reflected \emph{coefficients} \(a\) and \(b\) are exchanged,
the principal-orbit coordinates themselves are not relabelled across the
regular hypersurface \(t=t_*\).  Thus the gluing map of the two disk-bundle
halves is the identity on the ordered
\(S^2\times S^2\times S^3\) boundary factors.  Hence
\[
M
=
\left(
D^3\times S^2
\;\cup_{S^2\times S^2}\;
S^2\times D^3
\right)
\times S^3.
\]
The first parenthesis is the standard decomposition of the join
\[
S^2*S^2.
\]
Indeed, writing
\[
S^5
=
\{(x,y)\in\mathbb R^3\oplus\mathbb R^3:
|x|^2+|y|^2=1\},
\]
the regions \(|x|\ge|y|\) and \(|y|\ge|x|\) are respectively
diffeomorphic to
\[
D^3\times S^2
\quad\text{and}\quad
S^2\times D^3,
\]
and their common boundary is \(S^2\times S^2\), with the standard identity
identification of the two direction variables.  Therefore
\[
D^3\times S^2
\cup
S^2\times D^3
\cong
S^2*S^2
\cong
S^5.
\]
The \(S^3\)-factor is globally untouched by the collapse and the gluing.
Thus
\[
M\cong S^5\times S^3.
\]
This proves (iv).
\end{proof}

\begin{corollary}[No twisted compactification in the reflection construction]
\label{cor:no-twist}
The reflection compactification above produces the product
\(S^5\times S^3\), not a nontrivial \(S^3\)-bundle over \(S^5\) or a twisted
double disk bundle.
\end{corollary}

\begin{proof}
The \(S^3\)-coordinate is unchanged by the reflection and by the boundary
identification.  Hence the gluing map is the product of the standard
\(S^2*S^2\) gluing with the identity on \(S^3\).
\end{proof}

\begin{remark}[Orientation convention at the right singular orbit]
In the original coordinate \(t\), the collapsing right-hand profile has
derivative \(-1\) because \(t\) increases toward the endpoint.  In the
outward normal distance \(u=2t_*-t\), its derivative is \(+1\), exactly as
required by the smooth polar-coordinate criterion.
\end{remark}

\section{Proof of the sequence of real-analytic Einstein metrics theorem}\label{sec:sequence-proof}

All ingredients are now in place.  The inner problem supplies a sign
change, the outer problem transports the response, and the weighted matching
theorem transfers this information to the exact Einstein equations.  We now
eliminate the common phase and obtain one compact Einstein metric for each
sufficiently large phase label.

The two matching equations are controlled by two effective choices.  The
logarithmic dependence on the small parameter \(\eps\) determines a common
oscillatory phase, while the ratio \(\rho\) determines the relative position
of the two response channels.  We use the first matching equation to choose the phase.
After that choice has been made, the second matching equation is a single
scalar equation in \(\rho\).  This is the phase elimination that converts
the original two-component shooting problem into the sign-change argument
below.

\begin{figure}[t]
\centering
\begin{tikzpicture}[>=Latex,
  phasebox/.style={draw,rounded corners=2pt,align=center,text width=9.8cm,
                  minimum height=1.15cm,inner sep=6pt,font=\small},
  signbox/.style={draw,rounded corners=2pt,align=center,text width=2.5cm,
                 minimum height=1.05cm,inner sep=5pt,font=\small},
  arr/.style={->,thick}]
\node[phasebox] (phase) at (0,0)
 {\textbf{Step 1: choose the common phase.}\\
  For every \(\rho\in[\rho_-,\rho_+]\), choose
  \(\theta=\theta_k(\rho)\) so that the first exact matching error is zero.};

\node[align=center,font=\small] (vary) at (0,-1.55)
 {\textbf{Step 2: on this phase branch, vary only \(\rho\).}};

\node[signbox] (left) at (-3.6,-2.65)
 {\(\rho=\rho_-\)\\[2pt]\(h_k(\rho_-)<0\)};
\node[signbox] (middle) at (0,-2.65)
 {\(\rho=\rho_k\)\\[2pt]\(h_k(\rho_k)=0\)};
\node[signbox] (right) at (3.6,-2.65)
 {\(\rho=\rho_+\)\\[2pt]\(h_k(\rho_+)>0\)};

\draw[arr] (phase) -- (vary);
\draw[arr] (left) -- (middle);
\draw[arr] (middle) -- (right);

\node[align=center,text width=11.0cm,font=\small] (concl) at (0,-4.15)
 {At \(\rho_k\), the first error is zero by the phase choice and the second
 error is zero by the intermediate value theorem.\\
 Hence both matching conditions hold.};

\draw[arr] (middle) -- (concl);
\end{tikzpicture}
\caption{The final matching argument.  The phase is first chosen so that
one exact matching equation vanishes along a one-parameter branch.
Restricted to that branch, the second exact matching error is the scalar
function \(h_k(\rho)\).  Its opposite signs at \(\rho_-\) and \(\rho_+\)
force an intermediate zero \(\rho_k\).}
\label{fig:final-sign-change}
\end{figure}
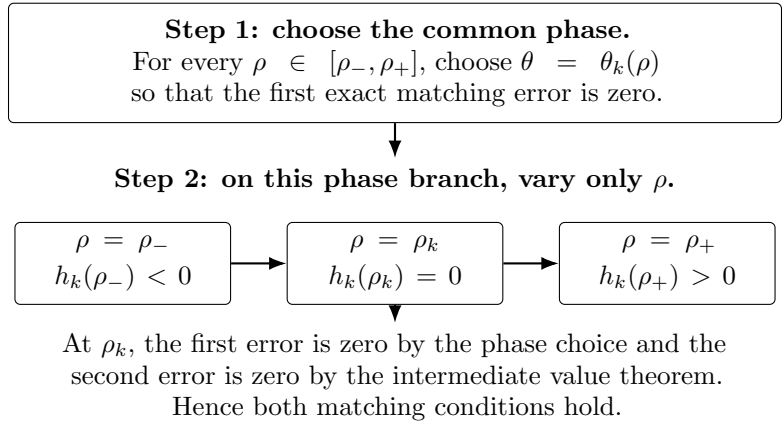

\begin{remark}[Why the scalar sign change solves the full matching problem]
No matching equation is discarded in the phase reduction.  The first exact
component is solved by the implicit-function branch
\(\theta=\theta_k(\rho)\), so it vanishes identically on that branch.
The scalar function \(h_k\) is the second exact matching component
restricted to the same branch.  Therefore
\[
h_k(\rho_k)=0
\]
is equivalent to the simultaneous vanishing of both components of
\(\mathcal F_k\).  This is precisely the payoff of the phase-reduced response
phase invariant: a two-component matching problem becomes one scalar sign
problem after the observable common phase has been removed.
\end{remark}

\begin{theorem}[Existence of the sequence of real-analytic Einstein metrics]
\label{thm:sequence}
For every sufficiently large logarithmic phase index \(k\), there are
parameters
\[
\eps^{(k)}>0,\qquad \rho^{(k)}\in(0,\sqrt2)
\]
such that the corresponding positive-Einstein solution reaches its first
\(H=0\) point with
\[
a=b,\qquad c'=0.
\]
It therefore reflects to a smooth positive Einstein metric \(g^{(k)}\) on
\(S^5\times S^3\) satisfying \(\Ric(g^{(k)})=7g^{(k)}\), and
\[
\log\eps^{(k)}
=
-\frac{k\pi}{\sqrt3}+O(1).
\]
In particular \(\eps^{(k)}\to0\) as \(k\to\infty\).

Moreover, there are an integer \(N\ge1\), an index \(k_0\), and the
arithmetic subsequence
\[
k_j=k_0+jN
\]
such that the selected scales may be chosen strictly decreasing:
\[
\eps^{(k_j)}\downarrow0.
\]
Hence the construction contains an infinite sequence of compact
real-analytic Einstein metrics with scales tending monotonically to zero.
\end{theorem}

\begin{proof}
Theorem~\ref{thm:smallrho} and \eqref{eq:Phi-anchor} give two values
\[
0<\rho_-<\rho_+<\sqrt2
\]
for which
\[
\Phi(\rho_-)<0<\Phi(\rho_+).
\]
Since \(Q\) never vanishes, the continuous phase branch normalized by
\[
e^{i\theta_0(\rho)}L_0Q(\rho)=i|L_0Q(\rho)|
\]
is well defined there.  By compactness and the nonvanishing of \(Q\), there is a uniform bound
\[
|Q(\rho)|\ge \zeta_0>0
\qquad
(\rho\in[\rho_-,\rho_+]).
\]
Choose any closed interval \(\Theta\) whose interior contains the
compact set \(\theta_0([\rho_-,\rho_+])\).  No restriction on the length
of \(\Theta\) is needed for the response-crossing argument.  To extract
manifestly disjoint logarithmic phase strips, choose an integer \(N\ge1\)
such that
\[
N\pi>|\Theta|.
\]
After fixing one sufficiently large \(k_0\), set
\[
k_j=k_0+jN,\qquad j=0,1,2,\ldots .
\]
Then the translated intervals \(k_j\pi+\Theta\) are pairwise disjoint.
For an arbitrary phase index \(k\ge1\) define
\begin{equation}\label{eq:epsktheta-final}
 \eps_k(\theta)
 =
 \exp\!\left(-\frac{k\pi+\theta}{\sqrt3}\right),
 \qquad \theta\in\Theta,
\end{equation}
and define the parity-normalized exact matching map
\[
 \mathcal F_k(\theta,\rho)
 =
 (-1)^k\eps_k(\theta)^{-3}
 \begin{pmatrix}
   \zeta(t_*(\eps_k(\theta),\rho))\\
   \xi_t(t_*(\eps_k(\theta),\rho))
 \end{pmatrix}.
\]
On the Einstein matching section \(H=0\), these two components are exactly equivalent
to the geometric reflection conditions.  Indeed,
\[
\zeta=0\quad\Longleftrightarrow\quad a=b,
\]
while, writing
\(\alpha=a'/a\), \(\beta=b'/b\), \(\gamma=c'/c\),
\[
H=2\alpha+2\beta+3\gamma=0
\]
gives
\[
\xi_t
=
\gamma-\frac{\alpha+\beta}{2}
=
\frac74\,\gamma.
\]
Since \(c>0\) at the principal matching point,
\[
\xi_t=0\quad\Longleftrightarrow\quad c'=0.
\]
Thus \(\mathcal F_k=0\) is precisely the pair
\(a=b,\ c'=0\) on the \(H=0\) section.
The factor \((-1)^k\) compensates exactly for
\[
 e^{i\sqrt3\log(1/\eps_k(\theta))}
 =(-1)^k e^{i\theta}
\]
and does not change the zero set.

The \(C^1\) matching theorem gives
\[
 \mathcal F_k\longrightarrow
 G(\theta,\rho)
 =
 \begin{pmatrix}
 \Re(e^{i\theta}L_0Q(\rho))\\
 \Re(e^{i\theta}L_1X(\rho))
 \end{pmatrix}
\]
in \(C^1(\Theta\times[\rho_-,\rho_+])\).

By construction,
\[
G_1(\theta_0(\rho),\rho)=0,
\qquad
\partial_\theta G_1(\theta_0(\rho),\rho)
=-|L_0Q(\rho)|\le-c_0<0.
\]
The uniform implicit-function theorem therefore produces, for all sufficiently
large \(k\), a unique nearby exact phase branch
\[
\theta_k(\rho)\to\theta_0(\rho)
\quad\text{in }C^1([\rho_-,\rho_+]).
\]
Because the limiting branch lies a positive distance from
\(\partial\Theta\), the exact branch also lies in
\(\operatorname{int}\Theta\) for all sufficiently large \(k\).  Thus it
defines a solution associated with the phase label \(k\).  The intervals
\(k\pi+\Theta\) need not be disjoint; disjointness is used only later on
the sparse subsequence \(k_j\).

Define
\[
 h_k(\rho)
 =
 (\mathcal F_k)_2(\theta_k(\rho),\rho).
\]
Then uniformly on \([\rho_-,\rho_+]\),
\[
h_k(\rho)\longrightarrow
h_0(\rho)
=
\frac{\Phi(\rho)}{|L_0Q(\rho)|}.
\]
The endpoint signs are therefore opposite:
\[
h_k(\rho_-)<0<h_k(\rho_+)
\]
for every sufficiently large \(k\).  By the intermediate value theorem
choose
\[
\rho_k\in(\rho_-,\rho_+)
\]
with
\[
h_k(\rho_k)=0.
\]
By construction,
\[
 (\mathcal F_k)_1(\theta_k(\rho_k),\rho_k)=0
\]
as well.  Put
\[
 \theta_k^*=\theta_k(\rho_k),\qquad
 \eps^{(k)}=\eps_k(\theta_k^*),\qquad
 \rho^{(k)}=\rho_k.
\]
Then \(\mathcal F_k(\theta_k^*,\rho^{(k)})=0\).  By the matching-coordinate
calculation above, at the unique transverse \(H=0\) event this is exactly
\[
a=b,\qquad c'=0.
\]
Lemma~\ref{lem:H-section-unique} shows that the event is the unique Einstein matching section for the solution, so the hypotheses of
Theorem~\ref{thm:reflection-compactification} are satisfied.

Since \(\theta_k^*\in\Theta\) and \(\Theta\) is a fixed bounded interval,
\eqref{eq:epsktheta-final} gives
\[
\log\eps^{(k)}
=
-\frac{k\pi}{\sqrt3}+O(1),
\]
uniformly for the selected solutions, and in particular
\(\eps^{(k)}\to0\).

For the monotone sequence use the indices fixed above,
\[
k_j=k_0+jN,\qquad N\pi>|\Theta|.
\]
The intervals \(k_j\pi+\Theta\) are pairwise disjoint and strictly ordered.
If \(j<\ell\), then every phase value in \(k_\ell\pi+\Theta\) is larger
than every phase value in \(k_j\pi+\Theta\); because
\(\eps=\exp(-\text{phase}/\sqrt3)\), it follows that
\[
\eps^{(k_\ell)}<\eps^{(k_j)}.
\]
Thus
\[
\eps^{(k_j)}\downarrow0.
\]
This proves both the phase-by-phase existence statement and the monotone
subsequence assertion.

The final compactification is supplied by
Theorem~\ref{thm:reflection-compactification}.
\end{proof}

\section{Uniform extension to all oscillatory three-factor dimensions}
\label{sec:uniform-extension}

The preceding sections were written in full detail for
\((p,q)=(2,3)\).  We now verify that no estimate used in the proof depends
essentially on those two numbers.  We record where the dimensions enter and
show that the three remaining possibilities satisfy the same inequalities.
The logical order is: uniform singular germs and Ricci-flat global dynamics,
uniform response nonvanishing and the two sign anchors, uniform outer and
\(C^1\) transfer, phase crossing, and finally intrinsic distinctness and
comparison with products and Böhm metrics.

Throughout this section let
\[
p,q\ge2,\qquad D=2p+q\le 8,
\]
and consider
\begin{equation}\label{eq:uniform-ansatz}
g=dt^2+a(t)^2g_{S^p}+b(t)^2g_{S^p}+c(t)^2g_{S^q},
\qquad \Ric(g)=Dg.
\end{equation}
Since \(p,q\ge2\) and \(D\le8\), the only possibilities are
\[
(p,q)=(2,2),(2,3),(2,4),(3,2).
\]
Set
\begin{equation}\label{eq:uniform-alpha-omega}
\alpha_D=\frac{D-1}{2},
\qquad
\omega_D=\frac12\sqrt{(D-1)(9-D)}.
\end{equation}
The three cases not treated explicitly above are
\[
\begin{array}{c|c|c|c}
(p,q)&D&\alpha_D&\omega_D\\ \hline
(2,2)&6&5/2&\sqrt{15}/2\\
(2,4)&8&7/2&\sqrt7/2\\
(3,2)&8&7/2&\sqrt7/2 .
\end{array}
\]

\subsection{Singular-orbit and Ricci-flat estimates}

We first record the positive-Einstein singular-orbit statement needed in
the uniform argument.  Let
\[
G=SO(p+1)\times SO(p+1)\times SO(q+1),\qquad
K=SO(p)\times SO(p)\times SO(q).
\]
At the left singular orbit the first \(SO(p)\)-factor in \(K\) enlarges
to \(SO(p+1)\), and at the right singular orbit the second one does.
Thus the two tubular neighborhoods are
\[
D^{p+1}\times S^p\times S^q,
\qquad
S^p\times D^{p+1}\times S^q,
\]
and their union is
\[
(S^p*S^p)\times S^q\cong S^{2p+1}\times S^q.
\]

\begin{proposition}[Uniform positive-Einstein singular-orbit germs]
\label{prop:uniform-positive-germs}
For every \(p,q\ge2\) and every \(b_0,c_0>0\), there is a unique
real-analytic Einstein germ within the labelled three-factor ansatz
\[
g=dt^2+a(t)^2g_{S^p}+b(t)^2g_{S^p}+c(t)^2g_{S^q},
\qquad \Ric=Dg,\qquad D=2p+q,
\]
satisfying
\[
a(0)=0,\quad a'(0)=1,\qquad
b(0)=b_0,\quad b'(0)=0,\qquad
c(0)=c_0,\quad c'(0)=0.
\]
The germ depends real-analytically on \((b_0,c_0)\).  Its first
nontrivial coefficients are
\[
b(t)=b_0+
\frac{p-1-Db_0^2}{2(p+1)b_0}\,t^2+O(t^4),
\qquad
c(t)=c_0+
\frac{q-1-Dc_0^2}{2(p+1)c_0}\,t^2+O(t^4).
\]
\end{proposition}

\begin{proof}
The smoothness representation is the dimension-independent polar one,
\[
a(t)=tA(t^2),\qquad b(t)=B(t^2),\qquad c(t)=C(t^2),
\qquad A(0)=1.
\]
Substitution into the tangential Einstein equations gives the same
analytic regular-singular system as in
Theorem~\ref{thm:singular-orbit-germ}, with the integers \(2,2,3\)
replaced by \(p,p,q\).  The constant terms of the \(b\)- and \(c\)-equations
give the displayed coefficients.  The radial constraint has the required
limiting value and then propagates by the contracted Bianchi identity.
The singular initial-value theorem of
Eschenburg--Wang~\cite{EschenburgWang2000} gives convergence, uniqueness,
and real-analytic dependence on the two positive non-collapsing radii.
\end{proof}

Writing
\[
b_0=\eps,\qquad c_0=\eps\rho
\]
therefore gives the exact two-parameter positive-Einstein family used in
the matching argument for every admissible pair.

After magnifying by \(\eps^{-1}\), normalize the Ricci-flat limit by
\[
A(0)=0,\quad A'(0)=1,\quad B(0)=1,\quad C(0)=\rho.
\]
The Ricci-flat equations are
\begin{align}
A''&=\frac{(p-1)(1-A'^2)}A
-p\frac{A'B'}B-q\frac{A'C'}C,\label{eq:uniform-RF-A}\\
B''&=\frac{(p-1)(1-B'^2)}B
-p\frac{A'B'}A-q\frac{B'C'}C,\label{eq:uniform-RF-B}\\
C''&=\frac{(q-1)(1-C'^2)}C
-p\frac{A'C'}A-p\frac{B'C'}B.\label{eq:uniform-RF-C}
\end{align}
The singular-orbit expansion gives
\begin{equation}\label{eq:uniform-bc-expansion}
B(\tau)
=
1+\frac{p-1}{2(p+1)}\tau^2+O(\tau^4),
\qquad
C(\tau)
=
\rho+\frac{q-1}{2(p+1)\rho}\tau^2+O(\tau^4).
\end{equation}
Both quadratic coefficients are positive for all four admissible pairs.
The first-contact proof of \(0<A',B',C'<1\) uses only the positivity of
\(p-1,q-1,p,q\), so the global continuation argument is unchanged.

Under the positive-Einstein rescaling \(t=\eps\tau\), the only new terms
in \eqref{eq:uniform-RF-A}--\eqref{eq:uniform-RF-C} are
\[
-D\eps^2A,\qquad-D\eps^2B,\qquad-D\eps^2C.
\]
Thus in logarithmic inner time \(s=\log\tau\) the perturbation has the same
order \(O(\eps^2e^{2s})\) as in the representative case.

The Ricci-flat cone is
\[
A_{\rm c}=B_{\rm c}
=\sqrt{\frac{p-1}{D-1}}\,\tau,
\qquad
C_{\rm c}
=\sqrt{\frac{q-1}{D-1}}\,\tau.
\]
Every relative shape variable satisfies at the linearized level
\[
w_{ss}+(D-1)w_s+2(D-1)w=0,
\]
with roots
\[
-\alpha_D\pm i\omega_D.
\]
Consequently the two independent anisotropies have leading form
\[
\tau^{-\alpha_D}
\Re\!\left(Z\tau^{i\omega_D}\right).
\]
Introduce the mean-curvature normalized variables exactly as in the
representative case:
\[
X_A=\frac{A'}{\mathcal H},\qquad
X_B=\frac{B'}{\mathcal H},\qquad
X_C=\frac{C'}{\mathcal H},
\]
\[
Y_A=\frac1{A\mathcal H},\qquad
Y_B=\frac1{B\mathcal H},\qquad
Y_C=\frac1{C\mathcal H},
\qquad
G=pX_A^2+pX_B^2+qX_C^2.
\]
They satisfy
\begin{align*}
\dot X_A&=X_A(G-1)+(p-1)Y_A^2,&
\dot Y_A&=Y_A(G-X_A),\\
\dot X_B&=X_B(G-1)+(p-1)Y_B^2,&
\dot Y_B&=Y_B(G-X_B),\\
\dot X_C&=X_C(G-1)+(q-1)Y_C^2,&
\dot Y_C&=Y_C(G-X_C),
\end{align*}
together with
\[
pX_A+pX_B+qX_C=1
\]
and
\[
G+p(p-1)Y_A^2+p(p-1)Y_B^2+q(q-1)Y_C^2=1.
\]
The corresponding nonnegative constrained phase space is compact and
forward invariant by the same boundary-vector-field argument as in
Lemma~\ref{lem:compact}.

\begin{lemma}[Uniform convergence to the full Ricci-flat cone]
\label{lem:uniform-positive-to-cone}
Let \(p,q\ge2\), \(D=2p+q\), and let \(U(\eta)\) be a forward normalized
Ricci-flat orbit in the nonnegative constrained phase space.  If
\[
Y_A(\eta_0)Y_B(\eta_0)Y_C(\eta_0)>0
\]
at one finite time, then
\[
U(\eta)\longrightarrow P_D
\qquad(\eta\to\infty),
\]
where
\[
P_D=
\left(
\frac1D,\frac1D,\frac1D;\
\frac1D\sqrt{\frac{D-1}{p-1}},
\frac1D\sqrt{\frac{D-1}{p-1}},
\frac1D\sqrt{\frac{D-1}{q-1}}
\right).
\]
\end{lemma}

\begin{proof}
The strict region \(Y_AY_BY_C>0\) is forward invariant.  Put
\[
V_D=Y_A^pY_B^pY_C^q.
\]
The normalized equations and the linear constraint give
\[
\frac{\dot V_D}{V_D}
=
p(G-X_A)+p(G-X_B)+q(G-X_C)
=
DG-1.
\]
Weighted Cauchy--Schwarz yields
\[
G
\ge
\frac{(pX_A+pX_B+qX_C)^2}{2p+q}
=
\frac1D.
\]
Hence \(V_D\) is nondecreasing.  Since the constrained phase space is
compact and
\[
V_D(\eta)\ge V_D(\eta_0)>0,
\]
the forward orbit remains in the compact positively invariant set
\[
\mathcal K_{\eta_0}
=
\{U:V_D(U)\ge V_D(U(\eta_0))\},
\]
which lies in the strict \(Y\)-interior.

Equality \(\dot V_D=0\) holds precisely when equality holds in weighted
Cauchy--Schwarz, namely
\[
X_A=X_B=X_C=\frac1D.
\]
A complete invariant orbit contained in this equality set must also have
\(\dot X_A=\dot X_B=\dot X_C=0\).  The normalized equations then force
\[
Y_A^2=Y_B^2=\frac{D-1}{D^2(p-1)},
\qquad
Y_C^2=\frac{D-1}{D^2(q-1)}.
\]
Positivity chooses the displayed coordinates of \(P_D\).  Thus \(P_D\)
is the maximal invariant subset of
\(\{\dot V_D=0\}\cap\mathcal K_{\eta_0}\).  LaSalle's invariance principle
gives \(U(\eta)\to P_D\).
\end{proof}

\begin{lemma}[Constrained spectrum at the full cone]
\label{lem:uniform-cone-spectrum}
At \(P_D\), the linearization of the normalized Ricci-flat system on the
four-dimensional tangent space cut out by the two constraints has
characteristic polynomial
\[
\left(
D^2\lambda^2+D(D-1)\lambda+2(D-1)
\right)^2.
\]
Hence
\[
\operatorname{spec}DF(P_D)
=
\left\{
\frac{-\alpha_D+i\omega_D}{D},
\frac{-\alpha_D-i\omega_D}{D}
\right\},
\]
each eigenvalue having complex multiplicity two.  In particular \(P_D\)
is a hyperbolic sink for every admissible pair.
\end{lemma}

\begin{proof}
Linearize the six normalized equations at \(P_D\), and impose the
linearizations of
\[
pX_A+pX_B+qX_C=1
\]
and
\[
G+p(p-1)Y_A^2+p(p-1)Y_B^2+q(q-1)Y_C^2=1.
\]
Eliminating any two coordinate variations gives a four-dimensional
linear system.  A direct determinant computation simplifies, using only
\(D=2p+q\), to
\[
\chi_{P_D}(\lambda)
=
\left(
D^2\lambda^2+D(D-1)\lambda+2(D-1)
\right)^2.
\]
The roots are
\[
-\frac{D-1}{2D}
\pm
\frac{i}{2D}\sqrt{(D-1)(9-D)}
=
\frac{-\alpha_D\pm i\omega_D}{D}.
\]
For \(D<9\) their real parts are strictly negative and their imaginary
parts are nonzero.
\end{proof}

Lemma~\ref{lem:uniform-positive-to-cone} together with
Lemma~\ref{lem:uniform-cone-spectrum} is the dimension-uniform
global asymptotic-conicality and hyperbolicity statement used below.

\begin{lemma}[Uniform asymptotic coordinates in the oscillatory dimensions]
\label{lem:uniform-asymptotic-coordinates}
Let \(D<9\), set
\[
\alpha_D=\frac{D-1}{2},
\qquad
\omega_D=\frac12\sqrt{(D-1)(9-D)},
\]
and consider
\[
Y_s=(A_D+E(s))Y
\]
on \([s_0,\infty)\), where \(E\in L^1\) and \(A_D\) has eigenvalues
\(-\alpha_D\pm i\omega_D\).  Then every real solution has a unique complex
amplitude \(Z\) in the convention
\[
Y(s)
=
e^{-\alpha_D(s-s_0)}
\bigl(\text{rotation by }\omega_D(s-s_0)\bigr)
(v_Z+o(1)),
\]
and the map from Cauchy data \(Y(s_0)\) to \((\Re Z,\Im Z)\) is a real-linear
isomorphism.
\end{lemma}

\begin{proof}
The proof is the same conjugation argument as in
Lemma~\ref{lem:asymptotic-coordinates}.  Since both eigenvalues have the
same real part, \(A_D\) is real-similar to
\[
-\alpha_DI+\omega_DJ.
\]
Thus
\[
e^{-A_D(s-s_0)}E(s)e^{A_D(s-s_0)}
\]
has \(L^1\)-norm bounded by a constant multiple of \(\|E(s)\|\).  The
conjugated fundamental matrix converges.  Liouville's formula shows that
its determinant converges to a nonzero limit, hence the limiting matrix is
invertible.  Resolving the limit in the complex eigencoordinate gives the
claimed amplitude isomorphism.
\end{proof}

\begin{proposition}[Uniform response map and nonvanishing]
\label{prop:uniform-response-nonzero}
For every admissible pair \((p,q)\), the normalized Ricci-flat
singular-orbit family defines smooth response functions
\[
Q,X:(0,\infty)\longrightarrow\mathbb C,
\]
and
\[
Q(\rho)\ne0\qquad\text{for every }\rho>0.
\]
Consequently \(Q\) is bounded away from zero on every compact interval
\(J\Subset(0,\infty)\).
\end{proposition}

\begin{proof}
Lemma~\ref{lem:uniform-positive-to-cone} puts every normalized
trajectory in the basin of \(P_D\), and
Lemma~\ref{lem:uniform-cone-spectrum} gives its hyperbolicity.  Smooth dependence
of the singular-orbit germ on \(\rho\), followed by smooth dependence of
the flow to a fixed late section and the asymptotic-coordinate map at
\(P_D\), gives the smooth functions \(Q(\rho)\) and \(X(\rho)\).

Because the first two factors have the same dimension, the relative
variable \(q_1=\log(A/B)\) satisfies, along each fixed Ricci-flat
trajectory, a homogeneous second-order linear equation.  Near \(P_D\) its
first-order coefficient matrix is an \(L^1\) perturbation of the constant
matrix with roots
\[
-\alpha_D\pm i\omega_D.
\]
Lemma~\ref{lem:uniform-asymptotic-coordinates} therefore gives an
isomorphism from late Cauchy data to the unique complex amplitude in the
convention
\[
q_1(s)
=
e^{-\alpha_Ds}
\Re\!\left(Q(\rho)e^{i\omega_Ds}\right)
+o(e^{-\alpha_Ds}).
\]
Thus the abstract asymptotic amplitude in the lemma is exactly the response
coefficient \(Q(\rho)\).  If \(Q(\rho)=0\), the Cauchy data
\((q_1,q_{1,s})\) vanish at a late section.  Uniqueness for the homogeneous
linear \(q_1\)-equation along the fixed trajectory gives \(q_1\equiv0\).
This is impossible for a left-singular-orbit trajectory: for every
sufficiently small positive \(\tau\), \(A(\tau)/B(\tau)\) is close to
\(\tau\), and hence is not identically \(1\).  Thus \(Q(\rho)\ne0\) for
every \(\rho>0\).  Smoothness of \(Q\) and compactness give the last
claim.
\end{proof}

\subsection{The two sign anchors}

Put
\[
\kappa=\sqrt{\frac{q-1}{p-1}},
\qquad
q_1=\log\frac AB,
\qquad
\xi=\log\frac{C}{\kappa\sqrt{AB}}.
\]
The small-ratio limiting chain is
\[
P_0\longrightarrow P_{p+q}\longrightarrow P_D.
\]
At the intermediate equilibrium the transverse \(Y_B\)-eigenvalue is
\(1/(p+q)>0\); the other constrained eigenvalues are
\[
-\frac{p+q-1}{p+q},
\qquad
-\frac{p+q-1}{2(p+q)}
\pm
\frac{i}{2(p+q)}
\sqrt{(p+q-1)(9-p-q)}.
\]
Thus there is exactly one unstable direction in all four cases.  In
particular, for \((p,q)=(2,4)\) the complex pair is
\[
-\frac5{12}\pm i\frac{\sqrt{15}}{12}.
\]

The preceding eigenvalue count alone does not yet justify the small-ratio
anchor: one must also check that the physical singular-orbit family meets
the stable manifold of the intermediate saddle transversely and that the
outgoing response coefficient does not vanish.  We record these points
explicitly.

For every normalized inner trajectory,
\[
\frac{Y_B}{Y_C}=\frac{C}{B},
\qquad
\rho=\lim_{\eta\to-\infty}\frac{Y_B}{Y_C}.
\]
At \(P_0\) the two physical unstable modes have the same exponent.
Use their leading coefficients \((a_B,a_C)\) as unstable coordinates.
Translation of normalized time multiplies both coefficients by the same
positive factor, while
\[
\frac{a_B}{a_C}=\rho
\]
for the physical family.  After fixing a smooth local section of
\(W^u(P_0)\), for instance by \(a_C=a_*>0\) with \(a_*\) fixed sufficiently
small, the orbit family is therefore a
\(C^1\) curve extending to the projective boundary direction
\(\rho=0\), with nonzero transverse derivative in the \(a_B\)-direction.
This argument takes place entirely on the unstable manifold and does not
differentiate a degenerate singular-orbit germ at \(\rho=0\).

The invariant face \(Y_B=0\) contains the limiting incoming connection.
At \(P_{p+q}\) its tangent space is the full stable eigenspace, because the
only complementary eigenvalue is the positive transverse eigenvalue
\(1/(p+q)\).  Hence, locally,
\[
W^s(P_{p+q})=\{Y_B=0\}
\]
inside the constraint manifold.  Transport of the nonzero transverse
direction from the section near \(P_0\) to a fixed incoming section near
\(P_{p+q}\) is governed along the limiting orbit by
\[
\dot w=(G_0-X_{B,0})w,
\]
whose fundamental solution never vanishes.  Consequently
\[
Y_{B,\mathrm{in}}(\rho)=c_*\rho+o(\rho),
\qquad c_*>0.
\]
The physical family therefore approaches \(W^s(P_{p+q})\) transversely.
The inclination lemma then gives convergence of the outgoing
intersections to the positive unstable heteroclinic from
\(P_{p+q}\) to \(P_D\).

That outgoing heteroclinic lies in the invariant manifold
\[
C=\kappa A.
\]
Consequently
\[
\xi=\frac12q_1
\]
identically along it.  Let $Q_0$ denote its $q_1$ response coefficient.
Then $Q_0\ne0$.  Indeed, because the first two factors have the same
dimension, $q_1=\log(A/B)$ satisfies a homogeneous second-order linear
equation.  Near $P_D$ its coefficient matrix is an $L^1$ perturbation of
the constant system with roots
\[
-\alpha_D\pm i\omega_D.
\]
The asymptotic-coordinate map from Cauchy data at a fixed late section to
the complex leading coefficient is therefore an isomorphism.  If $Q_0=0$,
uniqueness would force $q_1\equiv0$ on the heteroclinic, contradicting the
fact that it leaves $P_{p+q}$ in the positive $Y_B$ direction.  Hence
$Q_0\ne0$, and the identity $\xi=q_1/2$ gives $X_0=Q_0/2$.

Choose a fixed section in the basin of \(P_D\) met by the outgoing
heteroclinic.  The transverse saddle passage gives convergence of the
section points to the heteroclinic point.  In a fixed sink chart at
\(P_D\), the two anisotropy systems are \(L^1\) perturbations of the
constant system with roots \(-\alpha_D\pm i\omega_D\); the perturbation
and its first initial-data derivatives decay exponentially.  The
variation-of-constants formula therefore makes the asymptotic-amplitude
map continuous (indeed smooth) in the section point.  Consequently
\[
(Q(\rho),X(\rho))\longrightarrow(Q_0,X_0)
\qquad(\rho\downarrow0).
\]
In particular $Q(\rho)$ stays nonzero for all sufficiently small positive
$\rho$, and
\begin{equation}\label{eq:uniform-smallratio-anchor}
\frac{X(\rho)}{Q(\rho)}\longrightarrow\frac12
\qquad(\rho\downarrow0).
\end{equation}
At the balanced value \(\rho=\kappa\), the invariant relation is
\(C=\kappa B\), so
\begin{equation}\label{eq:uniform-balanced-anchor}
X(\kappa)=-\frac12Q(\kappa).
\end{equation}
These are the two sign anchors used in the phase-reduced matching argument.

\subsection{Outer transfer}

The positive-Einstein sine-cone is
\[
a_0=b_0=
\sqrt{\frac{p-1}{D-1}}\sin t,
\qquad
c_0=
\sqrt{\frac{q-1}{D-1}}\sin t.
\]
Both anisotropy channels satisfy
\begin{equation}\label{eq:uniform-jacobi}
Y''+D\cot t\,Y'
+2(D-1)\csc^2t\,Y=0.
\end{equation}
The amplitude convention is compatible with the inner scaling without
any dimension-dependent factor.  Indeed, if
\[
Y_{\rm in}(\tau)
=
\tau^{-\alpha_D}
\Re\!\left(Z\tau^{i\omega_D}\right)
+o(\tau^{-\alpha_D}),
\]
then, with \(t=\eps\tau\),
\[
Y_{\rm in}(t/\eps)
=
\eps^{\alpha_D}
\Re\!\left(
e^{i\omega_D\log(1/\eps)}
Z\,t^{-\alpha_D+i\omega_D}
\right)
+o(\eps^{\alpha_D}t^{-\alpha_D}).
\]
Thus the normalized Frobenius coefficient is exactly
\(e^{i\omega_D\log(1/\eps)}Z\).

Let
\[
\Psi(t)=t^{-\alpha_D+i\omega_D}(1+O(t^2))
\]
be the normalized complex Frobenius solution, and set
\[
L_0=\Psi(\pi/2),\qquad L_1=\Psi'(\pi/2).
\]
Abel's identity gives \(W=C\sin^{-D}t\).  Since
\(2\alpha_D+1=D\), the Frobenius normalization yields
\[
W(\Psi,\overline\Psi)
=
-2i\omega_D t^{-D}(1+O(t^2))
\]
at the singular end, and therefore
\[
W(\Psi,\overline\Psi)
=-2i\omega_D\sin^{-D}t.
\]
hence
\begin{equation}\label{eq:uniform-W-sign}
\Im(L_0\overline{L_1})=-\omega_D<0.
\end{equation}

\begin{proposition}[Uniform response orientation]
\label{prop:uniform-response-orientation}
For every admissible pair \((p,q)\), let
\[
\Phi(\rho)
=
\Im\!\left(L_0Q(\rho)\overline{L_1X(\rho)}\right),
\qquad
\kappa=\sqrt{\frac{q-1}{p-1}}.
\]
Then
\[
\Phi(\rho)<0
\qquad\text{for all sufficiently small }\rho>0,
\]
while
\[
\Phi(\kappa)>0.
\]
In particular, there exists \(0<\rho_-<\kappa\) such that
\[
\Phi(\rho_-)<0<\Phi(\kappa).
\]
\end{proposition}

\begin{proof}
By the small-ratio response relation
\eqref{eq:uniform-smallratio-anchor},
\[
X(\rho)
=
\left(\frac12+o(1)\right)Q(\rho)
\qquad(\rho\downarrow0).
\]
Since \(Q(\rho)\neq0\) by
Proposition~\ref{prop:uniform-response-nonzero}, it follows from
\eqref{eq:uniform-W-sign} that
\[
\Phi(\rho)
=
\left(\frac12+o(1)\right)
|Q(\rho)|^2\Im(L_0\overline{L_1})
<0
\]
for all sufficiently small positive \(\rho\).

At the balanced value, \eqref{eq:uniform-balanced-anchor} gives
\(X(\kappa)=-Q(\kappa)/2\).  Hence
\[
\Phi(\kappa)
=
-\frac12|Q(\kappa)|^2
\Im(L_0\overline{L_1})
=
\frac{\omega_D}{2}|Q(\kappa)|^2
>0.
\]
This proves the claimed orientation.
\end{proof}

Thus the scalar sign change is now isolated as a geometric response
statement; no dimension-specific special-function computation is needed.

\subsection{Uniform weighted matching estimates}

We record only the dimension-dependent points of the matching argument;
the full weighted proof for the representative case is in
Section~\ref{sec:weighted-matching}.  Put
\[
\delta_\eps=\eps^\beta.
\]
The reconstructed radial translation satisfies
\[
S=\frac D{\mathcal H}-\tau,
\qquad
S'=DG-1.
\]
Because the linear term in \(G-1/D\) vanishes on the constraint,
\[
DG-1=O(\tau^{-2\alpha_D}),
\]
and hence translation of a response mode contributes
\begin{equation}\label{eq:uniform-shift-error}
O(\eps/\delta_\eps).
\end{equation}

For the isotropic background set
\[
m=\frac12(\log a+\log b),
\qquad
H=Dm_u+q\xi_u,
\qquad
h=H-D\cot u.
\]
The exact identities needed below are
\begin{equation}\label{eq:uniform-h-identity}
h_u+\frac{H+D\cot u}{D}h
=
-\frac p2q_{1,u}^2-\frac{2pq}{D}\xi_u^2
\end{equation}
and, with \(K=e^{-2m}\),
\begin{equation}\label{eq:uniform-K-identity}
(\log(K/K_0))_u
=
-\frac2Dh+\frac{2q}{D}\xi_u.
\end{equation}
Since
\[
q_1,\xi=O(\eps^{\alpha_D}u^{-\alpha_D}),
\]
the stricter overlap condition is
\begin{equation}\label{eq:uniform-beta-condition}
\beta\le\frac{\alpha_D}{\alpha_D+2}.
\end{equation}
For the worst value \(\alpha_D=5/2\), this permits \(\beta=1/3\).

The Green kernel of the uniform Jacobi equation has singular weight
\(u^{\alpha_D+1}\).  Consequently the quadratic anisotropy source gives
\begin{equation}\label{eq:uniform-nonlinear-error}
O\!\left((\eps/\delta_\eps)^{\alpha_D}\right),
\end{equation}
while the isotropic coefficient error gives
\begin{equation}\label{eq:uniform-log-error}
O\!\left(\delta_\eps^2(1+|\log\delta_\eps|)\right).
\end{equation}
Together with \eqref{eq:uniform-shift-error},
\begin{equation}\label{eq:uniform-total-error}
E_\eps
=
O\!\left[
\frac{\eps}{\delta_\eps}
+
\delta_\eps^2(1+|\log\delta_\eps|)
+
\left(\frac{\eps}{\delta_\eps}\right)^{\alpha_D}
\right].
\end{equation}
For \(\delta_\eps=\eps^{1/3}\), this is
\[
O\!\left(\eps^{2/3}(1+|\log\eps|)\right)
\]
in all four dimensions.  The point requiring a separate proof is that
the same bound survives one parameter derivative; we give that argument
next.

We next justify the corresponding \(C^1\) estimate, including both moving
endpoints that occur in the matching construction.

\begin{lemma}[Uniform differentiated inner response estimate]
\label{lem:uniform-differentiated-inner}
Let
\[
s_\eps=\log\frac{\delta_\eps}{\eps},
\qquad
\delta_\eps=\eps^\beta,
\qquad
0<\beta<1.
\]
On every compact ratio interval \(I\Subset(0,\infty)\), the normalized
positive-Einstein and Ricci-flat response coordinates satisfy
\[
|\Delta Q_\eps|+|\Delta X_\eps|
+
|\partial_\rho\Delta Q_\eps|+|\partial_\rho\Delta X_\eps|
+
|D_\eps\Delta Q_\eps|+|D_\eps\Delta X_\eps|
\le C_I\delta_\eps^2
\]
at the moving overlap time \(s=s_\eps\).  Here
\(D_\eps=\eps\partial_\eps\) is taken at fixed normalized singular-orbit
ratio.
\end{lemma}

\begin{proof}
In logarithmic inner time the scaled positive-Einstein equation has the
dimension-independent perturbative form
\[
\dot U=F_0(U)+\eps^2e^{2s}R_D(U),
\]
where \(R_D\) is smooth on the common cone neighborhood.  By
Lemma~\ref{lem:uniform-cone-spectrum}, in physical logarithmic time the two
constrained anisotropy blocks have exponents
\[
-\alpha_D\pm i\omega_D.
\]
The finite set of four admissible pairs therefore has a common exponential
stability margin after entering a fixed cone neighborhood.  Duhamel's
formula gives
\[
\sup_{s_0\le s\le s_\eps}|U_\eps-U_0|
\le C_I\eps^2e^{2s_\eps}
=C_I\delta_\eps^2.
\]
As in Lemma~\ref{lem:innerweighted}, the isotropic Einstein forcing has
anisotropy projection vanishing on the cone and hence gains one anisotropy
factor.  After multiplication by the inverse anisotropy decay
\(e^{\alpha_Ds}\), its integral is still bounded by
\(C_I\delta_\eps^2\).  Thus
\[
|\Delta Q_\eps|+|\Delta X_\eps|
\le C_I\delta_\eps^2.
\]

A \(\rho\)-derivative satisfies the linearized difference equation.  The
coefficient derivatives and the derivatives of the Ricci-flat reference
orbit are uniformly bounded on \(I\), while every genuinely nonlinear
new term contains either \(U_\eps-U_0\) or the Einstein forcing.  The same
stable Duhamel estimate yields
\[
|\partial_\rho\Delta Q_\eps|
+
|\partial_\rho\Delta X_\eps|
\le C_I\delta_\eps^2.
\]

For \(D_\eps\), at fixed \(s\),
\[
D_\eps\!\left(\eps^2e^{2s}R_D(U_\eps)\right)
=
2\eps^2e^{2s}R_D(U_\eps)
+
\eps^2e^{2s}DR_D(U_\eps)D_\eps U_\eps,
\]
so the fixed-\(s\) differentiated response error is again
\(O_I(\delta_\eps^2)\).  It remains to differentiate the moving endpoint.
Since
\[
D_\eps s_\eps=\beta-1,
\]
we must add
\[
(\beta-1)\,\partial_s\mathcal A_\eps(s_\eps)
\]
for either renormalized response-coordinate difference
\(\mathcal A_\eps\).  The renormalized variation-of-constants equation
shows at the endpoint that
\[
\partial_s\mathcal A_\eps(s_\eps)
=
O_I(\eps^2e^{2s_\eps})+O_I(\delta_\eps^2)
=
O_I(\delta_\eps^2).
\]
Hence the moving-endpoint contribution is also \(O_I(\delta_\eps^2)\).
This proves the lemma.
\end{proof}

For the outer problem let \(\mathscr D\) denote either
\(\partial_\rho\) or \(D_\eps\), with the shifted coordinate \(u\) held
fixed.  Put
\[
h=H-D\cot u.
\]
Differentiating the exact identity
\[
h_u+\frac{H+D\cot u}{D}h
=
-\frac p2q_{1,u}^2-\frac{2pq}{D}\xi_u^2
\]
gives
\[
(\mathscr Dh)_u+\frac{H+D\cot u}{D}\mathscr Dh
=
-\frac hD\mathscr Dh
-pq_{1,u}\mathscr Dq_{1,u}
-\frac{4pq}{D}\xi_u\mathscr D\xi_u.
\]
Thus every new nonlinear source is bilinear: one undifferentiated
anisotropy times one differentiated anisotropy.  The differentiated inner
estimate and the outer bootstrap give
\[
|q_{1,u}|+|\xi_u|
+
|\mathscr Dq_{1,u}|+|\mathscr D\xi_u|
\le
C_I\eps^{\alpha_D}u^{-\alpha_D-1}.
\]
Therefore the differentiated quadratic source has the same weight
\[
O_I(\eps^{2\alpha_D}u^{-2\alpha_D-2})
\]
as the undifferentiated one.  The integrating-factor argument consequently
gives
\[
\sup_{\delta_\eps\le u\le u_0}u|\mathscr Dh(u)|
\le C_I\delta_\eps^2
\]
whenever
\[
\beta\le\frac{\alpha_D}{\alpha_D+1}.
\]

For
\[
z=\log(K/K_0)
\]
we have
\[
(\mathscr Dz)_u
=
-\frac2D\mathscr Dh+\frac{2q}{D}\mathscr D\xi_u.
\]
Its overlap value is \(O_I(\delta_\eps^2)\), and the anisotropic integral
has size
\[
O_I\!\left((\eps/\delta_\eps)^{\alpha_D}\right).
\]
Thus the stronger condition
\[
\beta\le\frac{\alpha_D}{\alpha_D+2}
\]
ensures
\[
\sup_{\delta_\eps\le u\le u_0}|\mathscr Dz(u)|
\le C_I\delta_\eps^2.
\]
For the worst value \(\alpha_D=5/2\), the right-hand side of the overlap
restriction is \(5/9\), so \(\beta=1/3\) has uniform margin.

We now include the moving lower limit in the outer Green formula.  A
normalized quadratic source has integrand bounded by
\[
C_I\eps^{\alpha_D}u^{-\alpha_D-1}
\]
after multiplication by the Green weight and division by the leading
amplitude.  Differentiating
\[
\int_{\delta_\eps}^{u_0}F_\eps(u)\,du
\]
with \(D_\eps\) produces the boundary term
\[
-(D_\eps\delta_\eps)F_\eps(\delta_\eps)
=
-\beta\delta_\eps F_\eps(\delta_\eps),
\]
whose size is
\[
O_I\!\left(
\eps^{\alpha_D}\delta_\eps^{-\alpha_D}
\right)
=
O_I\!\left((\eps/\delta_\eps)^{\alpha_D}\right).
\]
For a normalized background-coefficient source the Green integrand is
\(O_I(\delta_\eps^2u^{-1})\), so its moving-limit term is merely
\(O_I(\delta_\eps^2)\).  Finally, differentiation of the translated
singular mode contributes
\[
O_I(\eps/\delta_\eps),
\]
because
\[
D_\eps\alpha_{\eps,\rho}=O_I(\eps),
\qquad
\partial_\rho\alpha_{\eps,\rho}=O_I(\eps),
\]
and one derivative of \(u^{-\alpha_D+i\omega_D}\) costs \(u^{-1}\).
Thus the moving lower limit and the moving translation introduce no loss
beyond the three already present in the \(C^0\) estimate.

There is one further bookkeeping point for \(D_\eps\): it differentiates
the whole normalized matching vector.  Hence
\[
D_\eps(\eps^{-\alpha_D}q_1)
=
\eps^{-\alpha_D}(D_\eps q_1-\alpha_Dq_1),
\]
and
\[
D_\eps(\eps^{-\alpha_D}\xi_u)
=
\eps^{-\alpha_D}(D_\eps\xi_u-\alpha_D\xi_u).
\]
These \(-\alpha_D\)-terms are exactly the derivative of the leading
normalized response and do not create an additional error.

\begin{proposition}[Uniform weighted \(C^1\) matching estimate]
\label{prop:uniform-C1-matching}
Let \(I\Subset(0,\infty)\) be a compact ratio interval and choose
\[
\delta_\eps=\eps^{1/3}.
\]
At a fixed shifted equatorial section, define the normalized matching
error
\[
E_{\eps,\rho}
=
\eps^{-\alpha_D}
\begin{pmatrix}
q_1\\ \xi_u
\end{pmatrix}
-
\begin{pmatrix}
\Re\!\left(
e^{i\omega_D\log(1/\eps)}L_0Q(\rho)
\right)\\[2mm]
\Re\!\left(
e^{i\omega_D\log(1/\eps)}L_1X(\rho)
\right)
\end{pmatrix}.
\]
Then
\[
\|E_{\eps,\rho}\|
+
\|\partial_\rho E_{\eps,\rho}\|
+
\|D_\eps E_{\eps,\rho}\|
\le
C_I\left[
\frac{\eps}{\delta_\eps}
+
\delta_\eps^2(1+|\log\delta_\eps|)
+
\left(\frac{\eps}{\delta_\eps}\right)^{\alpha_D}
\right].
\]
Consequently,
\[
\|E_{\eps,\rho}\|_{C^1_{\rho,\log\eps}}
=
O_I\!\left(
\eps^{2/3}(1+|\log\eps|)
\right)
=o(1)
\]
uniformly for all four admissible dimension pairs.
\end{proposition}

\begin{proof}
Lemma~\ref{lem:uniform-differentiated-inner} controls the incoming
amplitude and its two parameter derivatives at the moving overlap.
The differentiated \(h,z\) estimates above control the two isotropic
coefficients in the outer Jacobi equations with exactly the same
\(O_I(\delta_\eps^2)\) weights as in the \(C^0\) argument.
Variation of constants then gives three contributions:
\[
O_I(\eps/\delta_\eps)
\]
from translated inner-to-outer conversion,
\[
O_I\!\left((\eps/\delta_\eps)^{\alpha_D}\right)
\]
from the quadratic anisotropy forcing, and
\[
O_I\!\left(
\delta_\eps^2(1+|\log\delta_\eps|)
\right)
\]
from the isotropic coefficient error.  The preceding moving-lower-limit
calculation proves that parameter differentiation does not enlarge these
orders, and the normalization identities account for differentiation of
\(\eps^{-\alpha_D}\).

For \(\delta_\eps=\eps^{1/3}\), the first two algebraic powers are
\[
\eps^{2/3},
\qquad
\eps^{2\alpha_D/3}.
\]
Since \(\alpha_D\ge5/2\), the nonlinear term is at worst
\(\eps^{5/3}\), while the background term is
\(O(\eps^{2/3}(1+|\log\eps|))\).  This proves the proposition.
\end{proof}

\begin{lemma}[Uniform replacement by the exact \(H=0\) event]
\label{lem:uniform-event-replacement}
Let \(I\Subset(0,\infty)\) be the compact ratio interval used in the
matching argument and put
\[
\alpha_{\eps,\rho}=\eps\sigma(\rho),
\qquad
u=t+\alpha_{\eps,\rho}.
\]
Let \(u_*(\eps,\rho)\) denote the shifted coordinate of the first exact
zero of \(H\).  Then
\[
u_*(\eps,\rho)-\frac\pi2=o(1)
\]
in \(C^1_{\rho,\log\eps}\), uniformly for \(\rho\in I\).  Moreover,
replacing the fixed section \(u=\pi/2\) by \(u=u_*(\eps,\rho)\) changes
the \(\eps^{-\alpha_D}\)-normalized matching vector
\[
\eps^{-\alpha_D}
\begin{pmatrix}
q_1\\ \xi_u
\end{pmatrix}
\]
by \(o(1)\) in \(C^1_{\rho,\log\eps}\).
\end{lemma}

\begin{proof}
The differentiated \(h\)-estimates used in
Proposition~\ref{prop:uniform-C1-matching} give, on a fixed neighborhood
of \(u=\pi/2\),
\[
H(t,\eps,\rho)
=
D\cot u+o(1)
\]
in \(C^1_{\rho,\log\eps}\), uniformly on \(I\).  The exact Einstein
identity
\[
H'=-|L|^2-D\le-D
\]
makes \(H\) strictly decreasing, while the translated sine-cone
background has its zero at \(u=\pi/2\) and derivative \(-D\) there.
The quantitative implicit-function argument of
Lemma~\ref{lem:quant-event}, with \(7\) replaced by \(D\), therefore gives
a unique first zero and
\[
u_*-\frac\pi2=o(1)
\]
together with
\[
\partial_\rho\!\left(u_*-\frac\pi2\right)=o(1),
\qquad
D_\eps\!\left(u_*-\frac\pi2\right)=o(1).
\]

On the same fixed equatorial neighborhood, the outer bootstrap and its
two parameter derivatives give
\[
q_1,\ q_{1,u},\ \xi_u,\ \xi_{uu}
=
O_I(\eps^{\alpha_D}),
\]
and the same estimates after applying either
\(\partial_\rho\) or \(D_\eps\).  Taylor's theorem yields
\[
q_1(u_*)-q_1(\pi/2)=o_I(\eps^{\alpha_D}),
\]
\[
\xi_u(u_*)-\xi_u(\pi/2)=o_I(\eps^{\alpha_D}).
\]
For \(\mathscr D\in\{\partial_\rho,D_\eps\}\),
\begin{align*}
\mathscr D\!\left(q_1(u_*)-q_1(\pi/2)\right)
&=
\mathscr Dq_1(u_*)-\mathscr Dq_1(\pi/2)
+q_{1,u}(u_*)\,\mathscr Du_*,\\
\mathscr D\!\left(\xi_u(u_*)-\xi_u(\pi/2)\right)
&=
\mathscr D\xi_u(u_*)-\mathscr D\xi_u(\pi/2)
+\xi_{uu}(u_*)\,\mathscr Du_* .
\end{align*}
The first differences are
\(o_I(\eps^{\alpha_D})\) because
\(u_*-\pi/2=o(1)\) and the differentiated \(u\)-derivatives are
\(O_I(\eps^{\alpha_D})\); the endpoint terms are \(o_I(\eps^{\alpha_D})\), because the first
part of the proof gives
\[
\mathscr Du_*
=
\mathscr D\!\left(u_*-\frac\pi2\right)
=o_I(1).
\]
Division by \(\eps^{\alpha_D}\) proves the claimed \(C^1\) replacement.  For the
\(D_\eps\)-derivative, the derivative of the normalizing factor cancels
between the two section values in the usual identity
\[
D_\eps(\eps^{-\alpha_D}f)
=
\eps^{-\alpha_D}(D_\eps f-\alpha_D f),
\]
so it introduces no additional loss.
\end{proof}

Finally,
\[
H'=-|L|^2-D\le-D,
\]
so the \(H=0\) event remains uniformly transverse.  On that section,
\[
H=p\frac{a'}a+p\frac{b'}b+q\frac{c'}c=0
\]
and
\[
\xi_t
=
\frac{c'}c-\frac12\left(\frac{a'}a+\frac{b'}b\right)
=
\frac{D}{2p}\frac{c'}c.
\]
Hence
\[
q_1=0,\quad\xi_t=0
\]
is exactly equivalent to
\[
a=b,\quad c'=0.
\]
The reflection argument therefore closes the metric on
\[
(S^p*S^p)\times S^q
=
S^{2p+1}\times S^q.
\]
For completeness, the proof of
Theorem~\ref{thm:reflection-compactification} is dimension-independent.
At a matching point with \(a=b>0\), \(c'=0\), and \(H=0\),
\[
0
=
p\frac{a'}a+p\frac{b'}b+q\frac{c'}c
\]
gives \(a'+b'=0\).  The Einstein equations are invariant under
\[
(t,a,b,c)\longmapsto(2t_*-t,b,a,c),
\]
so uniqueness at the regular principal point gives the reflected
continuation.  The left singular germ has the analytic parity form
\[
a(t)=tA(t^2),\qquad b(t)=B(t^2),\qquad c(t)=C(t^2);
\]
reflection gives the same parity form at the right singular orbit with
\(a\) and \(b\) interchanged.  Hence the compactified metric is
real analytic in the standard singular-orbit charts for every admissible
\((p,q)\).

\subsection{Uniform phase labels and conclusion}

Define
\[
\eps_k(\theta)
=
\exp\!\left(
-\frac{k\pi+\theta}{\omega_D}
\right).
\]
Then
\[
e^{i\omega_D\log(1/\eps_k(\theta))}
=(-1)^ke^{i\theta}.
\]
The parity-normalized exact matching map is now scaled by
\(\eps_k^{-\alpha_D}\).  Proposition~\ref{prop:uniform-C1-matching},
followed by Lemma~\ref{lem:uniform-event-replacement}, gives \(C^1\)
convergence at the exact first \(H=0\) event to the same two-response
limit map as before, with \(L_0,L_1,Q,X\) interpreted in the dimensions
\((p,p,q)\).  The first component selects the common phase; the second,
restricted to that phase branch, has opposite signs at a sufficiently
small ratio and at the balanced ratio.  The intermediate value theorem
therefore gives one exact reflection zero for every sufficiently large
phase label.

We have proved the following.

\begin{theorem}[Uniform three-factor existence theorem]
\label{thm:uniform-sequence}
Let \(p,q\ge2\) and \(D=2p+q\le 8\).  For every sufficiently large phase
index \(k\), there is a real-analytic cohomogeneity-one Einstein metric
\(g^{(k)}\) on \(S^{2p+1}\times S^q\) of the form
\eqref{eq:uniform-ansatz}, with small parameter \(\eps^{(k)}\) satisfying
\[
\log\eps^{(k)}
=
-\frac{k\pi}{\omega_D}+O(1).
\]
Consequently \(\eps^{(k)}\to0\).  In particular the construction contains
an infinite sequence of such metrics; after passing to a sufficiently
sparse arithmetic subsequence of phase labels, the scales may be chosen
strictly decreasing.

The theorem applies precisely to
\[
(p,q)=(2,2),(2,3),(2,4),(3,2).
\]
\end{theorem}

\begin{proof}
By Proposition~\ref{prop:uniform-response-orientation}, choose
\(0<\rho_-<\kappa\) such that
\[
\Phi(\rho_-)<0<\Phi(\kappa),
\]
and work on the compact interval \(I=[\rho_-,\kappa]\).
Proposition~\ref{prop:uniform-response-nonzero} gives
\(Q\ne0\) on \(I\).  Since
\[
\Im(L_0\overline{L_1})=-\omega_D\ne0,
\]
one has \(L_0\ne0\), so
\(A(\rho)=L_0Q(\rho)\) is nowhere zero.  The same response-orientation
proposition supplies the strict endpoint signs for
\[
\Phi(\rho)=
\Im\!\left(L_0Q(\rho)\overline{L_1X(\rho)}\right).
\]
Proposition~\ref{prop:uniform-C1-matching} and
Lemma~\ref{lem:uniform-event-replacement} give the exact-event \(C^1\)
convergence required by Proposition~\ref{prop:response-crossing}.
Applying that proposition to
\[
A=L_0Q,\qquad B=L_1X
\]
gives an exact reflection zero for every sufficiently large phase index.
The reflection and topology argument above gives the asserted
real-analytic Einstein metric on \(S^{2p+1}\times S^q\).

If the compact phase
interval has length \(|\Theta|\), choose \(N\) with
\(N\pi>|\Theta|\); then the phase intervals associated with
\(k_j=k_0+jN\) are disjoint and ordered, which makes the corresponding
\(\eps^{(k_j)}\) strictly decreasing.
\end{proof}

\begin{lemma}[Intrinsic curvature splitting at the singular orbit]
\label{lem:intrinsic-curvature-splitting}
Fix an admissible pair \((p,q)\) and a compact ratio interval
\(I\Subset(0,\infty)\).  At a point of the left singular orbit write
\[
T_xM=N\oplus B\oplus C,
\qquad
\dim N=p+1,\quad \dim B=p,\quad \dim C=q,
\]
where \(N\) is the \((p+1)\)-dimensional normal space, and
\(B,C\) are the two non-collapsing tangent summands.  Then there are
\(\eps_0,c>0\), uniform for \(\rho\in I\), such that whenever
\(0<\eps<\eps_0\),
\[
\ker\mathcal R_x=B\wedge C,
\]
and
\[
\|\mathcal R_x\omega\|
\ge c\eps^{-2}\|\omega\|
\qquad
\bigl(\omega\perp B\wedge C\bigr).
\]
\end{lemma}

\begin{proof}
For \(t>0\), the metric is a multiply warped product with mutually
orthogonal distributions
\[
\langle\partial_t\rangle,\qquad E_A,\qquad E_B,\qquad E_C.
\]
The standard multiply-warped curvature formulas show directly that the
curvature operator has no off-diagonal terms between the elementary wedge
types
\[
\Lambda^2E_A,\ \Lambda^2E_B,\ \Lambda^2E_C,\ 
E_A\wedge E_B,\ E_A\wedge E_C,\ E_B\wedge E_C,
\]
and that the radial wedges \(\partial_t\wedge E_i\) are diagonal as well.
Passing smoothly to the left singular orbit combines
\(\langle\partial_t\rangle\oplus E_A\) into the normal space \(N\).
Thus, independently of any possible low-dimensional equivalence of
isotropy representations, the curvature operator preserves
\[
\Lambda^2N,\quad
\Lambda^2B,\quad
\Lambda^2C,\quad
N\wedge B,\quad
N\wedge C,\quad
B\wedge C.
\]
The smooth singular-orbit expansions
\[
a(t)=t+O(t^3),\qquad
b(t)=\eps+b_2t^2+O(t^4),\qquad
c(t)=\eps\rho+c_2t^2+O(t^4)
\]
give \(b'(0)=c'(0)=0\).  The mixed warped-product curvature formula
therefore gives exactly
\[
\mathcal R_x|_{B\wedge C}=0.
\]
The first Taylor coefficients from
Proposition~\ref{prop:uniform-positive-germs} are
\[
2b_2
=
\frac{p-1-D\eps^2}{(p+1)\eps},
\qquad
2c_2
=
\frac{q-1-D\rho^2\eps^2}{(p+1)\rho\eps}.
\]
Hence the standard warped sectional-curvature formulas give
\[
\lambda_{BB}
=
\left.\frac{1-(b')^2}{b^2}\right|_{t=0}
=\eps^{-2},
\qquad
\lambda_{CC}
=
\left.\frac{1-(c')^2}{c^2}\right|_{t=0}
=\rho^{-2}\eps^{-2},
\]
and
\begin{align*}
\lambda_{NB}
&=
-\left.\frac{b''}{b}\right|_{t=0}
=
-\frac{p-1}{p+1}\eps^{-2}
+\frac{D}{p+1},\\
\lambda_{NC}
&=
-\left.\frac{c''}{c}\right|_{t=0}
=
-\frac{q-1}{(p+1)\rho^2}\eps^{-2}
+\frac{D}{p+1}.
\end{align*}
The same values arise as the limits of the collapsing-angular mixed
curvatures, so these are the eigenvalues on all of \(N\wedge B\) and
\(N\wedge C\).

Write
\[
a(t)=t+a_3t^3+O(t^5).
\]
The radial Einstein equation at \(t=0\) gives
\[
-6pa_3
-p\frac{2b_2}{\eps}
-q\frac{2c_2}{\eps\rho}
=D.
\]
Therefore
\[
-6a_3
=
\frac{p(p-1)+q(q-1)\rho^{-2}}{p(p+1)}
\eps^{-2}
+O_I(1).
\]
Smoothness of the normal disk identifies this with the curvature
eigenvalue on \(\Lambda^2N\).  Thus the five nonzero blocks have the
uniform expansions
\begin{align*}
\lambda_{BB}&=\eps^{-2},\\
\lambda_{CC}&=\rho^{-2}\eps^{-2},\\
\lambda_{NB}
&=-\frac{p-1}{p+1}\eps^{-2}+O_I(1),\\
\lambda_{NC}
&=-\frac{q-1}{(p+1)\rho^2}\eps^{-2}+O_I(1),\\
\lambda_{NN}
&=
\frac{p(p-1)+q(q-1)\rho^{-2}}{p(p+1)}
\eps^{-2}+O_I(1).
\end{align*}
Every displayed leading coefficient is bounded away from zero on
\(I\), since \(p,q\ge2\).  Hence, after decreasing \(\eps_0\), all five
blocks are uniformly invertible with inverse norm \(O_I(\eps^2)\).
This proves both assertions.
\end{proof}

\begin{lemma}[Algebraic recovery from the mixed kernel]
\label{lem:mixed-kernel-recovers-factors}
Let \(T\) be an inner-product space and suppose
\[
T=N\oplus B\oplus C,
\qquad
\dim B,\dim C\ge2,
\]
orthogonally.  If the subspace
\[
K=B\wedge C\subset\Lambda^2T
\]
is given intrinsically, then \(K\) determines \(B\oplus C\) and the
unordered pair \(\{B,C\}\).  Consequently it also determines \(N\) as
their orthogonal complement.  If \(\dim B\ne\dim C\), it determines the
ordered pair \((B,C)\).
\end{lemma}

\begin{proof}
First recover
\[
V:=B\oplus C
\]
as the smallest linear subspace \(S\subset T\) for which
\(K\subset\Lambda^2S\).  Equivalently, \(V\) is the linear span of all
vectors occurring as factors of decomposable bivectors belonging to
\(K\).  Since \(K=B\wedge C\) and both factors are nonzero, this support
is exactly \(B\oplus C\).

Now, for \(v\in V\), set
\[
d_K(v)=\dim\bigl(K\cap(v\wedge V)\bigr).
\]
If \(0\ne v\in B\), then
\[
K\cap(v\wedge V)=v\wedge C,
\qquad d_K(v)=\dim C.
\]
If \(0\ne v\in C\), then \(d_K(v)=\dim B\).  If
\(v=b+c\) with \(b\ne0\) and \(c\ne0\), write
\(w=b'+c'\).  The condition
\[
v\wedge w\in B\wedge C
\]
forces
\[
b\wedge b'=0,\qquad c\wedge c'=0,
\]
so \(b'\) is proportional to \(b\) and \(c'\) to \(c\).  Modulo the
kernel direction \(w\parallel v\), the resulting bivectors span precisely
the one-dimensional space generated by \(b\wedge c\).  Hence
\[
d_K(v)=1.
\]
Because \(\dim B,\dim C\ge2\), the nonzero locus
\[
\{v\in V:d_K(v)>1\}
\]
is exactly \((B\setminus\{0\})\cup(C\setminus\{0\})\).  Its two maximal
linear components are \(B\) and \(C\).  Their dimensions distinguish the
order when unequal, and \(N=V^\perp\).
\end{proof}

\begin{proposition}[Uniform intrinsic distinctness]
\label{prop:uniform-distinctness}
For every pair in Theorem~\ref{thm:uniform-sequence}, the constructed
sequence contains infinitely many isometry classes.
\end{proposition}

\begin{proof}
At the left singular orbit the non-collapsing \(S^p\)-factor has radius
\(\eps\).  Its tangent two-planes therefore have sectional curvature
\[
K_{BB}=\eps^{-2},
\]
and hence
\[
\sup_M|\Rm(g^{(k)})|\ge(\eps^{(k)})^{-2}\longrightarrow\infty.
\]
Since \(\Ric=Dg\) fixes the homothety freedom, \(\sup|\Rm|\) is an
isometry invariant and a subsequence is pairwise non-isometric.

More precisely, Lemma~\ref{lem:intrinsic-curvature-splitting} gives,
uniformly on the compact ratio interval used in the construction,
\[
\ker\mathcal R_x=B\wedge C
\]
for all sufficiently small \(\eps\).  By
Lemma~\ref{lem:mixed-kernel-recovers-factors}, this intrinsically recovers
\(B\) and \(C\) when \(p\ne q\), and their unordered pair when \(p=q\).
The possible interchange in the latter case is discrete and therefore
does not affect the identity-component dimension estimate.
\end{proof}

\begin{proposition}[Uniform non-productness and B\"ohm inequivalence]
\label{prop:uniform-comparison}
For every admissible pair \((p,q)\), all sufficiently small-scale members
of the sequence are not locally nontrivial Riemannian products, even up
to homothety, and are not homothetic to a classical two-summand B\"ohm
metric on \(S^{2p+1}\times S^q\).
\end{proposition}

\begin{proof}
Let \(\mathcal R\) denote the algebraic curvature operator.
Lemma~\ref{lem:intrinsic-curvature-splitting} gives, at the left singular
orbit and for sufficiently small \(\eps\),
\[
\ker\mathcal R=B\wedge C,\qquad \dim\ker\mathcal R=pq.
\]
Suppose locally the metric were a nontrivial Riemannian product of
dimensions \(r\) and \(n-r\), where
\[
n=\dim M=D+1.
\]
Every mixed bivector would lie in the curvature-operator kernel, so
\[
\dim\ker\mathcal R\ge r(n-r).
\]
Because the metric is Einstein with positive Einstein constant \(D\),
each local product factor would itself have positive Ricci curvature with
the same Einstein constant.  In particular neither factor can be
one-dimensional.  Hence
\[
2\le r\le n-2
\]
and therefore
\[
r(n-r)\ge2(n-2)=2(D-1).
\]
For the four admissible pairs,
\[
(pq,\,2(D-1))
=
(4,10),\ (6,12),\ (8,14),\ (6,14),
\]
respectively.  Thus \(pq<2(D-1)\), a contradiction.  This excludes every
nontrivial local, and hence global, product structure.  Homothetic
rescaling does not change the curvature-operator kernel dimension.

By Lemma~\ref{lem:mixed-kernel-recovers-factors}, the intrinsic kernel
\(B\wedge C\) recovers the two non-collapsing summands \(B,C\), up to a
discrete interchange when \(p=q\); its orthogonal complement then recovers
\(N\).  Any isometry fixing \(x\) preserves \(\mathcal R_x\), hence
preserves this recovered splitting up to that discrete interchange.
Moreover, the derivative representation of the isotropy group at \(x\)
is faithful: an isometry fixing \(x\) with identity derivative is the
identity.  Therefore the identity component of the isotropy embeds in
\[
O(p+1)\times O(p)\times O(q).
\]
Since every isometry orbit has dimension at most
\(\dim M=2p+q+1\), we obtain
\[
\dim\operatorname{Isom}(g)
\le
(2p+q+1)+p^2+\frac{q(q-1)}2.
\]
A classical two-summand B\"ohm metric on the same manifold has the
standard \(SO(2p+1)\times SO(q+1)\) symmetry, of dimension
\[
p(2p+1)+\frac{q(q+1)}2.
\]
The latter exceeds the preceding upper bound by
\[
p^2-p-1>0\qquad(p\ge2).
\]
Thus the metrics cannot be isometric.  After normalizing the Einstein
constant, the same argument excludes homothety.
\end{proof}

\section{An abstract response-crossing principle}
\label{sec:abstract-response-crossing}

The matching construction admits an abstract formulation in terms of two
oscillatory response channels.  The first channel determines the common
phase whenever its complex amplitude is nonzero.  Restriction to this phase
branch reduces the second matching condition to a scalar function of the
remaining parameter.  Its limiting value is, up to the positive factor
\(|A|^{-1}\), the phase-invariant quantity
\[
\Im(A\overline B).
\]
Consequently, an endpoint sign change of this quantity yields an exact
crossing for every sufficiently accurate \(C^1\) approximation of the
limiting response map.

The criterion below does not require the scalar crossing to be simple.  In
particular, no condition \(\Phi'(\rho_*)\ne0\) is imposed.  The
nondegeneracy needed to eliminate the phase variable follows instead from
\(A\ne0\), which gives a uniform nonzero derivative in the phase direction.

\begin{proposition}[Response-crossing criterion]
\label{prop:response-crossing}
Let \(I=[\rho_-,\rho_+]\) be a compact interval, and let
\(A,B:I\to\mathbb C\) be \(C^1\) functions with
\[
A(\rho)\neq0
\qquad (\rho\in I).
\]
Define
\[
\Phi(\rho)
=
\Im\!\left(A(\rho)\overline{B(\rho)}\right),
\]
and suppose
\[
\Phi(\rho_-)<0<\Phi(\rho_+).
\]
Let \(\Theta\subset\mathbb R\) be a compact interval and suppose that
\[
\mathcal F_k:\Theta\times I\longrightarrow\mathbb R^2
\]
are \(C^1\) maps satisfying
\[
\mathcal F_k\longrightarrow
G(\theta,\rho)
=
\begin{pmatrix}
\Re(e^{i\theta}A(\rho))\\
\Re(e^{i\theta}B(\rho))
\end{pmatrix}
\]
in \(C^1(\Theta\times I)\).  Assume that \(\Theta\) contains a continuous
phase branch \(\theta_0:I\to\operatorname{int}\Theta\) satisfying
\[
e^{i\theta_0(\rho)}A(\rho)=i|A(\rho)|.
\]
Then, for every sufficiently large \(k\), the map \(\mathcal F_k\) has a
zero in \(\Theta\times(\rho_-,\rho_+)\).
\end{proposition}

\begin{proof}
Write \(\mathcal F_k=(F_{k,1},F_{k,2})\) and \(G=(G_1,G_2)\).  Since
\(A\) is continuous and nowhere zero on the compact interval \(I\),
\[
m:=\min_{\rho\in I}|A(\rho)|>0.
\]
Since \(\theta_0(I)\) is compact and contained in
\(\operatorname{int}\Theta\), there is \(\delta\in(0,\pi/2)\) such that
\[
[\theta_0(\rho)-\delta,\theta_0(\rho)+\delta]\subset\Theta
\qquad(\rho\in I).
\]
For \(|s|\le\delta\), the normalization of the phase branch gives
\[
G_1(\theta_0(\rho)+s,\rho)
=
\Re\!\left(e^{is}i|A(\rho)|\right)
=-|A(\rho)|\sin s
\]
and hence
\[
\partial_\theta G_1(\theta_0(\rho)+s,\rho)
=-|A(\rho)|\cos s
\le -m\cos\delta<0.
\]
Moreover,
\[
G_1(\theta_0(\rho)-\delta,\rho)>0,
\qquad
G_1(\theta_0(\rho)+\delta,\rho)<0,
\]
uniformly in \(\rho\).

The \(C^1\)-convergence therefore implies that, for all sufficiently large
\(k\), the same two boundary signs hold for \(F_{k,1}\), while
\[
\partial_\theta F_{k,1}
\le -\frac12m\cos\delta<0
\]
throughout this tube.  Thus for each \(\rho\in I\) there is a unique
\[
\theta_k(\rho)\in
(\theta_0(\rho)-\delta,\theta_0(\rho)+\delta)
\]
with
\[
F_{k,1}(\theta_k(\rho),\rho)=0.
\]
By the implicit-function theorem and uniqueness, \(\theta_k\) is \(C^1\).
Furthermore, the mean-value theorem gives
\[
\|\theta_k-\theta_0\|_{C^0(I)}
\le
\frac{2}{m\cos\delta}
\|F_{k,1}-G_1\|_{C^0(\Theta\times I)},
\]
so \(\theta_k\to\theta_0\) uniformly.

Restrict the second matching equation to this exact phase branch and set
\[
h_k(\rho):=F_{k,2}(\theta_k(\rho),\rho).
\]
Then \(h_k\) is continuous and, uniformly on \(I\),
\[
h_k(\rho)\longrightarrow G_2(\theta_0(\rho),\rho).
\]
Since
\[
e^{i\theta_0(\rho)}
=\frac{i\overline{A(\rho)}}{|A(\rho)|},
\]
we obtain
\[
G_2(\theta_0(\rho),\rho)
=
\Re\!\left(
\frac{i\overline{A(\rho)}B(\rho)}{|A(\rho)|}
\right)
=
\frac{\Im(A(\rho)\overline{B(\rho)})}{|A(\rho)|}
=
\frac{\Phi(\rho)}{|A(\rho)|}.
\]
Consequently
\[
h_k\longrightarrow \frac{\Phi}{|A|}
\qquad\text{uniformly on }I.
\]
The strict endpoint signs of \(\Phi\) imply, for every sufficiently large
\(k\),
\[
h_k(\rho_-)<0<h_k(\rho_+).
\]
The intermediate value theorem yields
\(\rho_k\in(\rho_-,\rho_+)\) with \(h_k(\rho_k)=0\).  By the definition of
\(\theta_k\),
\[
F_{k,1}(\theta_k(\rho_k),\rho_k)=0
\]
as well.  Hence
\[
\mathcal F_k(\theta_k(\rho_k),\rho_k)=0,
\]
which proves the assertion.
\end{proof}

\begin{corollary}[Stability of the response crossing]
\label{cor:response-crossing-stability}
Under the hypotheses of Proposition~\ref{prop:response-crossing}, the
crossing persists under sufficiently small \(C^1\) perturbations of the
limiting amplitudes and of the matching maps.  More precisely, there is a
\(C^1\)-neighbourhood \(\mathcal U\) of \((A,B)\) such that, for every
\((\widetilde A,\widetilde B)\in\mathcal U\), the phase branch may be chosen
inside \(\operatorname{int}\Theta\),
\[
\widetilde\Phi(\rho)
:=\Im\!\left(\widetilde A(\rho)
\overline{\widetilde B(\rho)}\right)
\]
has the same strict endpoint signs as \(\Phi\), and every sequence of
\(C^1\) maps \(\widetilde{\mathcal F}_k\) satisfying
\[
\widetilde{\mathcal F}_k\longrightarrow
\widetilde G(\theta,\rho)
=
\begin{pmatrix}
\Re(e^{i\theta}\widetilde A(\rho))\\
\Re(e^{i\theta}\widetilde B(\rho))
\end{pmatrix}
\]
in \(C^1(\Theta\times I)\) has a zero in
\(\Theta\times(\rho_-,\rho_+)\) for all sufficiently large \(k\).
\end{corollary}

\begin{proof}
Set
\[
m=\min_I|A|>0,
\qquad
\eta=\min\{-\Phi(\rho_-),\Phi(\rho_+)\}>0.
\]
The conditions \(\min_I|A|>0\) and the two strict endpoint inequalities are
open in the \(C^0\) topology.  Hence, after shrinking a \(C^1\)-neighbourhood
of \((A,B)\), one has
\[
\min_I|\widetilde A|>\frac m2,
\qquad
\widetilde\Phi(\rho_-)<-\frac\eta2,
\qquad
\widetilde\Phi(\rho_+)>\frac\eta2.
\]
Because \(I\) is an interval and \(\widetilde A\) is nowhere zero, its
argument has a continuous lift.  Choosing the lift close to the one defined
by \(\theta_0\) gives a continuous phase branch
\(\widetilde\theta_0\) satisfying
\[
e^{i\widetilde\theta_0(\rho)}\widetilde A(\rho)
=i|\widetilde A(\rho)|.
\]
The compact set \(\theta_0(I)\) has positive distance from
\(\partial\Theta\); therefore, for a sufficiently small perturbation,
\(\widetilde\theta_0(I)\subset\operatorname{int}\Theta\).
Proposition~\ref{prop:response-crossing}, applied to
\((\widetilde A,\widetilde B,\widetilde{\mathcal F}_k)\), now gives the
claim.
\end{proof}

\begin{remark}[Relation with the Einstein matching argument]
\label{rem:response-crossing-application}
In the setting of Theorem~\ref{thm:sequence}, set
\[
A(\rho)=L_0Q(\rho),
\qquad
B(\rho)=L_1X(\rho).
\]
Then the associated phase-invariant response is
\[
\Im\!\left(L_0Q(\rho)\,\overline{L_1X(\rho)}\right)=\Phi(\rho).
\]
The weighted \(C^1\) matching theorem gives the required convergence to the
limiting oscillatory response map, and the nonvanishing of \(Q\) gives the
uniform phase nondegeneracy.  The endpoint sign conditions for \(\Phi\)
therefore imply an exact Einstein match for each sufficiently large phase
label by Proposition~\ref{prop:response-crossing}.
\end{remark}

\bigskip

\noindent
\textsc{Anna Siffert}\\
Universit\"at M\"unster, Mathematisches Institut,
Einsteinstr.\ 62, 48149 M\"unster, Germany\\
\texttt{asiffert@uni-muenster.de}

\end{document}